\documentclass{article}

\usepackage{amsmath,amssymb,mathtools,amsthm}
\usepackage{xcolor,graphicx,float,enumitem, marginnote}
\usepackage{stackengine}
\usepackage{comment}

\usepackage{bm}

\usepackage{tikz-cd}
\usetikzlibrary{cd}

\usepackage{tikz}
\usetikzlibrary{matrix}
\tikzset
{
  over/.style={preaction={draw=white,-,line width=6pt}},
}

\graphicspath{{figures/}}

\usepackage[default,scale=1]{opensans}
\usepackage[T1]{fontenc}

\usepackage[utf8]{inputenc}
\DeclareUnicodeCharacter{00A0}{ }
\usepackage{geometry}
\newcommand{\ee}{\varepsilon}

\usepackage{bm}

\newcommand{\R}{{\mathbb R}}
\newcommand{\Rd}{{\R^d}}

\newcommand{\Edisc}{E^{\Delta}}
\newcommand{\Ediscee}{E^{\Delta}_{\ee}}
\newcommand{\Edisceeinfty}{E^{\Delta, \infty}_{\ee}}
\newcommand{\inversedU}{T_{0,\alpha} \circ (U')^{-1}}
\newcommand{\inversedUee}{(U_{\ee}')^{-1}}
\newcommand{\degree}{{\mathrm{deg}}}

\DeclareMathOperator{\diver}{div}

\DeclareMathOperator{\sign}{sign}
\DeclareMathOperator{\supp}{supp}

\newcommand*\diff{\mathop{}\!\mathrm{d}}

\usepackage[hidelinks, unicode]{hyperref}
\usepackage[nameinlink]{cleveref}

\crefname{subsection}{subsection}{subsections}
\Crefname{subsection}{Subsection}{Subsections}

\newtheorem{theorem}{Theorem}[section]
\newtheorem{proposition}[theorem]{Proposition}%[section]
\newtheorem{lemma}[theorem]{Lemma}%[section]

\theoremstyle{definition}
\newtheorem{definition}[theorem]{Definition}%[section]
\newtheorem{remark}[theorem]{Remark}%[section]
\newtheorem*{remark*}{Remark}%[section]
\newtheorem*{remarks*}{Remarks}%[section]

\numberwithin{equation}{section}

\newlist{enumeratetheorem}{enumerate}{2}
\setlist[enumeratetheorem,1]{
	label=\rm \roman*), 
	ref=Theorem \thetheorem--Item \roman*,
	topsep=0ex
	}
\setlist[enumeratetheorem,2]{label=\rm \alph*), ref=\theenumeratetheoremi.\alph*}

\newlist{enumeratedefinition}{enumerate}{2}
\setlist[enumeratedefinition,1]{
	label=\rm \roman*), 
	ref=Definition \thedefinition--Item \roman*,
	topsep=0ex}
\setlist[enumeratedefinition,2]{label=\rm \alph*), ref=\theenumeratedefinitioni.\alph*}

\newlist{enumeratesteps}{enumerate}{2}
\setlist[enumeratesteps,1]{
  label=\textit{Step \arabic*\protect\thisstep},
  ref=\arabic*,
  align=left, leftmargin=0pt, labelindent=!,
  listparindent=\parindent, labelwidth=0pt, itemindent=!
}
\setlist[enumeratesteps,2]{
    label= \textit{Step \theenumeratestepsi.\alph*\protect\thisstep},
    ref=\theenumeratestepsi.\alph*,
    align=left, leftmargin=0pt, labelindent=\parindent,
  listparindent=\parindent, labelwidth=0pt, itemindent=!
}
\newcommand{\step}[1][]{%
  \if\relax\detokenize{#1}\relax
    \def\thisstep{.}%
  \else
    \def\thisstep{: \protect #1.}%
  \fi
  \item}

\Crefname{enumeratestepsi}{Step}{Steps}
\Crefname{enumeratestepsii}{Step}{Steps}

\newcommand{\dx}{\, dx}

\newcommand{\ds}{\, ds}

\usepackage[
url=false,
isbn=false,
maxcitenames=4,
giveninits,
maxbibnames=100,
doi=false,
sorting=nyt,
uniquename=false,
block=none]{biblatex}
\renewbibmacro{in:}{}
\DeclareFieldFormat[article]{citetitle}{#1}
\DeclareFieldFormat[article]{title}{#1}

\newcommand{\mob}{\mathrm{m}}
\newcommand{\mobee}{\mob_{\ee}}
\newcommand{\mobnd}{\mob^{(1)}}
\newcommand{\mobni}{\mob^{(2)}}
\newcommand{\mobeend}{\mob_{\ee}^{(1)}}
\newcommand{\mobeeni}{\mob_{\ee}^{(2)}}

\newcommand{\PhiPrimpp}{\Psi_{1}}
\newcommand{\PhiPrimnp}{\Psi_{2}}
\newcommand{\InvPhiPrimpp}{\varpi_{1}}
\newcommand{\InvPhiPrimnp}{\varpi_{2}}

\newcommand{\limitpp}{\varkappa_{1}}
\newcommand{\limitnp}{\varkappa_{2}}

\newcommand{\goodSemigroup}{free-energy dissipating semigroup}
\newcommand{\timeLimit}{time-limit operator}
\newcommand{\goodNumericalProblem}{free-energy dissipating numerical scheme}
\newcommand{\goodNumericalProblems}{free-energy dissipating numerical schemes}

\newcommand{\oneconstant}{\lambda}

\newcommand{\numericApproxStep}{J_\ee^\Delta}
\newcommand{\numericStep}{J^\Delta}

\makeatletter
\renewcommand*{\@fnsymbol}[1]{\ensuremath{\ifcase#1\or \star \or \dagger\or \ddagger\or
		\mathsection\or \mathparagraph\or \|\or **\or \dagger\dagger
		\or \ddagger\ddagger \else\@ctrerr\fi}}
\makeatother

\usepackage{todonotes}
\DeclareUnicodeCharacter{2212}{-}

\title{
Drift-diffusion equations with saturation
}
\author{
    Jos\'e A. Carrillo%
     \thanks{Mathematical Institute, University of Oxford, Oxford OX2 6GG, UK.  \href{mailto:carrillo@maths.ox.ac.uk}{carrillo@maths.ox.ac.uk}} %
    \and 
    Alejandro Fernández-Jiménez
     \thanks{Mathematical Institute, University of Oxford, Oxford OX2 6GG, UK.  \href{mailto:alejandro.fernandezjimenez@maths.ox.ac.uk}{alejandro.fernandezjimenez@maths.ox.ac.uk}} %
    \and 
    David Gómez-Castro%
     \thanks{Dpto.~de Matemáticas, Universidad Autónoma de Madrid \& Instituto de Ciencias Matem\'aticas, Madrid 28049, Spain.  \href{mailto:david.gomezcastro@uam.es}{david.gomezcastro@uam.es}}
}

\begin{document}

\maketitle

\begin{abstract}
    We focus on a family of nonlinear continuity equations for the evolution of a non-negative density $\rho$ with a continuous and compactly supported nonlinear mobility $\mathrm{m}(\rho)$ not necessarily concave. The velocity field is the negative gradient of the variation of a free energy including internal and confinement energy terms. Problems with compactly supported mobility are often called saturation problems since the values of the density are constrained below a maximal value. Taking advantage of a family of approximating problems, we show the existence of $C_0$-semigroups of $L^1$ contractions. We study the $\omega$-limit of the problem, its most relevant properties, and the appearance of free boundaries in the long-time behaviour. This problem has a formal gradient-flow structure, and we discuss the local/global minimisers of the corresponding free energy in the natural topology related to the set of initial data for the $L^\infty$-constrained gradient flow of probability densities. Furthermore, we analyse a structure preserving implicit finite-volume scheme and discuss its convergence and long-time behaviour.

    \smallskip

    \noindent\textbf{Keywords:}
    Saturation, nonlinear parabolic equations, long-time behaviour, $C_0$-semigroup, free boundary, Euler-Lagrange condition, implicit finite-volume scheme.

    \smallskip

    \noindent \textbf{MSC:}
    35K55,  	%Nonlinear parabolic equations
    35K65,  	%Degenerate parabolic equations
    35B40,  	%Asymptotic behavior of solutions to PDEs
    65M08,  	%Finite volume methods for initial value and initial-boundary value problems involving PDEs
    35Q70,  	%PDEs in connection with mechanics of particles and systems of particles
    35Q92,  	%PDEs in connection with biology, chemistry and other natural sciences
    47H20.  	%Semigroups of nonlinear operators
\end{abstract}

\setcounter{tocdepth}{2}
%\tableofcontents

\section{Introduction}

Aggregation and drift-diffusion equations are frequent in continuous descriptions of density populations since they are natural macroscopic models associated to microscopic particle dynamics, see for instance \cite{Carrillo_Murakawa_Sato_Togashi_Trush19} and the references therein. Some models contain a more general nonlinear mobility, usually called of saturation type, preventing overcrowding. This family of partial differential equations include models of the form
\begin{subequations}
    \label{eq:The problem}
    \begin{equation}
        \frac{\partial \rho}{\partial t} = \diver \left( \mob (\rho) \nabla \left( U' (\rho) + V \right) \right).
    \end{equation}
    Here, we consider $U$ convex and $V$ a given potential regular enough. Furthermore, we work on a bounded domain $\Omega$, where we set the natural no-flux condition
    \begin{equation}
        \mob (\rho)  \nabla (U'(\rho) + V ) \cdot \nu (x) = 0 \qquad \text{for all } t > 0, x \in \partial \Omega\,,
    \end{equation}
\end{subequations}
leading to conservation of the total mass. The case of linear mobility $\mob (\rho) = \rho$ is well understood, see \cite{CCY19, Bailo_Carrillo_GomezCastro24} and the references therein.
For this family of nonlinear parabolic equations \eqref{eq:The problem}, the well-posedness theory, their long-time behaviour, and the main qualitative properties of the solutions, self-similar solutions and their steady states have been fairly well-analysed \cite{ Carrillo_Jungel_Markowich_Toscani_Unterreiter01, Vazquez07, Carrillo_Delgadino_Dolbeault_Frank_Hoffmann19, Carrillo_Hittmeir_Volzone_Yao19, Carrillo_Gomez-Castro_Vazquez22, Carrillo_FJ_Gomez-Castro23}. Moreover, the family of Cauchy problems of the form \eqref{eq:The problem} with linear mobility are $2$-Wasserstein gradient flows \cite{Otto01, Carrillo_McCann_Villani03, Carrillo_McCann_Villani06,Ambrossio_Gigli_Savare08, Santambrogio15,Santambrogio17} of the free-energy functional
\begin{equation}\label{eq:Free eenrgy ee=0 Omega}
    \mathcal{F}[\rho] = \int_{\Omega} U(\rho (x)) \dx + \int_{\Omega} V(x) \rho (x) \dx.
\end{equation}

When $\mob(\rho)$ is a non-linear mobility (not necessarily bounded) there is also an extensive literature.  A suitable notion of generalised Wasserstein distance was introduced in \cite{DolbeaultNazaretSavare2009} by extending the Benamou-Brenier formulation. This approach only produces well-defined distances if $\mob(\rho)$ is concave.
The corresponding Otto calculus
yields that the formal gradient flow of the free energy \eqref{eq:Free-energy} in these non-linear mobility Wasserstein-type distances corresponds to the family of PDEs
\begin{equation*}
    \frac{\partial \rho }{\partial t} = \diver \left( \mob (\rho ) \nabla \frac{\delta \mathcal{F}}{\delta \rho}[\rho] \right).
\end{equation*}
We will take advantage of the free-energy dissipation structure of this formulation.
The result in \cite{DolbeaultNazaretSavare2009} has been extended to cover more cases, including more general families of non-linear mobilities \cite{Lisini_Marigonda10, CarrilloLisiniSavareSlepcev2010, Dolbeault_Nazaret_Savare12,  DiMarino_Portinale_Radici22}.
Several aggregation-diffusion related equations with non-linear mobility have been also analysed by different methods in PDE theory for instance:  Newtonian interaction potentials (i.e., $U=0$) \cite{Carrillo_Gomez-Castro_Vazquez22b, Carrillo_Gomez-Castro_Vazquez22a}, porous medium equations with non-local pressure
\cite{Caffarelli_Vazquez11, Stan_delTeso_Vazquez16, Stan_delTeso_Vazquez19},
Cahn-Hilliard type equations \cite{Bertozzi_Pugh98, Lisini_Matthes_Savare12,Elbar_Skrzeczkowski24}, or interaction systems on graphs \cite{Heinze_Pietschmann_Schmidtchen23}, among others.

Our work focuses on a non-linear mobility $\mob (\rho)$  of \textit{saturation}-type, i.e., the support of the mobility is a finite interval. More precisely, the mobility $\mob (\rho)$ satisfies the following assumptions:
\begin{enumerate}[label=(H$_{\arabic*}$)]\label{hyp:H}
    \item
          \label{hyp:mobility}
          There is some $\alpha \in (0, \infty)$ such that $\mob(0) = \mob(\alpha) = 0$, and $\mob > 0$ in $(0,\alpha)$.
          We assume that $\mob \in C([0,\alpha]) \cap C^1((0,\alpha))$.

    \item
          \label{hyp:admissible data}
          We deal with initial data $
              \rho_0$ in the admissible class of densities $\mathcal A \coloneqq \{ \rho \in L^1(\Omega) : 0 \le \rho \le \alpha \}.$

    \item
          \label{hyp:H_V}
          $V$ is of class $C^2 (\overline \Omega)$.
          Without loss of generality we assume $V \geq 0$.

    \item
          \label{hyp:H_U}
          $U \in W^{1,1} ((0,\alpha)) \cap C^{3} ((0, \alpha ))$ 
          and convex $U'' \geq 0 $.
          We assume that $U$ is not trivial, i.e.,
          \begin{equation}
              \label{eq:s0}
              \begin{aligned}
                   & \text{there exists } s_0 \in (0,\alpha ) \text{ such that } U'' (s_0) >0.
              \end{aligned}
          \end{equation}
\end{enumerate}
We will make a further technical assumption \ref{hyp:Phi}, that is postponed to the next section.
Notice that, unlike in previous literature, we do not assume that $\nabla V \cdot \nu = 0$ on $\partial \Omega$.
Thus, $\rho \equiv 0$ and $\rho \equiv \alpha$ are constants solutions.

These aggregation-diffusion equations with mobility of saturation-type appear naturally in mathematical biology, in order to explain chemotaxis models with prevention of overcrowding \cite{Chalub_Rodrigues06, Hillen_Painter01, Carrillo_Murakawa_Sato_Togashi_Trush19}, in mathematical physics, to describe the relaxation of gas of fermions \cite{Kaniadakis95, Kaniadakis_Quarati93}, in phase segregation \cite{Slepvcev08, Zamponi_Jungel17}, or in thin liquid films \cite{Lisini_Matthes_Savare12, Matthes_McCann_Savare09} among others. Despite the interest in applications of models with mobility of saturation-type, the literature devoted to problem \eqref{eq:The problem} with the saturation-type mobility satisfying assumptions \ref{hyp:mobility}-\ref{hyp:H_U}, up to our knowledge, is scarce. In \cite{Burger_DiFrancesco_Yasmin06, diFrancesco_Rosado08} the authors consider the Keller-Segel model with prevention of overcrowding, which they obtain by choosing $\mob (\rho) = \rho (1-\rho)$. They study the competition between the chemotaxis term with a  saturation effect and a linear diffusion term. More recently, in \cite{DiFrancesco_Fagioli_Radici19}, the authors obtain a rigorous limit from discrete distributions to a family of one-dimensional non-local interaction equations with saturation. Furthermore, this result is extended to a family of one-dimensional aggregation–diffusion equation in \cite{Fagioli_Radici18}. In both cases, the authors only cover the case $\mob(\rho)=\rho \phi(\rho)$ where $\phi$ is decreasing and positive only in a finite interval. In \cite{Fagioli_Tse22}, the authors prove well-posedness of entropy solutions for a wide class of one-dimensional non-local transport equations with a general saturation-type mobility.

To our knowledge, there is no  literature analysing the family of Cauchy problems \eqref{eq:The problem} with saturation-type mobility in higher dimensions. Furthermore, the only work dealing with the long-time behaviour is from the numerical analysis viewpoint by implicit finite-volume schemes introduced in \cite{BCH23}. They show the existence of certain weak stationary solutions with kinks depending on the mass of the solution and their numerical experiments suggest the appearance of kinks in the long-time behaviour for certain initial data. Therefore, the main goal of this manuscript is to provide a unified theory for the Cauchy problems \eqref{eq:The problem} with a general saturation-type mobility satisfying \ref{hyp:mobility}-\ref{hyp:Phi}, including the existence theory, minimisation of the free energy, and their long-time behaviour. Furthermore, we complement our mathematical analysis results with numerical analysis by showing the convergence of suitable implicit finite volume schemes related to \cite{BCH23} and clarifying their long time behaviour.

\subparagraph{Main analytical results}
Our main analytical results concern the
existence of certain solutions to \eqref{eq:The problem} with a general saturation-type mobility satisfying \ref{hyp:mobility}-\ref{hyp:Phi} via approximation arguments and semigroup theory,
the characterization of $L^1$-local minimisers of the associated free-energy functional in the admissible set of bounded integrable densities $\mathcal{A}$ defined in \ref{hyp:admissible data},
and the long-time behaviour of the constructed solutions in view of its gradient flow structure.
Moreover, we study these aspects with the greatest generality on both the saturation-type non-linear mobility $\mob(\rho)$ and the diffusion potential $U(\rho)$ which, in particular, includes the classical porous medium/fast diffusion non-linearities at zero density.

To tackle the existence of certain solutions to \eqref{eq:The problem}, we proceed by stability arguments within the family of problems of the form \eqref{eq:The problem} with a general saturation-type mobility satisfying \ref{hyp:mobility}-\ref{hyp:Phi}. More precisely, we construct suitable approximating problems of the form \eqref{eq:The problem} which admit classical solution while keeping the assumptions \ref{hyp:mobility}-\ref{hyp:Phi}, see \Cref{thm:Properties (Pee)}.
Passing to the limit in these approximating problems, we are able to construct a $C_0$-semigroup, denoted by $\{S_t\}_{t\geq 0}$, of weak solutions defined for any initial datum $\rho_0 \in \mathcal A$, see \Cref{thm:left side of D}. This semigroup $\{S_t\}_{t\geq 0}$, referred as \goodSemigroup{} in the sequel, enjoys mass conservation, comparison principle, $L^1$-contraction, and free-energy dissipation, see \Cref{def:Free-energy dissipating semigroup}. This notion of semigroup allows us next to study the long-time behaviour, leading to the first global-in-time existence result for this family of equations in higher dimensions allowing for free boundaries both at zero density and saturated density value $\alpha$.

The second goal of our analysis is to study the minimisation of the free-energy functional $\mathcal F$ in the class of admissible densities $\mathcal A$.
We obtain the Euler-Lagrange conditions for the $L^1$-local minimisation, see \Cref{thm:Euler-Lagrange}.
When $U$ is strictly convex, we show that the unique local minimiser for a fixed mass is explicit,
\[
    \widehat \rho (x) = \min\Bigg\{ \alpha, \Big(  (U')^{-1} (C_0 - V(x))     \Big)_+   \Bigg\}
\]
where the constant $C_0$ comes from the mass constraint.
Notice that this is a truncation by $\alpha$ of the usual family of minimisers for the linear-mobility case.
The Euler-Lagrange conditions are already well-understood for the linear mobility case, but this seems to be new in the literature for minimisation in the set $\mathcal A$
although related to constrained minimisation problems as in \cite{CT20}.

We next focus on the long-time behaviour of the constructed solutions  showing that there exists a \timeLimit{} $S_\infty: \mathcal{A} \rightarrow \mathcal{A}$, see \Cref{def:timeLimit}, such that for any $\rho_0 \in \mathcal A$, we have asymptotic time convergence of the constructed semigroup $\{S_t\}_{t\geq 0}$, that is
\[
    S_t \rho_0 \to S_\infty \rho_0, \qquad \text{in } L^1 (\Omega) \text{ as } t \to \infty.
\]
We show that $S_\infty$ is still an $L^1$-contraction.
Hence, the $\omega$-limit set $\{S_\infty \rho_0 : \rho_0 \in \mathcal A\}$ is an $L^1$-continuous subset of $\mathcal A$, see \Cref{thm:long-time behaviour}.

We further analyse the structure of the $\omega$-limit set. The classical solutions for the approximating problems have a unique element in the $\omega$-limit set, i.e., the global attractor, corresponding to the unique constant-in-time solution and the unique global (and $L^1$-local) minimiser of the free energy, see \Cref{th:ee continuous limit}. Under certain convexity assumption for the nonlinear diffusion, we can characterize fully the $\omega$-limit set again given by the unique constant-in-time solution and the unique global (and $L^1$-local) minimiser of the free energy, see \Cref{thm:Asymptotic (P0)}. On the other hand, we construct examples of degenerate non-linearities where $S_\infty \rho_0$ is not an $L^1$-local minimiser of $\mathcal F$,  but only  saddle points of the free energy, see \Cref{fig:Mass_saddle_point}.

Finally, we are able to justify mathematically the behaviour numerically observed in \cite{BCH23}: the appearance of kinks in the long-time behaviour and the complicated structure of the $\omega$-limit set when non-linearities are degenerate combined with a saturated-type mobility.

\subparagraph{Numerical analysis}
The design of numerical schemes for aggregation  and drift-diffusion equations is a crucial tool to understand the dynamics of this family of equations. In particular, we need to develop methods that keep the structural properties of the gradient flow of densities:
the non-negativity of the solution, the dissipation property, and a corresponding set of stationary states which capture the long-time asymptotics. Finite-volume methods allow us to obtain schemes with these properties. In \cite{Bessemoulin-ChatardFilbet12}, the authors propose first and second-order-accurate finite-volume schemes treating non-linear diffusion equations as a non-linear continuity equation. Another method is proposed in \cite{CarrilloChertockHuang15} for aggregation-diffusion equations. Moreover, a generalisation for high-order approximations is proposed in \cite{SunCarrilloChi-Wang18}. In \cite{AlmeidaBubbaPerthamePouchol19}, the authors propose several fully discrete, implicit-in-time discretizations for the Keller-Segel model in one dimension. This work is generalised in \cite{BailoCarrilloHu20}, where the authors introduce a fully discrete (in both space and time) implicit finite-volume scheme for the aggregation-diffusion equation with linear mobility $\mob (s) = s$.
Furthermore, as it is shown in \cite{BailoCarrilloMurakawaSchmidtchen20}, this method converges under suitable assumptions on the diffusion functions and potentials involved and assumptions on the boundary conditions.
In \cite{BCH23}, the authors extend this scheme to cover non-linear mobilities of saturation-type.

Here, we focus on a variation of the implicit finite-volume scheme introduced and analysed in \cite{BCH23}, where we study the case $\mob(s) = \mobnd(s) \mobni(s)$
where $\mobnd$ is non-decreasing and $\mobni$ is non-increasing. Our main results show that the proposed implicit finite volume scheme is well-defined, convergent and structure preserving together with a characterisation of the long-time behaviour of the fully discrete scheme. Moreover, we show that the long-time asymptotics of the numerical scheme capture the long-time behaviour of the constructed solutions to \eqref{eq:The problem}. More precisely, we prove well-posedness, free-energy dissipation, mass conservation, a discrete $L^1$ contraction property, and a comparison principle for the already mentioned method and an approximating version of it, see \Cref{thm:scheme well-posedness}. Furthermore, to keep the analogy with the continuous case and the $C_0$-semigroup theory, we show that our method admits a \goodNumericalProblem{}, see \Cref{def:goodNumericalProblem}. We also show that under high regularity of the solution the scheme converges, see \Cref{thm:Disc to continuous}.

We finally conclude by discussing the long-time behaviour of the numerical scheme, and its rate of convergence to the long-time behaviour of the continuous problem \eqref{eq:The problem}. For the approximating problem, we also prove that the long-time behaviour coincides with the unique constant-in-time solution and the global attractor, see \Cref{thm:discrete approx asymptotics}, analogously to the continuous problem. We also analyse the existence of a time-limit operator for the numerical scheme reproducing the theory studied at the continuous level, see \Cref{thm:Asymptotics discrete P0}. Moreover, we also show examples with complicated long-time asymptotics leading to free boundaries, infinitely many steady states with large basin of attraction, and saturation effects leading to “freezing” behaviour, i.e., free boundaries at the saturation level $\alpha$.

\subparagraph{Open problems}
Showing uniqueness of the constructed free-energy dissipating solutions is an interesting open problem. We may expect uniqueness of enhanced notions of solution as entropy solutions (see, e.g., \cite{Carrillo99,Karlsen_Risebro03}), but we do not deal with this question in this work. We only prove convergence of the numerical scheme as the mesh is refined in the case where the solution to the continuous problem is very regular. It would be interesting to have a proof of the convergence of the numerical scheme that does not use information of the continuous solution or reduces the regularity needed. We do not discuss higher regularity of the solutions. The study of $C^\alpha$ regularity is an interesting open problem, specially in the cases with free boundary or freezing behaviour. The problem with $V$ replaced by the aggregation term $V + W*\rho$ is completely open. Our results can be used to prove existence of a semigroup of solutions. However, there is no  $L^1$ contraction or comparison principle for general $W$ and their long-time behaviour is a difficult problem as numerically investigated in \cite{BCH23}.

\subparagraph{Structure of the paper.} In \Cref{sec:Main results}, we introduce the hypotheses,
the notion of solutions and the numerical scheme and present our main Theorems. In \Cref{sec:Analysis of Pee}, we study the existence of weak solutions, proving \Cref{thm:Properties (Pee),thm:left side of D}. In \Cref{sec:Local minimiser}, we analyse local minimisers of the free energy, proving \Cref{thm:Euler-Lagrange}. In \Cref{sec:Long time behaviour}, we focus on the long-time behaviour for both the regularised and the original problem, i.e. \Cref{thm:long-time behaviour}, and, subsequently, in \Cref{sec:Analysis long time}, we deal with its $\omega$-limit, \Cref{th:ee continuous limit,thm:Asymptotic (P0)}. Finally, we devote \Cref{sec:Numerical Analysis} to the numerical analysis of the implicit finite volume scheme, and, in particular, we prove \Cref{thm:scheme well-posedness,thm:Disc to continuous,thm:discrete approx asymptotics,thm:Asymptotics discrete P0}. The main goals of this work are schematically described in the diagrams \eqref{Diagram} and \eqref{Numeric Diagram}, which we present in \Cref{sec:Main results}.

\section{Main Results}\label{sec:Main results}

The aim of this section is to present our main results and the key ideas of their proof. Our focus is the initial value problem
\begin{equation}
    \tag{P}
    \label{eq:the problem Omega}
    \begin{dcases}
        \frac{\partial \rho}{\partial t} = \diver \left(  \mob (\rho) \nabla \left( U' (\rho) + V \right)  \right)  \qquad & \mathrm{in } \, (0, \infty ) \times \Omega,          \\
        \mob (\rho)  \nabla (U'(\rho) + V ) \cdot \nu (x) = 0                                                              & \mathrm{on } \, (0, \infty ) \times \partial \Omega, \\
        \rho (0, x) = \rho_0 (x)                                                                                           & x \in \Omega ,
    \end{dcases}
\end{equation}
where $\Omega \subseteq \Rd$ is a bounded, connected, and smooth domain. We make an additional technical assumption on the non-linearities.

\begin{enumerate}[label=(H$_{\arabic*}$),start=5]
    \begin{subequations}
        % \makeatletter
        % \def\@currentlabel{H}
        % \makeatother
        % \renewcommand{\theequation}{H$_\alph{equation}$}

        \item
        \label{hyp:Phi}
        We assume that the diffusion is continuous, in the sense that
        \begin{equation}
            \label{hyp:Phi' in L1}
            \mob U''  \in L^1(0,\alpha),
        \end{equation}
        and we define
        \[
            \Phi(s) \coloneqq \int_{s_0}^s \mob(\tau) U''(\tau) \diff \tau.
        \]
        Furthermore, we also assume that $\Phi$ is strictly increasing at $0$ and $\alpha$, i.e.,
        \label{hyp:Properties Phi}
        \begin{equation}
            \label{eq:Assumption Existence}
            \Phi(0) < \Phi (s) < \Phi (\alpha) \quad \text{for all } s \in (0, \alpha).
        \end{equation}
        Lastly, we impose a technical regularity condition which will be suitable for compactness estimates
        \begin{align}
            \label{eq:hypothesis Phi/Phi' ee = 0}
            \sup_{s \in [0,\alpha] }\left|\frac{(\Phi (\alpha)-\Phi (s))\Phi(s)}{\Phi'(s)} \right| + \left|\frac{(\Phi (s)-\Phi (0))\Phi(s)}{\Phi'(s)} \right| & < \infty.
        \end{align}
    \end{subequations}
\end{enumerate}

\noindent In order to study the steady states of \eqref{eq:the problem Omega} we will sometimes assume some of the following strict convexity to different degrees
\begin{align}
    \tag{SC$_U$}
    \label{eq:U locally strictly convex}
                            & U''(s)  >0 , \qquad \text{for a.e. } s \in (0,\alpha) . \\
    \tag{USC$_U$}
    \label{eq:U uniformly strictly convex}
    \inf_{s \in (0,\alpha)} & U''(s)  >0 .
\end{align}
For certain statements on numerical schemes, we will assume
\begin{equation}
    \label{eq: U is C1[0,alpha]}
    U \in C^1([0,\alpha]).
\end{equation}

\begin{remarks*}
    Notice that, since $\mob \in C^1((0,\alpha))$ and $U \in C^3((0,\alpha))$, then $\Phi \in C^2((0,\alpha))$.

    These main hypothesis \hyperref[hyp:H]{(H)}   and \eqref{eq:U locally strictly convex}
    are satisfied for the Porous Medium / Fast Diffusion cases, for example, if $\Phi$ is $C^2((0,\alpha))$ and
    \[
        \Phi'(s) = \begin{dcases}
            a(s) s^{m_1-1}             & \text{if } s \in (0,s_1),       \\
            -b(s) (\alpha - s)^{m_2-1} & \text{if } s \in (s_2, \alpha).
        \end{dcases}
    \]
    For $m_1, m_2 > 0$ and $0 < \underline a \le a(s) \le \overline a, 0 < \underline b \le b(s) \le \overline b$.
    The hypothesis \eqref{eq:U uniformly strictly convex} only holds if $m_1, m_2 < 1$.
    The hypothesis \eqref{eq: U is C1[0,alpha]} only holds if $m_1 , m_2 > 1$. With respect to assumption \ref{hyp:Phi}, if $m_1, m_2 > 0$ then $\Phi$ fulfills \ref{hyp:Phi}. However, if either $m_1$ or $m_2$ is strictly smaller than $-1$, i.e. Ultrafast Diffusion, then the assumption \ref{eq:hypothesis Phi/Phi' ee = 0} does not hold. Furthermore, let us remark that the mobility $\mob$ does not need to be concave in order to fulfill \ref{hyp:Phi}. For example, the case $\mob (s) = s^2 (\alpha - s^2)$ and $U(s) = s^m$ with $m > 0$.
\end{remarks*}

Our first aim is to construct a family of approximating problem \eqref{eq:the problem regularised} with a well-posedness theory in the classical sense. We next use these approximating problems \eqref{eq:the problem regularised} to obtain existence of \eqref{eq:the problem Omega} by compactness arguments. We study the long-time behaviour $t \rightarrow \infty$ for both problems  \eqref{eq:the problem regularised} and \eqref{eq:the problem Omega}. Furthermore, we also discuss whether the limits $t \rightarrow \infty$ and $\ee \rightarrow 0$ commute. The following diagram describes the different questions we analyse in the analytical part of this work.
\begin{equation}\label{Diagram}\tag{D$_1$}
    \hspace{2cm}
    \begin{tikzcd}[color=black, column sep = 2cm]
        \rho_t^{(\ee)} \text{ solution to \eqref{eq:the problem regularised}}
        \ar[r,"t \to \infty"]
        \ar[d,"\varepsilon \to 0",shift right=1.25cm]
        &
        \begin{minipage}{.45\linewidth}
            \begin{flushleft}
                $
                    \widehat \rho^{(\varepsilon)}(x)
                    =(U_\varepsilon')^{-1} \Big( C_\ee   - V (x)         \Big)
                $
            \end{flushleft}
        \end{minipage}
        \ar[d,"\varepsilon \to 0",  shift right=3.25cm]
        \\
        \rho_t^{(0)} \text{ solution to \eqref{eq:the problem Omega}}
        \ar[r,"t \to \infty" below, "?" above, dotted]
        &
        \begin{minipage}{.45\linewidth}
            \begin{flushleft}
                $
                    \widehat{\rho}^{(0)} (x)
                    =T_{0 , \alpha} \circ (U')^{-1} \Big( C_0   - V (x)         \Big).
                $
            \end{flushleft}
        \end{minipage}
    \end{tikzcd}
    \hspace{-2cm}
\end{equation}
Notice that we provide a counterexample for the commutativity of the diagram in \cref{sec:malicious counterexamples}. In fact, the regularisation \eqref{eq:the problem regularised} chooses always the $L^1$-local minimiser as $\ee \rightarrow 0$, see \Cref{th:ee continuous limit,thm:Asymptotic (P0)} (taking first $t\rightarrow \infty$ and then $\ee \rightarrow 0$ in \eqref{Diagram}). Nevertheless there exists some cases for which the diagram is commutative, for instance $\mob(s) = s (\alpha - s)$, $U(s) = s\log(s) $ and $V(x) = |x|^2$. 
A more complete version of this diagram, including the numerical results, is provided at the end of this section in \eqref{Numeric Diagram}.
Here, we use the truncation function defined as
\[
    T_{0,\alpha} (s) = \begin{dcases}
        \alpha & \text{if } s > \alpha ,     \\
        s      & \text{if }s \in [0,\alpha], \\
        0      & \text{if }s < 0.
    \end{dcases}
\]
We use the notation $T_{0,\alpha} \circ (U')^{-1}$ in a generalised sense.
Recall that $U' : (0,\alpha) \to \mathbb R$ is non-decreasing. We define
\begin{align}
    \label{eq:defn xi underline and overline}
    \underline{\zeta} \coloneqq U'(0^+ ) \quad \text{and} \quad    \overline{\zeta} \coloneqq U'(\alpha^-).
\end{align}
Either of these values can be infinite.
With this definition, we define
\begin{equation*}
    T_{0,\alpha} \circ (U')^{-1} (\zeta) \coloneqq \begin{dcases}
        \alpha           & \text{if } \zeta \ge \overline{\zeta},                      \\
        (U')^{-1}(\zeta) & \text{if } \zeta \in (\underline{\zeta}, \overline{\zeta}), \\
        0                & \text{if } \zeta \le \underline{\zeta} .                    \\
    \end{dcases}
\end{equation*}

\begin{figure}[H]
    \centering
    \includegraphics[width=0.75\textwidth]{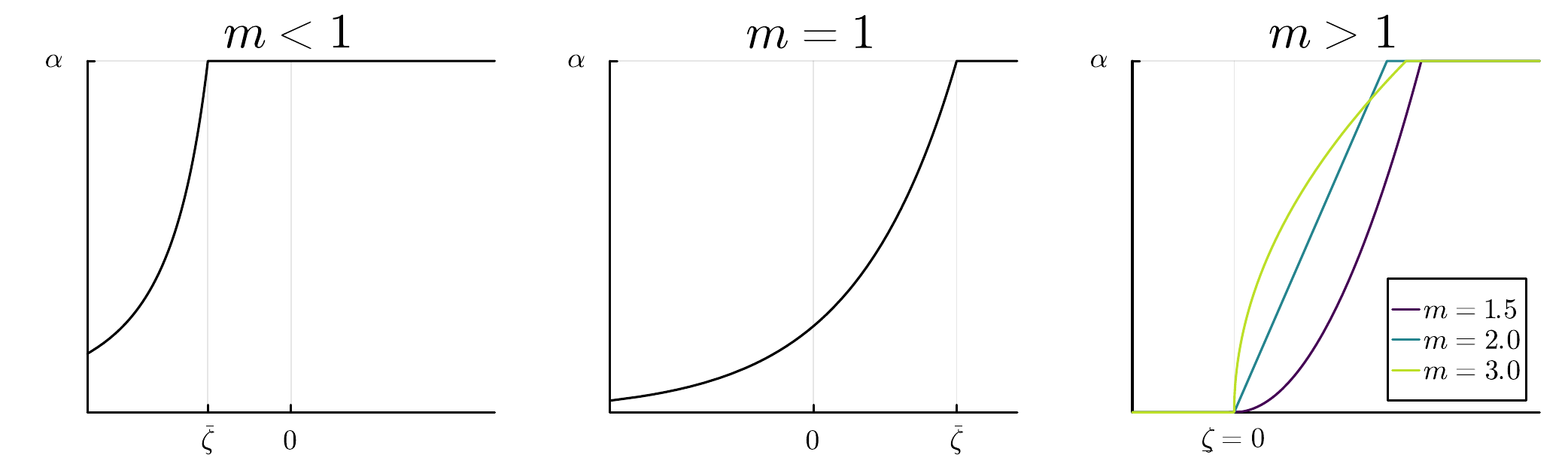}
    \caption{$\inversedU (\zeta)$ for $U(s)= \frac{s^m}{m-1}$ and different choices of the exponent $m$.}
    \label{fig:Figure_for_Theta_m}
\end{figure}

\subsection{Notions of solution}

Throughout this manuscript we use the following notion of weak solution.

\begin{definition}[Weak solution]
    \label{def:Weak solutions P0}
    We say $\rho$ is a weak solution of the problem \eqref{eq:the problem Omega} in $(0,T) \times \Omega$ if $\rho \in L^1((0,T) \times \Omega)$, $\Phi (\rho) \in L^2(0,T; H^1(\Omega))$, and
    \begin{equation*}
        \int_{\Omega} \rho_0 \varphi (0) + \int_0^T \int_{\Omega} \rho_t \frac{\partial \varphi}{\partial t}  = \int_0^T \int_{\Omega} \mob (\rho) \nabla ( U'(\rho) + V) \cdot \nabla \varphi ,
    \end{equation*}
    for all $\varphi \in C^{\infty}([0,T] \times \Omega)$ such that $\varphi(T,\cdot) = 0$. Let us recall that $\nabla \Phi(\rho) = \mob(\rho) \nabla U'(\rho)$.
    Respectively, for $\ee > 0$ and the problem \eqref{eq:the problem regularised}, we consider the analogous version of weak solution.
\end{definition}
Beside $\mathcal A$, we consider the following sets of initial data
\begin{align*}
    \mathcal{A}_+ & \coloneqq \left\{ \rho \in L^1 (\Omega) : \exists \delta > 0 \text{ s.t. } \delta \leq \rho \leq \alpha - \delta \right\}, \\
    \mathcal A_M  & \coloneqq \left\{ \rho \in \mathcal A :  \|\rho\|_{L^1(\Omega)} = M
    \right\} .
\end{align*}
We will work with the notion of semigroup of solutions (possibly non-unique) for \eqref{eq:the problem Omega} as follows.
\begin{definition}
    \label{def:Free-energy dissipating semigroup}
    We say that $S_t : \mathcal A \to \mathcal A$ is a \goodSemigroup{}  of solutions for \eqref{eq:the problem Omega} if
    \begin{enumeratedefinition}
        \item
        \label{item:Semi-group weak solution}
        For $\rho_0 \in \mathcal A$, $\rho_t = S_{t}\rho_0$ is a weak solution to \eqref{eq:the problem Omega}.

        \item
        \label{it:S^eps is C0}
        $S_t$ is a $C_0$-semigroup in $L^1$, i.e., for $t, h > 0$ we have
        \[
            S_{t+h} = S_t S_h, \qquad \lim_{t \to 0^+}\| S_t \rho_0 - \rho_0 \|_{L^1 (\Omega)} = 0 \text{ for all } \rho_0 \in \mathcal A.
        \]

        \item
        \label{it:S^eps is Lp continuous}
        $S_t : \mathcal A \to \mathcal A$ is an $L^1$-contraction,
        i.e., for any $\rho_0, \eta_0 \in \mathcal A$ we have that
        $
            \|S_t \rho_0 - S_t \eta_0 \|_{L^1(\Omega)} \le \|\rho_0 - \eta_0\|_{L^1(\Omega)}.
        $

        \item
        \label{item:W-11 equicontinuity}
        Free-energy dissipation and $C^{\frac 1 2}_{loc}([0,\infty), W^{-1,1}(\Omega))$ continuity:
        If $\rho_0 \in \mathcal A_+$ then calling $\rho_t = S_t \rho_0$ we have
        \begin{align}
            \nonumber
            \text{For all } 0 < t_1 < t_2 \text{ we get } \qquad                           &                                                      \\
            \label{eq:free energy dissipation}
            \int_{t_1}^{t_2} \int_\Omega \mob(\rho_\sigma) |\nabla(U'(\rho_\sigma) + V)|^2 & \le \mathcal F[\rho_{t_1}] - \mathcal F[\rho_{t_2}], \\
            \label{eq:W-11 continuity}
            \| \rho_{t_2} -  \rho_{t_1} \|_{W^{-1,1}(\Omega)}
                                                                                           & \leq
            \| \mob \|_{L^\infty(0,\alpha)}^{\frac{1}{2}}
            | \Omega |^{\frac{1}{2}} \left( \mathcal{F} [\rho_{t_1}] - \mathcal{F} [\rho_{t_2}] \right)^{\frac{1}{2}} |t_2-t_1|^{\frac{1}{2}}.
        \end{align}
        In particular, $t \mapsto \mathcal F[S_t \rho_0]$ is non-increasing.
    \end{enumeratedefinition}
\end{definition}
\begin{remarks*}
    Notice that for a $C_0$-semigroup it follows $t \mapsto S_t \rho_0$ in $C([0,T]; L^1(\Omega))$. Since $\rho \ge 0$, weak solutions are mass preserving by using the test function $\varphi(t,x) = 1$, i.e.,
    $
        \|S_t \rho_0\|_{L^1(\Omega)} = \|\rho_0\|_{L^1 (\Omega)}.
    $
    For $\rho \in \mathcal A$ we have $0 \le \| \rho \|_{L^1} \le \alpha |\Omega|$, mass conservation implies that
    $S_t 0 = 0$ and $S_t \alpha = \alpha$.
    Using that $s_+ = \frac{s+|s|}{2}$, mass conservation and the $L^1$-contraction directly imply the $L^1$ comparison principle
    \[
        \| (S_t \rho_0 - S_t \eta_0)_+ \|_{L^1(\Omega)}
        \le \|(\rho_0 - \eta_0)_+\|_{L^1(\Omega)}, \qquad \text{for all } \rho_0, \eta_0 \in \mathcal A.
    \] 
    Let us finally point out that we recall the definition of the negative Sobolev space and some properties at the beginning of \Cref{sec:Analysis of Pee}.
\end{remarks*}

\subsection{Existence of solutions for \texorpdfstring{\eqref{eq:the problem Omega}}{(P)} by approximation}\label{sec:approximating problem}

Consider $\ee \in (0,1]$.
We will work on approximating problems of the form
\begin{equation}
    \tag{P$_\ee$}
    \label{eq:the problem regularised}
    \begin{dcases}
        \frac{\partial \rho}{\partial t} = \Delta \Phi_{\varepsilon} (\rho ) + \diver \left(  \mobee (\rho ) \nabla V  \right)  \qquad & \mathrm{in} \, (0, \infty ) \times \Omega,          \\
        \left(\mobee (\rho )  \nabla (U_{\varepsilon}'(\rho ) + V ) \right) \cdot \nu (x) = 0                                          & \mathrm{on} \, (0, \infty ) \times \partial \Omega, \\
        \rho (0, x) = \rho_0 (x)                                                                                                       & x \in \Omega.
    \end{dcases}
\end{equation}
Here we regularise the mobility and the non-linear diffusion.
We also make the problem uniformly elliptic by assuming $\Phi_\ee \in C^3([0,\alpha])$ such that
\begin{equation}
    \label{eq:properties of Phi_ee}
    \begin{gathered}
        \underline {\Phi'_\ee} (s) \le \Phi_{\varepsilon}'(s) \le
        (1+\ee)
        \underline {\Phi'_\ee} (s)
        \quad  \text{where} \quad
        \underline {\Phi'_\ee} (s) =  \min (\Phi'(s),
        \kappa(\ee)^{-1}) + \ee ,
        \\
        \underline {\Phi_\ee} (s_0) = \Phi_{\ee} \left(s_0 \right) = 0
        \text{ for $s_0$ given by \eqref{eq:s0}},
        \\
        \Phi_\ee \to \Phi  \text{ in } C^2_{loc} ((0,\alpha)),
    \end{gathered}
\end{equation}
where $\kappa(\ee) \to 0$ as $\ee \to 0$.
Given any family $\kappa(\ee) \to 0$ and due to \eqref{eq:s0}, it is easy to construct such family of $\Phi_\ee$.

We also construct a suitable $\mobee \in C^1([0,\alpha])$ such that, for $\ee > 0$,
\begin{equation}
    \tag{M$_\ee$}
    \label{hyp:Sublinear saturation}
    \begin{aligned}
         & \mob_\ee(0) = \mob_\ee(\alpha) = 0, \qquad \mob_\ee > 0 \text{ in } (0,\alpha), \qquad
        \text{and } |\mobee'| > 1 \text{ near } 0 \text{ and } \alpha .
    \end{aligned}
\end{equation}
To connect this problem with \eqref{eq:the problem Omega} we consider the approximation of $U$ given by the conditions $U_\ee \in
    C^{3,1}((0,\alpha))$, $U_{\varepsilon} (\frac{\alpha}{2}) = U (\frac \alpha 2 ),  U_{\varepsilon}'(\frac{\alpha}{2}) = U'(\frac \alpha 2)$, and
we define
$
    U_{\varepsilon}'' (s) \coloneqq \Phi_{\varepsilon}' (s) / \mob_\ee (s) .
$
By construction, we already have $U_\ee'' \to U''$ point-wise in $(0,\alpha)$.
We make a few more assumptions for the convergence as $\ee \to 0$, namely
\begin{equation}
    \tag{M$_0$}\label{eq:mobee a.e. mob}
    \begin{aligned}
        \mobee \rightarrow \mob \text{ pointwise as } \ee \to 0, \qquad
        \sup_{\ee \in (0,1)} \| U_\ee' \|_{L^1((0, \alpha))} < \infty \text{ and } \\
        (\ee,s) \mapsto \mobee (s) \text{ is } C([0,1] \times [0,\alpha]) \cap C^\infty_{loc} ((0,1]\times[0,\alpha]).
    \end{aligned}
\end{equation}
In \cref{sec:Well Posedness Pee} we construct such regularised mobilities.
\begin{lemma}
    \label{lem:mobee exists}
    There exists $\kappa(\ee) \to 0$ and $\mobee$ such that \eqref{hyp:Sublinear saturation} and \eqref{eq:mobee a.e. mob} hold.
\end{lemma}
Moreover,  if we consider the free energy of the regularised problem  
\begin{equation}\label{eq:Free-energy}
    \mathcal{F}_{\varepsilon} [\rho] = \int_{\Omega} U_{\varepsilon}(\rho (x)) \, dx + \int_{\Omega} V(x) \rho (x) \, dx ,
\end{equation}
we can rewrite \eqref{eq:the problem regularised} again as a formal generalized mobility Wasserstein-type gradient flow.

First, we will prove a well-posedness result for \eqref{eq:the problem regularised}.
\begin{theorem}[Well-posedness of \eqref{eq:the problem regularised}]\label{thm:Properties (Pee)}
    Let $\ee > 0$ be fixed and assume \eqref{hyp:Sublinear saturation} and \eqref{eq:properties of Phi_ee}. It follows that:
    \begin{enumeratetheorem}
        \item
        \label{item:(Pee) classical solutions}
        If $\rho_0 \in \mathcal{A}_+ \cap C^2 (\overline{\Omega})$, then problem \eqref{eq:the problem regularised} has a unique classical solution.

        \item
        \label{item:(Pee) good semigroup}
        These classical solutions can be uniquely extended to a \goodSemigroup{}  for \eqref{eq:the problem regularised}, denoted by $S_t^{(\ee)}$.

        \item
        \label{it:S^eps maximum principle}
        % Strong maximum principle:

        If $\rho_0 \in \mathcal A \setminus \{0,\alpha\}$ then $0 < S_t^{(\ee)} \rho < \alpha$ in $\Omega$ for $t > 0$.

        \item
        \label{it:S^eps strict positivity}
        $S_t^{(\ee)} : \mathcal A_+ \to \mathcal A_+$.
    \end{enumeratetheorem}
\end{theorem}

We show \ref{item:(Pee) classical solutions}  in \Cref{sec:Well Posedness Pee} and discuss the remaining items in \cref{sec:Semigroup}.
We can now prove the left side of the diagram \eqref{Diagram}.

\begin{theorem}[Existence for \eqref{eq:the problem Omega}]
    \label{thm:left side of D}
    There exists a sequence $\ee_k \to 0$ and $S_t : \mathcal A \to \mathcal A$ a \goodSemigroup{} for \eqref{eq:the problem Omega} such that
    \[
        S^{(\ee_k)} \rho_0 \to S \rho_0 \qquad \text{in } C_{loc} ( [0,\infty); L^1( \Omega ) )\text{ for all } \rho_0 \in \mathcal A.
    \]
\end{theorem}

We prove the existence result on \cref{sec:Existence P0}.

\begin{remarks*}
    Notice in \eqref{hyp:Sublinear saturation} and \eqref{eq:mobee a.e. mob} that since $\mob_\ee \to \mob $ pointwise in $[0,\alpha]$ and the fact that the map $(\ee, s) \mapsto \mobee(s)$ is $C^1$ then it holds the convergence
    \begin{equation}\label{eq:mob uniform}
        \mobee \to \mob \quad \text{in } C([0,\alpha]) \cap C^1_{loc} ((0,\alpha)).
    \end{equation}
    The semigroup can also be constructed using the theory of m-accretive operators in $L^1(\Omega)$, see e.g., \cite{Brezis73}.
    With the construction we have made $U_\ee'' \ge \ee / \mobee > 0$, $\mobee(0) = \mobee(\alpha) = 0$, and $|\mobee'| > 1$ at $0,\alpha$, so
    \[
        U_\ee' : (0,\alpha) \to (-\infty,\infty) \text{ is a strictly-increasing bijection with continuous inverse}.
    \]
    Therefore, for each $C \in \mathbb R$ we have
    \[
        \rho_t^{(\ee)} = (U_\ee')^{-1} (C - V) \text{ is a constant-in-time classical solution to \eqref{eq:the problem regularised}}.
    \]
    Notice that  $\rho_t^{(\ee)}$ tend uniformly to $0$ as $C \to -\infty$ and to $\alpha$ as $C \to \infty$.
    We point out that if $\rho_0 \in \mathcal A_+$ then
    \begin{equation}
        \label{eq:initial datum A+ bounds}
        (U_\ee')^{-1} (C_1 - V) \le \rho_0 \le (U_\ee')^{-1} (C_2 - V) \text{ for some }C_1, C_2 \in \mathbb R.
    \end{equation}
    \ref{it:S^eps strict positivity} follows by using these bounds and the comparison principle can also be used for $S_t^{(\ee)} \rho_0$. 
    Notice that
    \begin{equation*}
        0 \le \Phi_\ee' \le (1+\varepsilon) \left( \Phi' + \varepsilon \right)
    \end{equation*}
    Since $\Phi_\ee' \to \Phi'$ uniformly over compacts of $(0,\alpha)$ we have also almost everywhere convergence.
    Since $\Phi' \in L^1(0,\alpha)$ we can use the Dominated Convergence Theorem to show that
    \begin{equation}
        \label{eq:stronger convergence of Phiee}
        \Phi_\ee' \to \Phi' \qquad \text{in } L^1(0,\alpha).
    \end{equation}
    This, in particular, implies that $\Phi_\ee \to \Phi$ uniformly in $[0,\alpha]$.
\end{remarks*}

\subsection{\texorpdfstring{$L^1$}{L1}-local minimisers of the free energy}

In \Cref{sec:Local minimiser} we characterise the $L^1$-local minimisers, by deducing (and solving in some cases) the corresponding Euler-Lagrange condition. 
Since \eqref{eq:the problem regularised} are particular examples of the family \eqref{eq:the problem Omega}, we state the following result only for the later, more general, class.

\begin{theorem}[Euler-Lagrange condition]
    \label{thm:Euler-Lagrange}
    If $\widehat \rho$ is a local minimiser of $\mathcal F$ on $\mathcal A_M$ with the $L^1$ topology, then there exists $C \in \mathbb R$ such that
    \begin{equation}
        \label{eq:EL for ee = 0}
        \begin{aligned}
            U'(\widehat \rho(x)) + V(x) \ge C, & \qquad \text{ for a.e. } x \text{ such that } 0 \le \widehat \rho(x) <\alpha.  \\
            U'(\widehat \rho(x)) + V(x) \le C, & \qquad \text{ for a.e. } x \text{ such that } 0 < \widehat \rho(x) \le\alpha .
        \end{aligned}
    \end{equation}
    Furthermore, if $U'$ is invertible
    \begin{equation}
        \label{eq:Euler-Lagrange for P}
        \widehat \rho (x)= \inversedU (C- V(x)) \quad \text{a.e.~in } \Omega.
    \end{equation}
    Lastly, if we assume \eqref{eq:U locally strictly convex} and $M \in (0,\alpha|\Omega|)$, there exists a unique $C$ such that \eqref{eq:Euler-Lagrange for P} has mass $M$.
\end{theorem}

\begin{remark*}
    The Euler-Lagrange condition for the case of $\mob(\rho) = \rho$ is well-understood  \cite{Balague_Carrillo_Laurent_Raoul13, CDM16,Carrillo_Delgadino_Patacchini19}. Here, we adapt these techniques in order to study the case with saturation.
\end{remark*}

\subsection{Long-time behaviour. Relation to free-energy minimisers}\label{sec:Main results Long time behaviour}

We begin this section by giving an interpretation of long-time behaviour in terms of semigroups.
\begin{definition}\label{def:timeLimit}
    We say that a semigroup $S$ for a problem \eqref{eq:the problem Omega}  has a \timeLimit{} $S_\infty : \mathcal A \to \mathcal A$ if all the following are satisfied:
    \begin{enumeratedefinition}
        \item For any $\rho_0 \in \mathcal A$ there exists a limit in time
        \[
            S_t \rho_0 \to S_\infty \rho_0 \qquad \text{ strongly in } L^1(\Omega) \text{ as } t \to \infty.
        \]
        \item $S_\infty$ is stationary for the semigroup, i.e., $S_t S_\infty = S_\infty$.
        \item For any $\rho_0 \in \mathcal A$, $S_\infty \rho_0$ is a constant-in-time weak solution to \eqref{eq:the problem Omega}.
    \end{enumeratedefinition}
\end{definition}

With this definition we can now proceed to study the other three sides of the diagram \eqref{Diagram}. First, in \Cref{sec:Long time behaviour} we construct a \timeLimit{}\, for the problems \eqref{eq:the problem regularised} and \eqref{eq:the problem Omega}.

\begin{theorem}[Long-time behaviour for \eqref{eq:the problem Omega} and \eqref{eq:the problem regularised}]
    \label{thm:long-time behaviour}
    We have that:
    \begin{enumeratetheorem}
        \item
        \label{item:Time limit operator regularised}
        For $\ee >0$ the \goodSemigroup{} $S^{(\ee)}$ for \eqref{eq:the problem regularised} has a \timeLimit{}, which we denote $S_\infty^{(\ee)}$.

        \item
        \label{item:Time limit operator}
        Any \goodSemigroup{} $S$ for \eqref{eq:the problem Omega} has a \timeLimit{}, which we denote $S_\infty$.
    \end{enumeratetheorem}
    Both $S^{(\ee)}_\infty$ and $S_\infty$ are $L^1$-contractions.
\end{theorem}

In \cref{sec:Asymptotic behaviour ee} we study the \timeLimit{}\, of problem \eqref{eq:the problem regularised}.

\begin{theorem}
    [The global attractors of \eqref{eq:the problem regularised}]
    \label{th:ee continuous limit}
    Let $\ee > 0$ be fixed, $M \in (0, \alpha |\Omega|)$ and define the corresponding form of \eqref{eq:Euler-Lagrange for P}, i.e.,
    \begin{equation}
        \label{eq:ee time limit}
        \widehat\rho^{(\ee)} (x) \coloneqq \inversedUee \left( C_{\varepsilon} - V(x) \right), \quad \text{in } \Omega ,
    \end{equation}
    where $C_\ee$ is uniquely determined by the mass condition
    $
        \int_\Omega \inversedUee \left( C_{\varepsilon} - V \right) = M.
    $
    Then:
    \begin{enumeratetheorem}
        \item
        \label{item:Pee long time 1}
        $\widehat\rho^{(\ee)}$ is the unique fixed-point of the semigroup in $\mathcal A_M$, i.e., $\widehat \rho \in \mathcal A_M$ such that $S_t^{(\ee)} \widehat \rho = \widehat \rho$ for all $t > 0$. \\
        In particular, it is the global attractor, i.e., for any $\rho_0 \in \mathcal A_M$ we have  $S^{(\ee)}_\infty \rho_0 = \widehat \rho^{(\ee)}$.

        \item
        \label{item:Pee long time 2}
        $\widehat\rho^{(\ee)}$ is the unique constant-in-time weak solution of \eqref{eq:the problem regularised} in $\mathcal A_M \cap \mathcal A_+$.

        \item
        \label{item:Pee long time 3}
        $\widehat\rho^{(\ee)}$ is the unique $L^1$-local minimiser of the free energy $\mathcal F_\ee$ over $\mathcal A_M$.
        It is also the unique global minimiser of the free energy $\mathcal F_\ee$ over $\mathcal A_M$.
    \end{enumeratetheorem}
\end{theorem}

We next focus on studying the $\omega$-limit of \eqref{eq:the problem Omega} in more detail. First, in \cref{sec:Stationary state} we  obtain the following result.

\begin{theorem}[On steady states for \eqref{eq:the problem Omega}]
    \label{thm:Asymptotic (P0)}
    Assume \eqref{eq:U locally strictly convex} and that $M \in (0,\alpha|\Omega|)$.
    Then we can define
    \begin{equation*}
        \widehat\rho^{(0)} (x) \coloneqq \inversedU (C_0 - V(x) ), \quad \text{in } \Omega,
    \end{equation*}
    where $C_0$ is uniquely determined by the mass condition
    $
        \int_{\Omega} \inversedU (C_0 - V ) = M.
    $
    We have that:
    \begin{enumeratetheorem}
        \item
        \label{item:Semi-group local minimiser}
        $S_t \widehat\rho^{(0)} = \widehat\rho^{(0)}$ for any $t > 0$.

        \item
        \label{item:minimiser free energy}
        $\widehat\rho^{(0)}$ is the unique $L^1$-local minimiser of the free energy \eqref{eq:Free eenrgy ee=0 Omega} over $\mathcal A_M$.
        It is the unique global minimiser over $\mathcal A_M$.
        \item
        \label{item:Limit of steady states}
        If we also assume \eqref{eq:U uniformly strictly convex},
        then $\widehat\rho^{(0)}$ is the limit as $\ee \to 0$ of the constant-in-time weak solutions of \eqref{eq:the problem regularised} given by
        \eqref{eq:ee time limit}
        \begin{equation*}
            \widehat{\rho}^{(\ee)} \rightarrow \widehat{\rho}^{(0)} \quad \text{in } L^1(\Omega ) \text{ as } \ee \rightarrow 0.
        \end{equation*}

    \end{enumeratetheorem}
\end{theorem}

We finally show that the double limits in diagram \eqref{Diagram} do not commute by giving a counterexample. In \cref{sec:malicious counterexamples}, we provide a non-linearity $U$, and a potential $V$, such that the problem \eqref{eq:the problem Omega} has infinitely many steady states different from $\widehat \rho^{(0)}$. Moreover, each of them have a large basin of attraction of initial data.

\begin{remarks*}
    The long-time behaviour result holds even if $\mathcal F[\rho_t]$ never becomes finite. This is a powerful consequence of the $L^1$-contraction theory.
    The Wasserstein-type gradient-flow theory is usually not able to deal with these cases. In \cref{sec:malicious counterexamples}  we construct $\mob, U, V$ such that the global minimiser of $\mathcal F$ is not the global attractor, i.e., there exists $\rho_0$ such that $S_\infty \rho_0 \ne \widehat \rho^{(0)}$. Furthermore, we construct a curve of stationary weak solutions such that each of them attracts some initial data (see \Cref{fig:Double_Well_V} and \Cref{fig:Mass_saddle_point}). For certain choices of $V$ (e.g., $V$ convex) it is easy to show that there exists a global attractor for \eqref{eq:the problem Omega}. See, e.g., \cite[Section 3]{Carrillo_Jungel_Markowich_Toscani_Unterreiter01} and \cite{Kim_Lei10}.
\end{remarks*}

\subsection{Numerical analysis}\label{sec:Intro Numerical analysis}

We now study the implicit Finite-Volume scheme proposed by Bailo, Carrillo, and Hu in \cite{BCH23}. Here, we consider a small generalisation of the method in \cite{BCH23} that also fit  to the regularised problem \eqref{eq:the problem regularised}. For the sake of clarity, in this manuscript we only cover the $1$-dimensional case.
We devote \Cref{sec:Numerical Analysis} to the numerical analysis. Let us fix $\ee \ge 0$. 
We prove a first result which we will need for the structure of our paper.
We will assume that
\begin{equation}
    \tag{M$^{\Delta}$}
    \label{hyp:mobility for numerics}
    \begin{aligned}
         & \mob (s) = \mob^{(1)}(s) \mob^{(2)} (s),
    \end{aligned}
\end{equation}
where $\mob^{(j)}$ are Lipschitz continuous, $\mob^{(1)}$ is non-decreasing, and $\mob^{(2)}$ is non-increasing.
This is required to perform the up-winding below. 
This assumption is not too restrictive, since we show the following result in \cref{sec:Numerical method}.
\begin{lemma}
    \label{lem:existence of decomposition}
    Assume \ref{hyp:mobility}.
    Then, there exists $\mob^{(1)} \in C^1((0,\alpha))$ non-decreasing and $\mob^{(2)} \in C^1((0,\alpha))$ non-increasing such that \eqref{hyp:mobility for numerics}.
    Furthermore, if $\mob$ is Lipschitz in $[0,\alpha]$, $\mob' \ge 0$ in a neighbourhood of $0$ and $\mob' \le 0$ in a neighbourhood of $\alpha$, then $\mob^{(1)}$ and $\mob^{(2)}$ are Lipschitz.
\end{lemma}

Without loss of generality, we can restrict to $x \in (0,1)$, we pick
$I = \{1, \cdots, N\}$ for some
$N \in \mathbb N $,
and let $\Delta x = 1/N$.
We now present (a small generalisation of) the Finite-Volumes scheme constructed in \cite{BCH23}.
The method is given as
\begin{equation}
    \label{def:Numerical method}
    \tag{P$^{\Delta}$}
    \begin{aligned}
        \frac{\rho_i^{ n+1} - \rho_i^{n}}{\Delta t} & = - \frac{F^{}_{i+\frac 1 2} (\rho^{n+1}) - F^{}_{i-\frac 1 2} (\rho^{ n+1})}{\Delta x}  , \qquad i \in I, n \in \mathbb N , \\
        F_{i+\frac 1 2}^{} (\rho)                   & =
        \mob^{(1)}(\rho_i) \mob^{(2)} (\rho_{i+1})   (v_{i+\frac 1 2}^{} (\rho))^+
        + \mob^{(1)}(\rho_{i+1}) \mob^{(2)} (\rho_{i}) (v_{i+\frac 1 2}^{}(\rho))^- ,                                                                                              \\
        v_{i+\frac 1 2}^{} (\rho)                   & = -\frac{\xi_{i+1}^{}(\rho) - \xi_{i}^{}(\rho)}{\Delta x} ,                                                                  \\
        \xi_i^{}(\rho)                              & = U_{}'(\rho_i) + V (x_i) ,                                                                                                  \\
        F_{\frac 1 2}^{} (\rho)                     & = F_{N + \frac 1 2}^{} (\rho) = 0 .
    \end{aligned}
\end{equation}
Here, we are using the notation $u = u^+ + u^-$.
We consider the initial condition
\begin{equation}
    \rho^0_i = \frac{1}{\Delta x} \int_{x_{i-\frac 1 2}}^{x_{i+\frac 1 2}} \rho_0(x) \diff x.
\end{equation}
In \cite[Theorem 2.4]{BCH23} the authors prove the decay of a discrete energy, defined for $\ee \ge 0$ as
\begin{equation}\label{def:Discrete energy}
    \Edisc[\rho] =
    \Delta x \sum_{i \in I} \left( U(\rho_i) + V(x_i) \rho_i  \right).
\end{equation}
We introduce the following discrete version of \eqref{eq:the problem regularised}:
\begin{equation}
    \label{eq:problem regularised discrete}
    \tag{P$_\varepsilon^{\Delta}$}
    \begin{minipage}{.5\textwidth}
        \eqref{def:Numerical method} when we replace $\mob^{(i)}$ and $U$ by $\mobee^{(i)}$ and $U_\ee$. \\
        We will use the notation $\rho^{\ee,n}, F^\ee, v^\ee, \xi^\ee$. \\
        The corresponding energy is denoted $\Ediscee$.
    \end{minipage}
\end{equation}

\begin{remarks*}
    For the approximating problems, due to ellipticity, $U_\ee'(0) = -\infty, U_\ee'(\alpha) = \infty$. We can only deal with solutions in $0 < \rho_i^0 < \alpha$.
    This is also the case in general if $U \not \in C^1([0,\alpha])$. We will separate the case of $U \in C^1([0, \alpha])$.
    Even if $U \in C^1([0,\alpha])$, from the computational point of view the approximating problem \eqref{eq:problem regularised discrete} is usually better than \eqref{def:Numerical method} since typically we will be able to use Newton iteration to compute the next step (we can guarantee the Jacobian is invertible).
    The scheme introduced in \cite{BCH23} is done for $\mob^{(1)}(s) = s$ and $\xi_i = U'(\rho) + V(x_i) + W*\rho$. This problem is significantly more difficult. For example, it does not have a comparison principle.
\end{remarks*}

\subsubsection{Finite-time properties of the numerical schemes}

We introduce
\begin{align*}
    \mathcal A_{\Delta } = \{ \rho \in \mathbb R^{|I|} : 0 \le \rho_i \le \alpha \},
    \qquad
    \mathcal A_{\Delta , +} = \{ \rho \in \mathbb R^{|I|} : 0 < \rho_i < \alpha \},
    \qquad
    \| \rho \|_{L^1_{\Delta }} = \sum_{i \in I} |\rho_i|.
\end{align*}
We will use the notation $\Delta = ( \Delta t , \Delta x )$.
In the same fashion we worked with semigroups, we can write
\[
    \numericStep : \rho^{0} \in \mathcal A_{\Delta} \mapsto \rho^{1}, \qquad \numericApproxStep : \rho^{0} \in \mathcal A_{\Delta} \mapsto \rho^{\ee,1} .
\]
Formally, $\rho^{n} \coloneqq (\numericStep)^n \rho^{0} $.
For $\rho, \eta \in \mathcal A_{\Delta ,+}$ we say that $\rho \le \eta$ if $\rho_i \le \eta_i$ for all $i\in I$.
We use the following definition.

\begin{definition}\label{def:goodNumericalProblem}
    Fix $\Delta$. We say that \eqref{def:Numerical method} is a \goodNumericalProblem{} over $\mathcal B \subset \mathcal A_{\Delta}$ if
    \begin{enumeratedefinition}
        \item
        \label{it:numerics existence under finite energy}
        For $\rho^{0} \in \mathcal B$, there exists a unique $\rho \in \mathcal B$ that solves \eqref{def:Numerical method}. We call this unique solution $\rho^{1}$.

        \item
        \label{it:numerics semigroup}
        The solution map
        There is mass conservation.

        \item
        \label{it:numerics energy dissipation}
        There is free-energy dissipation, i.e., for $\rho^0 \in \mathcal B$ we have
        $\Edisc [\numericStep \rho^{0}] \leq \Edisc [\rho^{0}]$.
    \end{enumeratedefinition}
\end{definition}

In \cite{BCH23} the authors do not prove that this implicit scheme admits a solution, or whether it is unique.

\begin{theorem}[Well-posedness theory]
    \label{thm:scheme well-posedness}
    We have that
    \begin{enumeratetheorem}
        \item \label{thm:scheme well-posedness regularised}
        \eqref{def:Numerical method}
        and
        \eqref{eq:problem regularised discrete}
        for $\ee > 0$
        are \goodNumericalProblems{} in $\mathcal A_{\Delta , +}$.
        If $\rho^0 \in \mathcal A_{\Delta, +}$ then $\numericApproxStep \rho^0 \to \numericStep \rho^0$ as $\ee \to 0$.

        \item
        \label{thm:scheme well-posedness with free boundary}
        If we also assume $U\in C^1([0,\alpha])$ then
        \eqref{def:Numerical method} is a \goodNumericalProblem{} in $\mathcal A_{\Delta}$.
    \end{enumeratetheorem}

\end{theorem}
This scheme is convergent at least under high regularity of the solution. We only include a small remark below about the regularity result, we will not discuss it any further and just focus on the convergence.
\begin{theorem}[Convergence as $\Delta \to 0$]
    \label{thm:Disc to continuous}
    Let $\rho_0 \in \mathcal A_{+}$ be fixed, $\rho$ a solution to \eqref{eq:the problem Omega},
    and $R \coloneqq \rho([0,T]\times \overline \Omega) \subset [0,\alpha]$,
    $U \in C^{3+\gamma}(R)$, $\mobee^{(j)} \in C^{1+\gamma}(R)$ for $j=1,2$, $V \in C^{2+\gamma}(\overline \Omega)$, and $\rho \in C^{1+\beta}_t, C^{2+\gamma}_x$. Then,
    if $\rho_i^n$ is the solution to \eqref{def:Numerical method} it is such that
    \[
        \sup_{ 0 \le n \le \frac{T}{\Delta t}}\Delta x \sum_{i \in I } |\rho_i^{n} - \rho(t_n, x_i)| \le C( (\Delta t)^\beta + (\Delta x)^\gamma).
    \]
    Hence, the numerical solution converges to the solution of the continuous problem as $\Delta \to 0$.
\end{theorem}
Notice that this theorem can also be applied to \eqref{eq:the problem regularised} if the solution of the continuous problem had the regularity needed.
The proof of these results can be found in \cref{sec:Numerical method}.

\begin{remarks*}
    The comparison principle
    holds,
    i.e., for $\underline\rho^0, \overline\rho^{0} \in \mathcal B$ if $\underline\rho^{0} \le \overline\rho^{0}$ then $\numericStep \underline\rho_i^{0} \le \numericStep\overline\rho_i^{0}$.
    In order to obtain the convergence result presented in
    \Cref{thm:Disc to continuous}, we rely on consistency and stability. Thus, it requires for the solution of the problem to be smooth. We expect this to happen in \eqref{eq:the problem regularised} with $\ee > 0$ due to uniform ellipticity. For \eqref{eq:the problem Omega} we can construct examples where it does not hold.
    \Cref{thm:Disc to continuous} implies uniqueness of smooth solutions of \eqref{eq:the problem Omega}.
\end{remarks*}

\subsubsection{Long-time behaviour for the numerical problems}

In \cref{sec:Discrete regalurised stationary state} we focus on the long-time analysis of the numerical method. First, we study the limit in the time step of the solutions of \eqref{eq:problem regularised discrete}.

\begin{theorem}[Asymptotics for \eqref{eq:problem regularised discrete}]
    \label{thm:discrete approx asymptotics}
    Let $\ee > 0$ be fixed. Assume \eqref{hyp:mobility for numerics}, $M \in (0,\alpha|\Omega|)$ and let
    \begin{equation*}
        \rho^{\ee , \infty}_i = \left( U_{\ee}' \right)^{-1} \left( C_{\ee}^{\Delta } - V(x_i) \right),
    \end{equation*}
    where $C_{\ee}^{\Delta }$ is uniquely determined by the mass condition
    $
        \Delta x \sum_i \left( U_{\ee}' \right)^{-1} \left( C_{\ee}^{\Delta } - V(x_i) \right) = M.
    $
    Then:
    \begin{enumeratetheorem}
        \item
        \label{it:discrete approx unique steady state}
        $\rho^{\ee , \infty}$ is a unique constant-in-time solution of \eqref{eq:problem regularised discrete} in $\mathcal A_{\Delta , +}$ of mass $M$.

        \item
        \label{it:discrete approx global attractor}
        It is the global attractor, i.e., for any
        $(\numericApproxStep)^n\rho^0
            \rightarrow \rho^{\ee , \infty}$ in  $\R^{|I|}$
        as $n \rightarrow \infty $.

        \item
        \label{thm:Asymptotic convergence discrete to continuous ee}
        Consider $\widehat \rho^{(\ee)}$ obtained in \Cref{th:ee continuous limit}. Then
        $
            | C_\ee^{\Delta } - C_{\ee}| \le \mathsf C (\ee) \Delta x.
        $
    \end{enumeratetheorem}

\end{theorem}

For \eqref{def:Numerical method} we can prove existence of a steady state, but not its uniqueness. Furthermore, we study some properties of one of the steady states.

\newcommand{\numericAsymp}{J^{\Delta, \infty}}

\begin{theorem}[Asymptotics for \eqref{def:Numerical method}]
    \label{thm:Asymptotics discrete P0}
    Assume \eqref{eq: U is C1[0,alpha]}. Then:
    \begin{enumeratetheorem}
        \item
        \label{it:discrete P0 exists asymptotic limit}
        For every $\rho^0 \in \mathcal A$ there exists $\numericAsymp \rho^0 \coloneqq \lim_n (\numericStep)^n \rho^0$. This limit is a fixed point of $J^{\Delta}$ (i.e., a constant-in-time solution).

        \item
        \label{it:discrete P0 asymptotic limit is L1 contraction}
        The operator $\numericAsymp$ is an $L^1_{\Delta}$-contraction.
    \end{enumeratetheorem}
    \noindent Assume, furthermore, that $U$ satisfies
    \eqref{eq:U locally strictly convex}.
    Let $M \in (0,\alpha|\Omega|)$ and define
    \begin{equation*}
        \rho^{0, \infty}_i =\inversedU (C_0^{\Delta } - V(x_i)), \quad \text{in } \R^{|I|},
    \end{equation*}
    where $C_0^{\Delta }$ is determined by the mass condition
    $
        \Delta x \sum_i \inversedU (C_0^{\Delta } - V (x_i)) = M .
    $
    Then:
    \begin{enumeratetheorem}[resume]
        \item
        \label{it:discrete P0 discrete global minimiser}
        $\rho^{0, \infty}$ is a constant-in-time solution to \eqref{def:Numerical method}, i.e., $\numericStep\rho^{0, \infty} = \rho^{0, \infty}$. In particular, $J^{\Delta ,\infty}  \rho^{0, \infty} = \rho^{0, \infty} $.

        \item
        \label{it:discrete asymptotic state limit as ee to 0}
        For the same mass, $\rho^{\ee , \infty} \rightarrow \rho^{0, \infty}$  as $ \ee \to 0$.

        \item
        \label{cor:Convergence disc to cont stat state ee=0}
        Let $\widehat{\rho}^{(0)}$ be as obtained in \Cref{thm:Asymptotic (P0)}.
        Then,
        $C_0^{\Delta } \rightarrow C_0$ as $\Delta x \to 0$.
        If \eqref{eq:U uniformly strictly convex} and $(U')^{-1} \in C^\gamma ([0,\alpha])$, then $ |C_0^{\Delta } - C_0 | \le \mathsf C (\Delta x)^\gamma.$
    \end{enumeratetheorem}
\end{theorem}

As we mention above for \eqref{def:Numerical method} we cannot prove uniqueness of a steady state. In \cref{sec:Num counterexample} we reproduce the example from \cref{sec:malicious counterexamples}. Thus, analogously to the continuous case, in the discrete problem there exists a steady state different from $\rho^{0, \infty}$ that attracts a large class of initial data. Finally, in \cref{sec:Numerical experiments} we show some numerical experiments.

\subsection{A complete convergence diagram}
\label{sec:Diagram}

As a summary of the results in the previous sections, we present the diagram \eqref{Numeric Diagram} superseding the diagram \eqref{Diagram}.
This diagram summarises the most important results of this work.
\begin{equation}\label{Numeric Diagram}\tag{D$_{2}$}
    \begin{tikzpicture}[baseline=(current  bounding  box.center)]
        \matrix[matrix of math nodes,column sep={8.75em,between origins},row sep={5.5em,between origins},nodes={inner sep=0.2pt}]
        {
        |(e)|  \hyperref[thm:scheme well-posedness]{\rho^{\ee,n}_i \text{ solution to \eqref{eq:problem regularised discrete}}}
        & & |(0)|
        \hyperref[thm:discrete approx asymptotics]{\rho^{\varepsilon, \infty}_i
            =(U_\varepsilon')^{-1} \Big( C_\ee^{\Delta }   - V (x_i)         \Big)}
        \\
        & |(1)|  \hyperref[thm:Properties (Pee)]{\rho^{(\ee)}_t \text{ solution to \eqref{eq:the problem regularised}}} & & |(01)|   \hyperref[th:ee continuous limit]{\widehat \rho^{(\varepsilon)}(x)
            =(U_\varepsilon')^{-1} \Big( C_\ee   - V (x) \Big)}  \\
        |(2)|  \hyperref[thm:scheme well-posedness]{\rho_i^{0,n} \text{ solution to \eqref{def:Numerical method}}} & & |(02)| \hyperref[thm:Asymptotics discrete P0]{\rho^{0, \infty}_i
        = T_{0,\alpha} \circ (U')^{-1} \Big( C_0^{\Delta }   - V (x_i)         \Big)} \\
        & |(12)| \hyperref[thm:left side of D]{\rho^{(0)}_t \text{ solution to \eqref{eq:the problem Omega}}} & & |(012)| \hyperref[thm:Asymptotic (P0)]{\widehat{\rho}^{(0)} (x)
        =T_{0 , \alpha} \circ (U')^{-1} \Big( C_0   - V (x)         \Big)} \\
        };

        \draw[->] (e)  to node[anchor=north west] {\hyperref[thm:scheme well-posedness]{\scriptsize $\ee \to 0$}} (2);
        \draw[->] (0)  to node[above=-8mm]%[anchor=north west] 
        {\, \, \, \, \, \,  \hyperref[it:discrete asymptotic state limit as ee to 0]{\scriptsize $\ee \to 0$}} (02);
        \draw[->] (e)  to node[anchor=south ] {\hyperref[it:discrete approx global attractor]{\scriptsize $n \to \infty$}} (0);
        \draw[->,dotted] (2)  to node[anchor=south east] {\hyperref[sec:Num counterexample]{\scriptsize ?}} node[anchor=north] {\hyperref[it:discrete P0 exists asymptotic limit]{\scriptsize $n \to \infty$} $\, \,$} (02);

        \draw[->,over] (1)  to node[above=7mm] {$\qquad \, \, \, $ \hyperref[thm:left side of D]{\scriptsize $\ee \rightarrow 0$}} (12);
        \draw[->]      (01) to node[anchor= west] {\hyperref[item:Limit of steady states]{\scriptsize $ \ee \to 0$}} (012);
        \draw[->,over] (1)  to node[anchor=south west] {\hyperref[item:Pee long time 1]{\scriptsize $t \to \infty$}} (01);
        \draw[->,dotted]      (12) to node[anchor=north west] {\hyperref[thm:Asymptotic (P0)]{\scriptsize $t \to \infty$}} node[anchor=south west] {$\, \, \,$ \hyperref[sec:malicious counterexamples]{\scriptsize ?}} (012);

        \draw[->] (e)  to node[anchor= west] {$\, \, \, \,$ \hyperref[thm:Disc to continuous]{\scriptsize $\Delta \to 0$}} (1);
        \draw[->] (0)  to node[anchor= south west] {\hyperref[thm:Asymptotic convergence discrete to continuous ee]{\scriptsize $\Delta  \to 0$}} (01);
        \draw[->] (2)  to node[anchor= west] {$\, \, \, \, $ \hyperref[thm:Disc to continuous]{\scriptsize $\Delta \to 0$}} (12);
        \draw[->] (02) to node[anchor= west] {$\qquad $\hyperref[cor:Convergence disc to cont stat state ee=0]{\scriptsize $\Delta  \to 0$}} (012);
    \end{tikzpicture}
\end{equation}
Notice again that  a similar counterexample for the commutativity of the diagram \eqref{Diagram} can be given here to the commutativity of the backward face of \eqref{Numeric Diagram}  in \cref{sec:Num counterexample}.

\section{Well-posedness theory}
\label{sec:Analysis of Pee}

We divide this section as follows. In \cref{sec:Well Posedness Pee} we construct the regularised mobility $\mobee$, we obtain some auxiliary results, we discuss  well-posedness, and we obtain a Strong Maximum Principle using classical theory.
In \cref{sec:a priori} we obtain some \textit{a priori} estimates for the sequence $\rho$ that come from the structure of the equation. Then, in \cref{sec:Semigroup} we understand the problem through a free-energy dissipating semigroup. Finally, we devote \cref{sec:Existence P0} to prove existence of weak solutions for the problem \eqref{eq:the problem Omega} as a limit of classical solutions of the problem \eqref{eq:the problem regularised}.

Before we begin, we recall some functional analysis results. We will take advantage of the negative Sobolev space
\begin{equation*}
    W^{-1,1} (\Omega) \coloneqq \left\lbrace \rho \in \mathcal{D}'(\Omega) : \, \exists F \in L^1(\Omega) \, \, \mathrm{s.t.} \, \,  \rho = \mathrm{div}(F) \right\rbrace
    \quad
    \text{ with } \| \rho \|_{W^{-1,1}(\Omega)} := \inf_{\rho = \diver (F)} \| F \|_{L^1(\Omega)}.
\end{equation*}
This space will be very helpful, since it allows us to use the free-energy dissipation to prove that if $\rho$ solves \eqref{eq:the problem Omega} then 
\begin{align}\label{eq:drho/dt W -1,1}
    \begin{split}
        \int_{t_1}^{t_2}  \left\| \frac{\partial \rho}{\partial t} \right\|^2_{W^{-1, 1}(\Omega)} & = \int_{t_1}^{t_2} \left\|  \diver  \left( \mob (\rho) \nabla \left( U' \left( \rho \right) +  V   \right)  \right)  \right\|^2_{W^{-1, 1}(\Omega)}                                        \\
        & \le \int_{t_1}^{t_2} \left\|  \mob (\rho) \nabla \left( U' \left( \rho \right) +  V   \right)   \right\|^2_{L^{1}(\Omega)}                                                                 \\
        & \leq  \| \mob \|_{L^{\infty}([0,\alpha])} |\Omega| \int_{t_1}^{t_2}   \left\|  \mob (\rho)^{\frac{1}{2}} \nabla \left( U' \left( \rho \right) +  V    \right)   \right\|^2_{L^{2}(\Omega)} \\
        & \leq  \| \mob \|_{L^{\infty}([0,\alpha])} |\Omega| \left( \mathcal{F}[\rho_{t_1} ] - \mathcal{F} [ \rho_{t_2} ] \right).
    \end{split}
\end{align}
The decay of the free energy allows to control the right-hand side uniformly. 
\begin{lemma}
    \label{lem:negative Sobolev Ascoli-Arzela}
    Let $\rho^{[n]} \in C([t_1,t_2]; W^{-1,1}(\Omega))$ such that
    \[
        \int_{t_1}^{t_2}  \left\| \frac{\partial \rho^{[n]}}{\partial t} \right\|^2_{W^{-1, 1}(\Omega)} \le C
    \]
    and uniformly bounded in $L^1((t_1,t_2) \times \Omega)$. Then, up to a subsequence,
    \begin{equation*}
        \rho^{[n]} \rightarrow \rho^{\infty} \text{ in } C([t_1,t_2]; W^{-1,1}(\Omega)) \text{ as } n \to \infty.
    \end{equation*}
\end{lemma}

\begin{proof}
    By properties of the Bochner integral 
    (see, e.g., \cite{Mikusinski1978BochnerIntegral}),
    we get that a family of $\rho^{[n]}$ with such uniform bound are equicontinuous as $C([t_1,t_2]; W^{-1,1}(\Omega))$ functions.
    We recall that $L^1((t_1,t_2) \times \Omega) = L^1 (t_1,t_2 ; L^1 (\Omega))$ and that $L^1(\Omega)$ is compactly embedded in $W^{-1,1} (\Omega)$. Thus, the result follows from the Ascoli-Arzelà theorem.
\end{proof}

\subsection{Well-posedness and interior bounds for \texorpdfstring{\eqref{eq:the problem regularised}}{(Peps)}}\label{sec:Well Posedness Pee}

We start by showing some auxiliary results. 
\begin{proof}[Proof of \Cref{lem:mobee exists}]
    First, we take $0\leq\kappa<\alpha/2$.
    Since $\Phi'_\ee \le (1+\ee)(\Phi' + \ee)$
    we observe that
    \begin{align*}
        \int_0^\alpha \left| \int_{\tfrac \alpha 2}^s \frac{\Phi_\ee'(\sigma)}{\mob(\sigma)} \chi_{[\kappa, \alpha - \kappa]} d\sigma \right| ds
         & \le (1+\ee)
        \int_0^\alpha \left| \int_{\tfrac \alpha 2}^s \frac{\Phi'(\sigma)}{\mob(\sigma)} d\sigma \right| ds \\
         & \quad + (1+\ee)
        \ee \left( \int_0^{\tfrac \alpha 2} \int_s^{\tfrac \alpha 2} \frac{\chi_{[\kappa , \alpha - \kappa]}(\sigma)}{\mob(\sigma)}  d\sigma ds + \int_{\tfrac \alpha 2}^\alpha \int_{\tfrac \alpha 2}^s \frac{\chi_{[\kappa , \alpha - \kappa]}(\sigma)}{\mob(\sigma)}  d\sigma ds \right).
    \end{align*}
    The first term is finite since $U' \in L^1 ((0,\alpha))$.
    We pick
    \[
        \kappa(\ee) = \ee \frac{\alpha}{2} + (1-\ee)\inf\left\{ \kappa \in (0,\tfrac \alpha 2] : \int_0^{\tfrac \alpha 2} \int_s^{\tfrac \alpha 2} \frac{\chi_{[\kappa , \alpha - \kappa]}(\sigma)}{\mob(\sigma)}  d\sigma ds + \int_{\tfrac \alpha 2}^\alpha \int_{\tfrac \alpha 2}^s \frac{\chi_{[\kappa , \alpha - \kappa]}(\sigma)}{\mob(\sigma)}  d\sigma ds \le \ee^{-\frac 1 2} \right\}.
    \]
    The set is non-empty, since the function inside infimum is continuous with respect to $\kappa$, and takes value $0$ at $\kappa=\frac \alpha 2$.
    Since the function inside the infimum is non-negative and non-increasing, we have that $\kappa(0^+) = 0$.

    We will use it to bound the $L^1$ norm of
    $U_\ee' (s) = U'(\tfrac \alpha 2) + \int_{\tfrac \alpha 2}^s \frac{\Phi_\ee'(\sigma)}{\mobee(\sigma)} d\sigma $.
    Thus, it suffices to estimate
    \[
        \int_0^\alpha \left| \int_{\tfrac \alpha 2}^s \frac{\Phi_\ee'(\sigma)}{\mobee(\sigma)} d\sigma \right| ds = \int_0^{\tfrac \alpha 2} \int_s^{\tfrac \alpha 2} \frac{\Phi_\ee'(\sigma)}{\mobee(\sigma)} d\sigma ds + \int_{\tfrac \alpha 2}^\alpha \int_{\tfrac \alpha 2}^s \frac{\Phi_\ee'(\sigma)}{\mobee(\sigma)} d\sigma ds.
    \]
    We split
    \[
    \frac{\Phi_\ee'(\sigma)}{\mobee(\sigma)} = \frac{\Phi_\ee'(\sigma)}{\mobee(\sigma)}\chi_{[0,\kappa(\ee))}
            + \frac{\Phi_\ee'(\sigma)}{\mobee(\sigma)}\chi_{[\kappa(\ee),\alpha - \kappa(\ee)]}
            + \frac{\Phi_\ee'(\sigma)}{\mobee(\sigma)}\chi_{( \alpha - \kappa(\ee) ,\alpha}].
            \]
            First we build a lower bound with a mobility of the form
            \[
                \underline \mob(\ee, s) =
                \begin{dcases}
                    s                                                   & \text{if } 0 \le s < \kappa(\ee),                  \\
                    \max\{ \kappa(\ee), \mob(\kappa(\ee)) \}            & \text{if } s = \kappa(\ee) ,                       \\
                    \mob(s)                                             & \text{if } \kappa(\ee) < s < \alpha - \kappa(\ee), \\
                    \max\{ \kappa(\ee), \mob (\alpha - \kappa( \ee)) \} & \text{if } s = \alpha - \kappa(\ee) ,              \\
                    \alpha - s                                          & \text{if } \alpha - \kappa(\ee) < s \le  \alpha.
                \end{dcases}
            \]
            Since we know that $\Phi'_\ee \le (1+\ee)(\kappa(\ee)^{-1} + \ee)$ and if $\mob_\ee \ge \underline \mob (\ee, \cdot)$, it follows that
            \begin{align*}
                \int_0^\alpha \left| \int_{\tfrac \alpha 2}^s \frac{\Phi_\ee'(\sigma)}{\mob_\ee (\sigma)} \chi_{[0, \kappa(\ee)]} d\sigma \right| ds
                 & \le (1+\ee)\int_0^\alpha \left| \int_{\tfrac \alpha 2}^s \frac{\kappa(\ee)^{-1} + \ee}{\underline \mob(\ee , \sigma)} \chi_{[0, \kappa(\ee)]} d\sigma \right| ds
                \\
                 &
                = (1+\ee)(\kappa(\ee)^{-1} + \ee)  \int_0^{\kappa(\ee)} \int_{s}^{\kappa(\ee)} \frac{d\sigma}{\sigma} ds
                = (1+\ee)(\kappa(\ee)^{-1} + \ee) \kappa(\ee),
            \end{align*}
            so it is bounded, and similarly for the terms $\chi_{(\alpha - \kappa(\ee), \alpha]}$.
            We also build the following upper bound
            \[
                \overline \mob(\ee,s) =
                \begin{dcases}
                    B_\ee s                                                          & \text{if } 0 \le s < \kappa(\ee),                       \\
                    \min\{ B_\ee \kappa(\ee),  (1+\ee) \mob(\kappa(\ee)) \}          & \text{if }  s = \kappa(\ee) ,                           \\
                    (1 + \ee) \mob(s)                                                & \text{if } \kappa(\ee) < s < \alpha - \kappa(\ee),      \\
                    \min\{ B_\ee \kappa(\ee),  (1+\ee) \mob(\alpha - \kappa(\ee)) \} & \text{if }  s =  \alpha - \kappa(\ee) , \\
                    B_\ee (\alpha - s)                                               & \text{if } \alpha - \kappa(\ee) < s \le  \alpha,
                \end{dcases}
            \]
            where
            \[
                B_\ee = \max\left\{ 1 + \ee , \frac{(1+\ee) \mob(\kappa(\ee))}{\kappa(\ee)}, \frac{(1 + \ee) \mob(\alpha - \kappa(\ee))}{\kappa(\ee)} \right\}.
            \]
            With this construction it is easy to see that in $(0,1] \times[0,\alpha]$ we have $\overline \mob \ge \underline \mob$. Furthermore, ${\overline \mob}$ is lower semi-continuous, and $\underline \mob$ is upper semi-continuous.
        The same holds when we divide by $s(\alpha-s)$.
        Due to the Katětov–Tong insertion theorem there exists $v_0 \in C((0,1] \times [0,\alpha])$ such that $\frac{\underline \mob}{s(\alpha-s)} \le v_0 \le \frac{\overline \mob}{s(\alpha-s)}$.

        We first build an auxiliary sequence $v_n$. Let $K_n \coloneqq [\ee_n,1] \times [0,\alpha]$ where $\ee_n = \frac 1{n+1}$ for $n \ge 0$.
        We are going to inductively construct a sequence of smooth functions $u_n$ for $n \ge 1$ such that
        \[
            u_n = u_{n+1} \text{ in } K_{n-1} \quad \text{and} \quad \underline \mob \le u_{n+1} \le \overline \mob + \ee s (\alpha-s) \text{ in } K_{n+1}.
        \]
        There exists $v_1 \in C^\infty(K_1)$ with $\|\tfrac \ee 2 + v_0 - v_1 \|_{L^\infty (K_1)} \le \frac {\ee_1} 2$, which in particular implies $\underline{\mob} (\ee, s) \leq v_1 (\ee , s) s (\alpha - s) \leq \overline{\mob} (\ee, s) + \ee s (\alpha - s)$ for all $(\ee, s) \in K_1$. We now set $u_1 = s(\alpha-s)v_1$. 
        Again, there exists $v_{n+1} \in C^\infty(K_{n+1})$ such that $\|\tfrac \ee 2 + v_0 - v_{n+1} \|_{L^\infty(K_{n+1})} \le \frac {\ee_{n+1}} 2$.
        Once more, this bound implies that $\underline{\mob} (\ee, s) \leq v_{n+1} (\ee, s) s(\alpha-s) \leq \overline{\mob} (\ee, s) + \ee s (\alpha - s)$ for all $(\ee, s) \in K_{n+1}$.

        We now construct $u_{n+1}$ from $u_n$ and $v_{n+1}$.
        Consider a non-decreasing function $\psi \in C^\infty([0,1])$ such that $\psi(0) = \psi^{(k)} (0) = \psi^{(k)}(1) = 0$ and $\psi(1) = 1$ for all $k \ge 1$.
        Let $\theta_{n+1}(\ee) = \psi (\tfrac {\ee_{n-1}-\ee}{\ee_{n-1} - \ee_{n}} )$.
        With this construction, we define
        \[
            u_{n+1} (\ee, s) = \begin{dcases}
                {u_n(\ee,s)}               & \text{in } K_{n-1} ,                 \\
                {u_n(\ee,s)} ( 1 - \theta_{n+1}(\ee) ) + v_{n+1}(\ee,s) s(\alpha-s) \theta_{n+1}(\ee)
                                           & \text{in } K_{n} \setminus K_{n-1} , \\
                v_{n+1}(\ee,s) s(\alpha-s) & \text{in } K_{n+1} \setminus K_n,
            \end{dcases}
        \]
        which also satisfies the desired bounds. Since we are pasting $C^\infty$ functions with $C^\infty$ contact conditions, we have $u_{n+1} \in C^\infty (K_{n+1})$.

        Let us now define the limit function $u \in C^\infty((0,1] \times[0,\alpha])$ so that $u = u_n$ in $K_{n-1}$ for all $n \ge 1$. Notice that  $u(\ee,0) = u(\ee,\alpha) = 0$ for $\ee > 0$. Furthermore, due to the lower bound we get $\frac{\partial u}{\partial s} (\ee, 0), -\frac{\partial u}{\partial s} (\ee, \alpha) \ge 1$ for any $\ee > 0$.
        Since $\underline \mob(0,s) = \overline \mob(0,s) = \mob(s)$ we have an extension $u \in C([0,1] \times [0,\alpha])$ where $u(0,s) = \mob(s)$.
        Finally, we define $\mob_\ee(s) = u (\ee,s)$. Since $\mob_\ee \ge \underline \mob (\ee,\cdot)$ the corresponding $U_\ee$ satisfies the desired properties.
\end{proof}

\begin{lemma}
    \label{lem:hypothesis Phi/Phi' is uniform}
    If \eqref{eq:hypothesis Phi/Phi' ee = 0} is satisfied, then it follows that
    \begin{align}
        \label{eq:hypothesis Phi/Phi'}
        \sup_{(\ee,s) \in (0,1] \times [0,\alpha] } \left|\frac{(\Phi_\ee (\alpha)-\Phi_\ee (s))\Phi_{\ee}(s)}{\Phi_{\ee}'(s)} \right| + \left|\frac{(\Phi_\ee (s)-\Phi_\ee (0))\Phi_{\ee}(s)}{\Phi_{\ee}'(s)} \right| & <\infty .
    \end{align}
\end{lemma}
\begin{proof}
    We work first on the term $\Phi_{\ee}(s)- \Phi_{\ee}(0)$. Using \eqref{eq:properties of Phi_ee}, we point out that
    \begin{align*}
        \Phi_\ee(s) = \int_{s_0}^s \Phi_\ee'(\sigma) d\sigma , & \quad \Phi_\ee(s) - \Phi_\ee(0) = \int_0^s \Phi_\ee'(\sigma) d\sigma \ge 0,
        \quad
        |\underline {\Phi_\ee} (s)| \le |\Phi_\ee(s)| \le (1+\ee) |\underline {\Phi_\ee} (s)| ,                                              \\
                                                               & \text{and }
        \quad \underline{\Phi_\ee}(s) - \underline{\Phi_\ee}(0) \leq \Phi_\ee(s) - \Phi_\ee(0) \leq (1+\ee) (\underline{\Phi_\ee}(s) - \underline{\Phi_\ee}(0)).
    \end{align*}
    Therefore, we have the estimates
    \[
        0 \leq \frac{\underline{\Phi_\ee}(s) - \underline{\Phi_\ee}(0)}{(1+\ee) \underline{\Phi_\ee'}(s)}
        |\underline \Phi_\ee(s)|
        \le
        \frac{\Phi_\ee(s) - \Phi_\ee(0)}{\Phi_\ee'(s)} |\Phi_\ee(s) |
        \le (1 + \ee)^2 \frac{ \underline{\Phi_\ee}(s) - \underline{\Phi_\ee}(0)}{\underline{\Phi_\ee'}(s)} |\underline \Phi_\ee(s)|.
    \]
    We distinguish two cases:

    \textit{Case I: $s$ such that $\Phi'(s) \le \kappa(\ee)^{-1}$.} In this case, we have $\underline {\Phi_\ee}'(s) = \Phi'(s) + \varepsilon$, and we point out that
    \[
        \underline {\Phi_\ee} (s) - \underline {\Phi_\ee} (0) = \int_0^s \underline {\Phi_\ee}'(\sigma) d\sigma \le \int_0^s (\Phi'(s) + \ee) ds =  \Phi(s) - \Phi(0) + \ee s.
    \]
    Therefore, we can estimate
    \begin{align*}
        \frac{\underline {\Phi_\ee} (s) - \underline {\Phi_\ee} (0)}{\underline {\Phi_\ee'} (s)}
         & \leq
        \frac{\Phi(s) - \Phi(0) + \ee s}{\Phi'(s) + \ee}
        \leq
        \frac{\Phi(s) - \Phi(0) }{\Phi'(s)
            + \varepsilon
        } + s,
    \end{align*}
    and therefore
    \begin{align*}
        \frac{\underline {\Phi_\ee} (s) - \underline {\Phi_\ee} (0)}{\underline {\Phi_\ee'} (s)} |\underline{\Phi_\ee}(s)| & \le
        \left( \frac{\Phi(s) - \Phi(0)}{\Phi'(s)  + \varepsilon} + s \right) (|\Phi(s)| + \ee)  < \infty.
    \end{align*}

    \textit{Case II: $s$ such that $\Phi'(s) > \kappa(\ee)^{-1}$.}
    In this case we have that $\underline{\Phi'_\ee}(s) = \kappa(\ee)^{-1} + \ee \ge \kappa(\ee)^{-1}$, and hence
    \begin{equation*}
        \frac{\underline {\Phi_\ee} (s) - \underline {\Phi_\ee} (0)}{\underline {\Phi_\ee'} (s)} \le \kappa(\ee)
        \left({\Phi(s) - \Phi(0) + \ee s} \right).
    \end{equation*}
    In particular, we recover that
    \begin{equation*}
        \frac{\underline {\Phi_\ee} (s) - \underline {\Phi_\ee} (0)}{\underline {\Phi_\ee'} (s)} |\underline{\Phi_\ee}(s)| \leq \kappa(\ee)
        \left({\Phi(s) - \Phi(0) + \ee s} \right)(|\Phi(s)| + \ee s) < \infty.
    \end{equation*}
    We can proceed similarly for $\Phi(\alpha) - \Phi(s)$.
\end{proof}

From \eqref{eq:properties of Phi_ee}, our problem has uniformly elliptic diffusion, and from our set of basic assumptions, $V$ is regular enough. Then, the problem \eqref{eq:the problem regularised} is uniformly parabolic and existence, uniqueness, and the maximum principle hold from the classical theory \cite{LSU68, GT01, Ama90, Yin05}, further details can also be seen in \cite{Carrillo_Gomez-Castro_Vazquez22}. Furthermore, from \cite{Ama90}, we know that classical solutions to \eqref{eq:the problem regularised} are as regular as we want depending on the regularity of the initial data $\rho_0$. This corresponds to \ref{item:(Pee) classical solutions}.

Afterwards, we continue by proving an interior bounds result for the classical solutions.

\begin{lemma}[Strong Maximum Principle]\label{lem:Strong Max Principle Classical Solutions}
    Assume \eqref{hyp:Sublinear saturation} and $\rho_0 \in \mathcal{A} \cap C^2 (\overline{\Omega})$. Then, the unique classical solution of \eqref{eq:the problem regularised} satisfies either $\rho^{(\ee)} \equiv 0$, $\rho^{(\ee)} \equiv \alpha$, or $0 < \rho_t^{(\ee)} < \alpha$ in $\Omega$ for all $t> 0$.

\end{lemma}

\begin{proof}%[Proof of \Cref{item:(Pee) Strong Maximum Principle}-\Cref{thm:Properties (Pee)}]
    Assume $\rho^{(\ee)}$ is not constantly $0$ or $\alpha$.
    Let us freeze the coefficients to study strict positivity.
    Consider the extension
    $$
        c(s) = \begin{cases}
            \frac{\mob_\ee (s)}{s} & \text{if } s \in (0,\alpha], \\
            \mob_\ee'(0)           & \text{if } s = 0.
        \end{cases}
    $$
    Then $u = \rho^{(\ee)}$ is a solution of the following problem
    \begin{equation*}
        \frac{\partial u}{\partial t} = \Phi_{\ee}'(\rho^{(\ee)}) \Delta u + \left( \Phi_{\ee}'' (\rho^{(\ee)}) \nabla \rho^{(\ee)} + \mobee' (\rho^{(\ee)}) \nabla V \right) \cdot \nabla u  + c ( \rho^{(\ee)} )  (\Delta V ) u \quad  \text{in }  (0, \infty ) \times \Omega.
    \end{equation*}
    Since $\rho^{(\ee)}$ is a classical solution, all the coefficients are continuous. Furthermore, from \eqref{hyp:Sublinear saturation} it follows that $c(\rho^{(\ee)}) \Delta V $ is bounded.
    With the change of variable $v = e^{-\lambda t} u$ for any $\lambda \in \R$ we can replace the zero order coefficient by a non-negative quantity.
    Thus, we can apply the Strong Maximum Principle \cite[Section 7 - Theorem 12]{Evans10} to deduce that the solution $u = \rho^{(\ee)}$ is strictly positive for all $t>0$.

    Analogously, $u = \alpha - \rho^{(\ee)} $ solves the problem,
    \begin{equation*}
        \begin{dcases}
            \frac{\partial u}{\partial t} = \diver \left( \Phi_{\varepsilon}'(\rho^{(\ee)}) \nabla u \right) + \mobee'(\rho^{(\ee)} ) \nabla u \cdot \nabla V  - \frac{\mobee (\rho^{(\ee)})}{\alpha - \rho^{(\ee)} }  (\Delta V) u    \qquad & \text{in }  (0, \infty ) \times \Omega,                   \\
            u = \alpha - \rho^{(\ee)} \geq 0 \qquad                                                                                                                                                                                           & \text{on }  (0, \infty ) \times \partial \Omega,          \\
            u = \alpha - \rho_0 \geq 0 \qquad                                                                                                                                                                                                 & \text{on }  \left\lbrace t=0 \right\rbrace \times \Omega.
        \end{dcases}
    \end{equation*}
    Since $0 \leq \frac{\mobee (\rho^{(\ee)})}{\alpha -\rho^{(\ee)}} \leq \varepsilon^{-1}$ due to \eqref{hyp:Sublinear saturation}, once again, from the Strong Maximum Principle  it follows that $u > 0$ in $(0, \infty) \times \Omega$ and, in particular, $\rho^{(\ee)} < \alpha$ in $\Omega$ for all $t >0$.
\end{proof}

\subsection{A priori estimates for \texorpdfstring{\eqref{eq:the problem regularised}}{(Peps)}}\label{sec:a priori}

Let us first comment on some properties of  $\Phi_\ee$ and $\rho^{(\ee)}$. We point out that for every $0 \leq t_1 < t_2 < \infty$ and $p \in [1, \infty]$, the unique classical solution of \eqref{eq:the problem regularised} is such that,
\begin{equation}\label{eq:Lp estimate}
    \| \rho^{(\ee)} \|_{L^p((t_1,t_2) \times \Omega )} \leq C(p, t_1, t_2) \coloneqq \alpha (t_2-t_1)^{\frac{1}{p}} |\Omega|^{\frac{1}{p}}.
\end{equation}

\subsubsection{Spatial regularity}

In order to obtain an \textit{a priori} estimate on $\nabla \rho^{(\ee)}$ we need to define the following auxiliary flow,
\begin{equation}\label{eq:G_ee}
    G_{\ee}''(s) = \frac{1}{\Phi_{\varepsilon}' (s)}, \quad G_{\ee}' \left(0 \right) =G_{\ee} \left( 0 \right) = 0.
\end{equation}
\begin{lemma}[\textit{A priori} estimates on $\nabla \rho^{(\ee)}$]
    Assume \eqref{hyp:Sublinear saturation} and let $0 < t_1 < t_2 < \infty$. Then, the unique classical solution $\rho^{(\ee )}$ of \eqref{eq:the problem regularised} is such that
    \begin{equation}\label{eq:nabla rho a priori}
        \| \nabla \rho^{(\ee)} \|^2_{L^2((t_1,t_2) \times \Omega)} \leq 2 \int_{\Omega} G_{\ee}( \rho_{t_1}^{(\ee)})  + (t_2-t_1)  \left\|\frac{{ \mobee (\rho^{(\ee)}) }}{\Phi_{\ee}' (\rho^{(\ee)}) } \right\|_{L^{\infty} (\Omega)}  \| \nabla V \|_{L^2(\Omega)}^2.
    \end{equation}
\end{lemma}
\begin{proof}
    We compute in order to obtain that 
    \begin{align*}
        \frac{\partial }{\partial t} \int_{\Omega} G_{\ee}(\rho^{(\ee)}) & = - \int_{\Omega} \nabla G_{\ee}'(\rho^{(\ee)}) \cdot \left( \nabla  \Phi_{\varepsilon} (\rho^{(\ee)}) + \mobee (\rho^{(\ee)}) \nabla V \right) \\
                                                                         & = - \int_{\Omega} | \nabla \rho^{(\ee)} |^2 - \int_{\Omega} \nabla \rho^{(\ee)} G_{\ee}'' (\rho^{(\ee)})  \mobee (\rho^{(\ee)}) \cdot \nabla V.
    \end{align*}
    Therefore, applying Young's inequality and integrating in time from $t_1$ to $t_2$ we get,
    \begin{equation*}
        \int_{t_1}^{t_2}\| \nabla \rho^{(\ee)} \|^2_{L^2( \Omega)} \leq \int_{\Omega} G_{\ee}( \rho_{t_1}^{(\ee)}) -\int_{\Omega} G_{\ee}( \rho_{t_2}^{(\ee)}) + \frac{1}{2} \int_{t_1}^{t_2} \| \nabla \rho^{(\ee)} \|^2_{L^2( \Omega)} + \frac{1}{2} \int_{t_1}^{t_2}  \left\|\frac{{ \mobee (\rho^{(\ee)}) }}{\Phi_{\ee}' (\rho^{(\ee)}) } \right\|_{L^{\infty} (\Omega)}  \| \nabla V \|_{L^2(\Omega)}^2.
    \end{equation*}
    Since $G_\ee (s) \geq 0$ for all $s \geq 0$ we recover \eqref{eq:nabla rho a priori}.
\end{proof}
Finally, we would also like to obtain an \textit{a priori} estimate on $\nabla \Phi_{\varepsilon} (\rho^{(\ee)} )$.
\begin{lemma}[\textit{a priori} estimates on $\Phi_{\varepsilon} (\rho^{(\ee)})$]\label{lem:nabla phi(rho) a priori}
    Assume \eqref{hyp:Sublinear saturation}, $\ee \in (0, 1]$ and let $0 < t_1 < t_2 < \infty$. Then, the unique classical solution $\rho^{(\ee )}$ of \eqref{eq:the problem regularised} is such that
    \begin{equation}\label{eq:nabla phi(rho) a priori}
        \| \nabla \Phi_{\varepsilon} (\rho^{(\ee)}) \|_{L^2((t_1,t_2) \times \Omega)}^2 \leq   4 \alpha |\Omega| \|\Phi_\ee\|_{L^\infty (0,\alpha)} + (t_2-t_1)\| \mobee \|_{L^{\infty}(0,\alpha)} \| \nabla V \|_{L^2(\Omega)}^2.
    \end{equation}
\end{lemma}
\begin{proof}
    We define
    $
        \vartheta_{\ee}(s) \coloneqq \int_0^s \Phi_{\varepsilon} (\sigma) \, d \sigma.
    $
    We can compute that,
    \begin{equation*}
        \frac{\partial}{\partial t} \int_{\Omega} \vartheta_{\ee} (\rho^{(\ee)} ) = - \int_{\Omega} | \nabla \Phi_{\varepsilon} (\rho^{(\ee)}) |^2 - \int_{\Omega} \nabla \Phi_{\varepsilon} (\rho^{(\ee)}) \cdot \mobee (\rho^{(\ee)} )  \nabla V.
    \end{equation*}
    Therefore, applying Young's inequality and integrating in time from $t_1$ to $t_2$ we get that
    \begin{align*}
        \int_{t_1}^{t_2} \| \nabla \Phi_{\varepsilon} (\rho^{(\ee)}) \|_{L^2( \Omega)}^2 & \leq \int_{\Omega} \vartheta_{\ee}( \rho_{t_1}^{(\ee)}) -\int_{\Omega} \vartheta_{\ee}( \rho_{t_2}^{(\ee)}) + \frac{1}{2} \int_{t_1}^{t_2} \| \nabla \Phi_{\varepsilon} (\rho^{(\ee)}) \|_{L^2( \Omega)}^2 \\
                                                                                         & \quad + \frac{1}{2} \int_{t_1}^{t_2} \left\| \mobee (\rho^{(\ee)}) \right\|_{L^{\infty} ( \Omega)} \| \nabla V \|_{L^2(\Omega)}^2.
    \end{align*}
    Finally, we point out that $| \vartheta_\ee (\rho) | \le \alpha  \| \Phi_\ee \|_{L^\infty (0,\alpha)}$ for any $\rho \in \mathcal{A}_+$. Hence, we recover the desired inequality.
\end{proof}

\subsubsection{Free energy and its dissipation}\label{sec:compactness in time}
In order to obtain an \textit{a priori} estimate on the time derivative and be able to discuss the long time asymptotic behaviour of the problem, we need to take advantage of the gradient flow structure of the problem.

\begin{proposition}\label{prop:Convergence of the free_energies}
    Assume \eqref{hyp:Sublinear saturation} and let $\rho_0 \in \mathcal A_+ \cap C^2(\overline \Omega)  $. Then, the unique classical solution $\rho^{(\ee)}$ of the problem \eqref{eq:the problem regularised}  is such that,
    \begin{equation}\label{eq:flow of the free energy}
        \int_{t_1}^{t_2} \int_{\Omega} \mobee (\rho^{(\ee)}) \left| \nabla \left( U'_{\varepsilon}(\rho^{(\ee)}) + V  \right) \right|^2 \, dx \, dt = \mathcal{F}_{\varepsilon} [\rho_{t_1}^{(\ee)}] - \mathcal{F}_{\varepsilon} [\rho_{t_2}^{(\ee)}] \qquad \text{for all } 0 \le t_1 < t_2.
    \end{equation}
    In particular, $\mathcal{F}_{\varepsilon} [\rho_t^{(\ee)}]$ is a non-increasing sequence  in $t$.
\end{proposition}

\begin{proof}
    Combining the sufficient regularity of $\rho^{(\ee)}$,   the regularity of $U_\ee$ in $(0,\alpha)$ and \eqref{eq:initial datum A+ bounds}, we can rigorously take the derivative
    \begin{equation*}
        \frac{d}{dt} \mathcal{F}_{\varepsilon } [\rho^{(\ee)}] = - \int_{\Omega} \mobee (\rho^{(\ee)} )  \left| \nabla \left( U_{\varepsilon }' (\rho^{(\ee)} ) + V) \right) \right|^2,
    \end{equation*}
    Integrating in $(t_1, t_2)$ yields the result.
\end{proof}

For the set $\mathcal A$, we can prove the following auxiliary result about the free energy.
\begin{lemma}
    \label{prop:Free energy bound from below}
    Let $\ee \ge 0$.
    Then $\mathcal{F}_\ee$ has a global minimiser in $\mathcal A$ and is lower semicontinuous in $\mathcal{A}$ with respect to the weak $L^1$-topology.
\end{lemma}

\begin{proof}
    Using that $0 \le \rho \le \alpha$, the positivity of $V$, and the continuity of $U_\ee$, it is trivial to see that
    $\mathcal F$ is bounded from  below.

    Let us prove the continuity claim. The classical theorems for lower semi-continuity are written for functions $U : \mathbb R \to \mathbb R$ (e.g., \cite{GoffmanSerrin1964SublinearFunctionsMeasures}).
    However, the result in our setting has a very simple proof.
    First, we prove strong $L^1$ continuity.
    If $u_n \to u$ strongly in $L^1(\Omega)$, there exists a subsequence, $u_{n_k}$, converging a.e.~in $\Omega$.
    Since $U\in C([0,\alpha])$, $V$ is continuous and $0\le u_{n_k} \le \alpha$, by the Dominated Convergence Theorem we get $\mathcal F_\ee[u_{n_k}] \to \mathcal F_\ee[u]$. Since every sequence has a convergent subsequence, and they all share a limit, $\mathcal F_\ee[u_n] \to \mathcal F_\ee[u]$. The strong $L^1$ continuity is proven.

    We now prove weak-$L^1$ lower-semicontinuity. Let $u_n \rightharpoonup u$ weakly in $L^1(\Omega)$
    and $L = \liminf_n \mathcal F_\ee[u_n]$. Notice that $L$ is finite since $\mathcal F$ is bounded below.
    Pick a subsequence $v_i = u_{n_i}$ such that $\mathcal F_\ee[v_i] \to L$.
    Due to Mazur's lemma, there exist convex combinations
    \[
        y_j = \sum_{i=j}^{N_j} \theta_i^{(j)} v_i, \quad \text{where } \theta_i^{(j)} \in [0,1] \text{ and } \sum_{i=j}^{N_j} \theta_i^{(j)} = 1,
    \]
    such that $y_j \to u$ strongly in $L^1 (\Omega)$.
    Let $\delta > 0$. There exists $\mathcal{I}(\delta)$ such that for $i \ge \mathcal{I}(\delta)$ we get $\mathcal F_\ee[v_i] \le \delta + L$.
    Since $U_\ee$ is convex, so is $\mathcal F_\ee$, and for $j \ge \mathcal{I}(\delta)$
    \[
        \mathcal F_\ee[y_j] \le \sum_{i=j}^{N_j} \theta_i^{(j)} \mathcal F_\ee[v_i] \le \sum_{i=j}^{N_j} \theta_i^{(j)} \left( \delta + L \right) = \delta + L.
    \]
    Using the strong-$L^1$ lower semi-continuity and this estimate we deduce that $\mathcal F[u] \le L + \delta$ for any $\delta > 0$. Hence, we deduce that $\mathcal F[u] \le L$.

    Lastly, we prove the existence of a global minimiser. Taking a minimising sequence $\rho_n \in \mathcal A$, we have that it is bounded in $L^\infty(\Omega)$, so up to a subsequence, $\rho_{n_k} \rightharpoonup \rho^\star$ weakly-$\star$ in $L^\infty(\Omega)$.
    Due to the lower semi-continuous in the weak-$L^1$ topology  we notice that the infimum is achieved at $\rho^\star$.
\end{proof}

\subsubsection{Stability}
We devote this subsection to show some stability results that will appear several times through the manuscript.
\begin{lemma}[$L^1$-stability of weak solutions]
    \label{lem:L1 stability of weak solutions}
    Let $\mob^{\{k\}} \to \mob$ and $\Phi^{\{k\}} \to \Phi$ in $C([0,\alpha])$, $\rho^{\{k\}} \to \rho$ in $L^1((0,T) \times \Omega)$
    and assume
    $\Phi^{\{k\}}(\rho^{\{k\}})$ are uniformly bounded in $L^2(0,T; H^1(\Omega))$.
    If $\rho^{\{k\}}$ are weak solutions of \eqref{eq:the problem Omega} with $\mob^{\{k\}}$ and $\Phi^{\{k\}}$, then $\rho$ is a solution of \eqref{eq:the problem Omega} with $\mob$ and $\Phi$.
\end{lemma}
\begin{proof}
    We take a subsequence that converges a.e.~in $(0,T) \times \Omega$.
    Since $\Phi^{\{k\}}$ converges uniformly and $0 \le \rho^{\{k\}} \le \alpha$, by the Dominated Convergence Theorem we recover
    \begin{equation*}
        \Phi^{\{k\}} (\rho^{\{k\}} ) \rightarrow  \Phi (\rho ) \text{ in }L^p ( (0,T) \times  \Omega) \text{ for all } p \in [1, \infty) .
    \end{equation*}
    Similarly, $m^{\{k\}} (\rho^{\{k\}}) \to m(\rho)$ strongly in $L^p((0,T) \times \Omega)$.
    Due to the boundedness in $L^2 (0,T; H^1(\Omega))$,
    up to a further subsequence
    \begin{equation*}
        \Phi^{\{k\}}(\rho^{\{k\}}) \rightharpoonup  \Phi(\rho ) \quad \text{weakly in } L^2(0,T; H^1(\Omega) ).
    \end{equation*}
    This fact and the $L^1(\Omega)$ convergence for
    $\rho^{\{k\}}$ are sufficient to pass to the limit in the weak formulation.
\end{proof}

Using standard arguments of Calculus of Variations (see e.g., \cite[Chapter 2]{Rindler18}) we prove suitable convergence of the free energy dissipation terms for the approximating problems.
\begin{lemma}[Stability of the dissipation term]
    \label{lem:free energy dissipation LHS is lsc}
    Assume that $\Phi^{\{k\}}, \Phi \in C^1 ((0,\alpha))$, $\mob^{\{k\}} \in C([0,\alpha])$. 
    Let $\mob^{\{k\}} \to \mob$ and $\Phi^{\{k\}} \to \Phi$ in $C([0,\alpha])$ and $\rho^{\{k\}} \to \rho$ in $L^1((t_1,t_2) \times \Omega)$.
    Let us consider the associated $U^{ \{k\} }, U$ (i.e., $(U^{\{k\}})'' = (\Phi^{\{k\}})' / \mob^{\{k\}}$, $U'' = \Phi' / \mob$ with prescribed values of $U^{\{k\}} (\frac\alpha 2) = U(\frac\alpha 2) $ and $(U^{\{k\}})' (\frac\alpha 2) = U'(\frac\alpha 2) $) and the energy dissipation terms
    \begin{equation*}
        \mathcal D [\rho] \coloneqq \int_{t_1}^{t_2} \int_\Omega \mob(\rho) | \nabla (U'(\rho) + V) |^2, \qquad \mathcal D_k [\rho] \coloneqq \int_{t_1}^{t_2} \int_\Omega \mob^{\{k\}} (\rho) | \nabla ((U^{\{k\}})'(\rho) + V) |^2 .
    \end{equation*}
    Then,
    \[
        \mathcal D [\rho] \le \liminf_{k \to \infty} \mathcal D_k[\rho^{\{k\}}].
    \]
\end{lemma}

\begin{proof}
    We assume that this $\liminf$ is finite, otherwise there is nothing to prove. There exists a subsequence $k_j \to \infty$ such that
    \[
        \lim_{j\to \infty} \mathcal D_{k_j}[\rho^{\{k_j\}}] = \liminf_{k \to \infty} \mathcal D_{k} [\rho^{\{k\}}].
    \]
    We pick a subsequence, not relabelled, with convergence a.e.~in $(t_1,t_2) \times \Omega$.
    Because of uniform boundedness, we can use the Dominated Convergence Theorem to prove $L^p$ convergence of $\mob^{\{k_j\}}(\rho^{\{k_j\}})$ and $\Phi^{\{k_j\}}(\rho^{\{k_j\}})$ for any $p < \infty$.
    Define
    \[
        \xi_j \coloneqq \frac{1}{\mob^{\{k_j\}}(\rho^{\{k_j\}})^{\frac 1 2}} \nabla \Phi^{\{k_j\}}(\rho^{\{k_j\}}) + \mob^{\{k_j\}}(\rho^{\{k_j\}})^{\frac 1 2} \nabla V.
    \]
    This is an $L^2 ((t_1,t_2) \times \Omega)^d$ bounded sequence. Thus, from the Banach-Alaoglu Theorem, it has a weakly convergent subsequence (which we will not relabel). Let $\xi \in L^2((t_1,t_2) \times \Omega)^d$ be its limit.

    We know strong $L^1$ convergence of all terms except $\nabla \Phi^{\{k_j\}}(\rho^{\{k_j\}})$. But we can write the weak $L^p$ convergence for $p \in (1,2)$
    \[
        \nabla \Phi^{\{k_j\}}(\rho^{\{k_j\}}) = \mob^{\{k_j\}}(\rho^{\{k_j\}})^{\frac 1 2}\Big(\xi_j - \mob^{\{k_j\}}(\rho^{\{k_j\}})^{\frac 1 2} \nabla V\Big) \rightharpoonup \mob(\rho)^{\frac{1}{2}}\Big(\xi - \mob(\rho)^{\frac 1 2} \nabla V\Big)
    \]
    Thus $\Phi^{\{k_j\}}(\rho^{\{k_j\}})$ converges weakly in $L^p(t_1,t_2; W^{1,p} (\Omega))$ for $p \in [1,2)$ and we have
    \[
        \nabla \Phi(\rho) = \mob(\rho)^{\frac 1 2} \Big(\xi - \mob(\rho)^{\frac 1 2} \nabla V\Big) .
    \]
    Now we can apply Mazur's lemma similarly to the proof of \Cref{prop:Free energy bound from below} using the convexity of the map $\xi \mapsto |\xi|^2$.
\end{proof}

\subsubsection{$L^1$ contraction and continuous dependence}\label{sec:L1 contraction and continuous dependece}
We study the $L^1$ contraction and continuous dependence result for classical solutions of the regularised problem \eqref{eq:the problem regularised}.

\newcommand{\onesolution}{\rho^{(\varepsilon)}}
\newcommand{\othersolution}{\eta^{(\varepsilon)}}
\begin{lemma}\label{lem:L1 contraction and continuous dependece}
    Let $\ee > 0$.
    Consider $\onesolution$ and $\othersolution$ two classical solutions of the problem \eqref{eq:the problem regularised} corresponding to the initial data $\rho_0^{(\ee)}$ and $\eta_0^{(\ee)}$ respectively. Then, it follows that for all $t>0$ the problem satisfies the $L^1$-contraction estimate
    \begin{equation}\label{eq:Pee L1 contraction}
        \| \onesolution_t - \othersolution_t \|_{L^1(\Omega)} \leq \| \onesolution_0 - \othersolution_0 \|_{L^1(\Omega)},
    \end{equation}
    and the $L^1$-comparison principle
    \begin{equation}\label{eq:Pee L1 continuous dependence}
        \int_\Omega ( \onesolution_t - \othersolution_t )^+ \leq \int_\Omega ( \onesolution_0 - \othersolution_0 )^+.
    \end{equation}
\end{lemma}

\begin{proof} We divide the proof in several steps.
    \begin{enumeratesteps}
        \step[Auxiliary function]
        Let us fix $\delta > 0$. We choose an auxiliary function $j_\delta$ such that $j_\delta''(s) = \frac{1}{2 \delta}\chi_{(-\delta, \delta)}(s)$ and $j_\delta (s) \to |s|$ as $\delta \rightarrow 0$. We also define $\psi_\delta$ such that $\psi_\delta'(s) = s j_\delta '' (s) = \frac{s}{2 \delta} \chi_{(-\delta, \delta)}(s)$ and $\psi_\delta (0)=0$. Thus, $|\psi_\delta'(s)|\leq \frac{1}{2}\chi_{(-\delta, \delta)}(s)$, and $|\psi_\delta(s)|\leq \delta$. Hence,
        \begin{align*}
            \frac{\partial}{\partial t} \int_\Omega j_\delta (\onesolution_\tau - \othersolution_\tau) & = \int_\Omega j_\delta' (\onesolution_\tau - \othersolution_\tau) \diver \left(  \nabla \left( \Phi_\ee (\onesolution_\tau) -   \Phi_\ee (\othersolution_\tau) \right) \right)                        \\
                                                                                                       & \quad + \int_\Omega j_\delta' (\onesolution_\tau - \othersolution_\tau) \diver \left( (\mobee(\onesolution_\tau) - \mobee(\othersolution_\tau)) \nabla V \right) \eqqcolon \mathcal{I} + \mathcal{J}.
        \end{align*}

        \step[Diffusion]
        Let us start studying the diffusive terms. We compute to obtain,
        \begin{align*}
            \mathcal I & = - \int_\Omega  j_\delta'' (\onesolution_\tau - \othersolution_\tau) \nabla  (\onesolution_\tau - \othersolution_\tau) \cdot \nabla \left( \Phi_\ee (\onesolution_\tau) -   \Phi_\ee (\othersolution_\tau) \right)                                                              \\
                       & = -\int_\Omega j_\delta'' (\onesolution_\tau - \othersolution_\tau) \Phi_\ee'(\onesolution_\tau) | \nabla (\onesolution_\tau - \othersolution_\tau) |^2                                                                                                                          \\
                       & \quad - \int_\Omega j_\delta'' (\onesolution_\tau - \othersolution_\tau) \left( \Phi_\ee' (\onesolution_\tau) -   \Phi_\ee' (\othersolution_\tau) \right) \nabla \othersolution_\tau \cdot \nabla (\onesolution_\tau - \othersolution_\tau)                                      \\
                       & \leq - \int_\Omega \frac{\Phi_\ee' (\onesolution_\tau) -   \Phi_\ee' (\othersolution_\tau)}{\onesolution_\tau - \othersolution_\tau} \nabla \psi_\delta (\onesolution_\tau - \othersolution_\tau) \cdot \nabla \othersolution_\tau                                               \\
                       & = - \int_\Omega \left(\frac{\Phi_\ee' (\onesolution_\tau) -   \Phi_\ee' (\othersolution_\tau)}{\onesolution_\tau - \othersolution_\tau} - \Phi_\ee'' (\othersolution_\tau) \right) \nabla \psi_\delta (\onesolution_\tau - \othersolution_\tau) \cdot \nabla \othersolution_\tau \\
                       & \quad - \int_\Omega \Phi_\ee'' (\othersolution_\tau) \nabla \psi_\delta (\onesolution_\tau - \othersolution_\tau) \cdot \nabla \othersolution_\tau \eqqcolon \mathcal I_1 + \mathcal I_2.
        \end{align*}
        In order to study $\mathcal I_1$ we take advantage of the following remark. Using a Taylor expansion we have that
        \begin{equation*}
            \left| \frac{\Phi_\ee'(\onesolution_\tau) - \Phi_\ee'(\othersolution_\tau)}{\onesolution_\tau - \othersolution_\tau} - \Phi_\ee''(\othersolution_\tau) \right| \leq \frac{1}{2} \| \Phi_\ee''' \|_{L^{\infty}(0, \alpha)} | \onesolution_\tau - \othersolution_\tau |.
        \end{equation*}
        Therefore,
        \begin{align*}
            \mathcal I_1 & \leq \frac{1}{2} \|\Phi_\ee''' \|_{L^{\infty}(0, \alpha)} \int_\Omega | \onesolution_\tau - \othersolution_\tau | |\psi_\delta' (\onesolution_\tau - \othersolution_\tau) \nabla (\onesolution_\tau - \othersolution_\tau) \cdot \nabla \othersolution_\tau |                                              \\
                         & = \frac{1}{2} \|\Phi_\ee''' \|_{L^{\infty}(0, \alpha)} \int_{| \onesolution_\tau - \othersolution_\tau | < \delta} | \onesolution_\tau - \othersolution_\tau | |\psi_\delta' (\onesolution_\tau - \othersolution_\tau) \nabla (\onesolution_\tau - \othersolution_\tau) \cdot \nabla \othersolution_\tau | \\
                         & \leq \frac{\delta}{2} \|\Phi_\ee''' \|_{L^{\infty}(0, \alpha)}  \|  \nabla (\onesolution_\tau - \othersolution_\tau) \|_{L^2(\Omega)} \| \nabla \othersolution_\tau \|_{L^2(\Omega)}.
        \end{align*}
        Next, we use properties of $\psi_\delta$ in order to deal with $\mathcal I_2$. We integrate by parts in order to obtain that
        \begin{align*}
            \mathcal I_2 & = \int_\Omega \psi_\delta (\onesolution_\tau - \othersolution_\tau) \diver \left( \Phi_\ee'' (\othersolution_\tau) \nabla \othersolution_\tau \right) - \int_{\partial \Omega} \psi_\delta (\onesolution_\tau - \othersolution_\tau) \Phi_\ee'' (\othersolution_\tau) \nabla \othersolution_\tau \cdot {n} ,
        \end{align*}
        and $|\mathcal I_2| \leq C \delta$.

        \step[Drift]
        Let us focus now on the drift term. Integrating by parts we obtain that
        \begin{align*}
            \mathcal J & = - \int_\Omega j_\delta''(\onesolution_\tau - \othersolution_\tau) \nabla (\onesolution_\tau - \othersolution_\tau) (\mobee(\onesolution_\tau) - \mobee(\othersolution_\tau)) \cdot \nabla V                                                  \\
                       & = - \int_\Omega \frac{\mobee(\onesolution_\tau) - \mobee(\othersolution_\tau)}{\onesolution_\tau - \othersolution_\tau} \nabla \psi_\delta (\onesolution_\tau - \othersolution_\tau) \cdot \nabla V                                            \\
                       & = - \int_\Omega \left(\frac{\mobee(\onesolution_\tau) - \mobee(\othersolution_\tau)}{\onesolution_\tau - \othersolution_\tau} - \mobee'(\onesolution_\tau) \right) \nabla \psi_\delta (\onesolution_\tau - \othersolution_\tau) \cdot \nabla V \\
                       & \quad - \int_\Omega \mobee'(\onesolution_\tau) \nabla \psi_\delta (\onesolution_\tau - \othersolution_\tau) \cdot \nabla V \eqqcolon \mathcal J_1 + \mathcal J_2.
        \end{align*}
        We study $\mathcal J_1$ in the same way we have studied $\mathcal I_1$. Using a Taylor expansion we have that
        \begin{equation*}
            \left| \frac{\mobee(\onesolution_\tau) - \mobee(\othersolution_\tau)}{\onesolution_\tau - \othersolution_\tau} - \mobee'(\onesolution_\tau) \right| \leq \frac{1}{2} \| \mobee'' \|_{L^{\infty}} | \onesolution_\tau - \othersolution_\tau |.
        \end{equation*}
        Hence, we can estimate
        \begin{align*}
            |\mathcal J_1| & \leq \frac{1}{2} \| \mobee'' \|_{L^{\infty}} \int_{\Omega} | \onesolution_\tau - \othersolution_\tau | \left| \psi_\delta'( \onesolution_\tau - \othersolution_\tau ) \nabla ( \onesolution_\tau - \othersolution_\tau) \cdot \nabla V \right|                                            \\
                           & = \frac{1}{2} \| \mobee'' \|_{L^{\infty}} \int_{| \onesolution_\tau - \othersolution_\tau | < \delta} | \onesolution_\tau - \othersolution_\tau | \left| \psi_\delta'( \onesolution_\tau - \othersolution_\tau ) \nabla ( \onesolution_\tau - \othersolution_\tau) \cdot \nabla V \right| \\
                           & \leq \frac{\delta}{2} \| \mobee'' \|_{L^{\infty}} \| \nabla ( \onesolution_\tau - \othersolution_\tau) \|_{L^2(\Omega)} \| \nabla V \|_{L^2(\Omega)}.
        \end{align*}
        Integrating by parts and  using similar ideas, and the properties of $\psi_\delta$ we can deal with $\mathcal J_2$,
        \begin{align*}
            \mathcal J_2 = \int_{\Omega} \psi_\delta( \onesolution_\tau - \othersolution_\tau ) \diver \left( \mobee(\onesolution_\tau) \nabla V \right) - \int_{\partial \Omega} \psi_\delta (\onesolution_\tau - \othersolution_\tau ) \mobee(\onesolution_\tau) \nabla V \cdot {n} ,
        \end{align*}
        and therefore $|\mathcal J_2|  \leq C \delta$.

        \step[Limit $\delta \rightarrow 0$]
        Hence, combining everything and integrating in time from $0$ to $t$ we recover that,
        \begin{equation*}
            \int_\Omega j_\delta(\onesolution_t - \othersolution_t) \leq \int_\Omega j_\delta(\onesolution_0 - \othersolution_0) + C \delta.
        \end{equation*}
        Thus, if we take the limit $\delta \rightarrow 0$ we obtain the $L^1$ contraction,
        \begin{equation*}
            \| \onesolution_t - \othersolution_t \|_{L^1(\Omega)} \leq \| \onesolution_0 - \othersolution_0 \|_{L^1(\Omega)}.
        \end{equation*}
        Now, we choose $j_\delta$ such that $j_\delta''(s) = \frac{1}{\delta} \chi_{(0, \delta)} (s)$ and $j_\delta (s) \rightarrow s^+$ as $\delta \rightarrow 0$. Analogously to the previous case we recover,
        \begin{equation*}
            \int_\Omega ( \onesolution_t - \othersolution_t )^+ \leq \int_\Omega ( \onesolution_0 - \othersolution_0 )^+. \qedhere
        \end{equation*}
        \qedhere
    \end{enumeratesteps}
\end{proof}

\begin{remark}
    This argument is valid only for classical solutions. For strong solutions, the classical argument for $L^1$-contraction can be found in the book from Vázquez \cite{Vazquez07}, but it seems difficult to adapt it directly to the case of non-linear mobility.
    A more suitable notion of solution to deal with $L^1$ contractions is that of entropy solutions, but we will not discuss them here.
    We point the reader to \cite{Carrillo99,Karlsen_Risebro03}.
\end{remark}

\subsection{Semigroup theory for the problem \texorpdfstring{\eqref{eq:the problem regularised}}{(Peps)}. Proof of \texorpdfstring{\Cref{thm:Properties (Pee)}}{well-posedness of (Peps)}}\label{sec:Semigroup}

In this subsection we intend to develop the semigroup theory for the problem \eqref{eq:the problem regularised}. In order to do that, we take advantage of the classical theory constructed in the previous subsections, where we consider initial data $\rho_0 \in \mathcal{A}_+ \cap C^2 (\overline{\Omega})$. We have already discussed well-posedness, i.e., \ref{item:(Pee) classical solutions}, in \cref{sec:Well Posedness Pee}. We now proceed to prove the remaining claims.

\paragraph{Proof of \ref{item:(Pee) good semigroup}. $S^{(\ee)}$ is a \goodSemigroup{}.}
We divide the proof in several steps.
\begin{enumeratesteps}
    \step[Extension to $\mathcal{A}$]
    First we point out that, so far, we have defined $S_t^{(\ee)}$ in $\mathcal A_+ \cap C^2(\overline \Omega)$ in \ref{item:(Pee) classical solutions}. Since $S_t^{(\ee)}$ is an $L^1$-contraction in $\mathcal A_+ \cap C^2(\overline \Omega)$ (i.e., $1$-Lipschitz) and $\mathcal A_+ \cap C^2(\overline \Omega)$ is $L^1$-dense in $\mathcal A$, $S_t^{(\ee)}$ admits a unique continuous extension to $\mathcal A$. 
    We prove the four properties of a \goodSemigroup{} given in \Cref{def:Free-energy dissipating semigroup}.

    \step[Mass preserving and $L^1$-contraction property]
    They follow  from the conservation of mass for classical solutions and \Cref{lem:L1 contraction and continuous dependece} combined with the density of $\mathcal A_+ \cap C^2(\overline \Omega)$ in $\mathcal{A}$.

    \step[$C_0$-semigroup]
    Take $\eta_0 \in \mathcal A_+ \cap C^2(\overline \Omega)$. Then
    \begin{align*}
        \|S_h \rho_0 - \rho_0\|_{L^1 (\Omega)} & \le \|S_h \rho_0 - S_h \eta_0\|_{L^1(\Omega)} + \|S_h \eta_0 - \eta_0 \|_{L^1(\Omega)} + \|\eta_0 - \rho_0\|_{L^1(\Omega)} \\
                                               & \le \|S_h \eta_0 - \eta_0 \|_{L^1(\Omega)} + 2 \|\eta_0 - \rho_0\|_{L^1(\Omega)}
    \end{align*}
    Letting $h \to 0$, and using the properties of classical solutions
    \begin{align*}
        \limsup_{h \to 0} \|S_h \rho_0 - \rho_0\|_{L^1(\Omega)}
         & \le 2 \|\eta_0 - \rho_0\|_{L^1(\Omega)}.
    \end{align*}
    Now we can take the infimum over $\eta_0 \in \mathcal A_+ \cap C^2(\overline \Omega)$.

    \step[{Energy dissipation and $C([0,T]; W^{-1,1})$}]
    For data in $\rho_0 \in \mathcal A_+ \cap C^2(\overline \Omega)$ we recall \Cref{prop:Convergence of the free_energies}. Hence, we can apply \eqref{eq:drho/dt W -1,1}.
    In particular,
    \begin{align*}
        \| \rho_{t_2} -  \rho_{t_1} \|_{W^{-1,1}(\Omega)}
         & \leq \int_{t_1}^{t_2} \left\| \frac{\partial \rho_\sigma}{\partial t} \right\|_{W^{-1,1}(\Omega)} d \sigma
        \\
         & \le
        \| \mobee \|_{L^{\infty}([0,\alpha])}^{\frac{1}{2}} |\Omega|^{\frac{1}{2}}  \left( \mathcal{F}_{\varepsilon} [\rho_{t_1}] - \mathcal{F}_{\varepsilon} [ \rho_{t_2} ] \right)^{\frac{1}{2}} |t_2-t_1|^{\frac{1}{2}}.
    \end{align*}
    Let $\rho_0 \in \mathcal A_+$. Due to the upper and lower bound \eqref{eq:initial datum A+ bounds},
    there exists an approximating sequence $\rho_0^{\{k\}} \in \mathcal A_+ \cap C^2(\overline \Omega) $ such that
    $$
        0 < c_1 \leq (U_\ee')^{-1} (C_1 - V) \le  \rho_0^{\{k\}} \le (U_\ee')^{-1} (C_2 - V) \leq c_2 < \alpha
    $$
    uniformly. Since we have a comparison principle, these bounds are satisfied by $S_t^{(\ee)} \rho_0^{ \{k\} }$ for all positive times.
    Due to \Cref{lem:free energy dissipation LHS is lsc} and the continuity of $U_\ee$, we can pass to the limit in the estimates for $S_t^{(\ee)} \rho_0^{\{k\}}$.

    \step[Weak solution]
    It follows from the stability result in \Cref{lem:L1 stability of weak solutions}.
\end{enumeratesteps}

\paragraph{Proof of \ref{it:S^eps maximum principle}. Strong Maximum Principle.}
Take $\rho_0 \in \mathcal A$ and $\rho_0 \not \equiv 0, \alpha$.
We prove first that $S_t^{(\ee)}\rho > 0$.
Since $\rho_0 \ge 0$, there exists a $c > 0$ and small ball $B_r(x)$ such that $ \rho_0 \ge c$ in $B_r(x)$. Take a smooth function $0 \le \eta_0 \le c$ that is $0$ outside $B_r(x)$ and positive in $B_{r/2} (x)$.
By the comparison principle and \Cref{lem:Strong Max Principle Classical Solutions} we have
$S_t^{(\ee)} \rho_0 \ge S_t^{(\ee)} \eta_0 > 0$ in $\Omega$.
The claim $\rho < \alpha$ follows similarly.

\paragraph{Proof of \ref{it:S^eps strict positivity}. Strict positivity.}
As already mentioned, by using  \eqref{eq:initial datum A+ bounds} we have $0 < (U_\ee')^{-1} (C_1 - V) \le S_t^{(\ee)} \rho_0 \le (U_\ee')^{-1} (C_2 - V) < \alpha$.
\qed

\subsection{Existence for \texorpdfstring{\eqref{eq:the problem Omega}}{(P)}. Proof of \texorpdfstring{\Cref{thm:left side of D}}{existence}}\label{sec:Existence P0}

We divide the proof of \Cref{thm:left side of D} in several Lemmas.

\begin{lemma}
    \label{lem:Uee converges to U}
    Under the assumption \ref{hyp:H_U} the regularised diffusion $U_\ee$ from \eqref{eq:properties of Phi_ee}, \eqref{eq:mobee a.e. mob}, and \eqref{hyp:Sublinear saturation} is such that
    $U_\ee \to U$ uniformly in $[0,\alpha]$ and
    $U'_\ee \to U', U''_\ee \to U''
    $
    uniformly over compacts of $(0,\alpha)$.
\end{lemma}
\begin{proof}
    Recall that $U_\ee'' = \frac{\Phi_\ee'}{\mobee}$
    and \eqref{eq:properties of Phi_ee} includes that $\Phi_\ee \to \Phi$ in $C^2_{loc}((0,\alpha))$.
    From \eqref{eq:mob uniform}, the regularised mobility $\mobee$ converges uniformly to $\mob$, a continuous function that is positive in $(0,\alpha)$.
    Thus, $U_\ee'' \to U''$ uniformly over compacts of $(0,\alpha)$. Since $U_\ee'(\tfrac \alpha 2) = U'(\tfrac \alpha 2)$ we also get $U_\ee' \to U'$ uniformly over compacts of $(0,\alpha)$.
    Similarly, since $U_\ee(\tfrac \alpha 2) = U(\tfrac \alpha 2)$ we get convergence of $U_\ee \to U$ uniformly over compacts of $(0,\alpha)$.
    Lastly, since $\| U_\ee' \|_{L^1((0,\alpha))}$ is bounded, we have that $U_\ee$ are uniformly continuous in $[0,\alpha]$, and the uniform convergence over $[0,\alpha]$ follows.
\end{proof}

Using a sequence of classical solutions of the approximating problems \eqref{eq:the problem regularised} we are able to prove existence of a \goodSemigroup{} for \eqref{eq:the problem Omega}.
First, we prove the convergence of the semigroup in the weaker space $W^{-1,1}$.

\begin{lemma}[$W^{-1,1}$ precompactness]
    \label{lem:S^ee to S in W-11}
    There exists a sequence $\ee_k \to 0$ and a semigroup
    $
        S : [0,\infty) \times \mathcal A \to W^{-1,1} (\Omega)
    $
    such that, for each $\rho_0 \in \mathcal A$,
    \begin{equation}\label{eq:S^ee to S in W-11}
        S^{(\ee_k)} \rho_0 \to S\rho_0, \qquad \text{in } C_{loc} ([0,\infty);W^{-1,1} (\Omega)).
    \end{equation}
\end{lemma}

\begin{proof} 
    Our proof relies on the Ascoli-Arzelà theorem on $C(X,Y)$ with the topology of compact con\-ver\-gen\-ce (i.e., uniform convergence over compact sets), where
    $X := [0,\infty) \times \mathcal A$ and $Y := W^{-1,1} (\Omega)$.
    We observe that $X$ is a topological space and $Y $ is a metric space (in particular a uniform space).
    We want to show that
    $H \coloneqq \{ S^{(\varepsilon)} : \varepsilon \in (0,1) \}$
    is precompact in $C(X,Y)$ with the topology of compact convergence.
    It suffices to prove that (i) $H$ are equicontinuous and (ii) that
    for each $x = (t, \rho_0) \in X_{\ell}$ the set
    $H(x) \coloneqq \{ S_t^{(\varepsilon)}\rho_0 : \varepsilon \in (0,1) \}
    $ is pre-compact in $Y$.

    The proof of (ii) is simple. Given $x = (t,\rho_0) \in X$ we have that $\|S_t^{(\varepsilon)} \rho_0\|_{L^1} = \|\rho_0\|_{L^1}$. Due to the compact embedding of $L^1$ in $W^{-1,1}$ we observe that the set $H(x)$ is precompact in $W^{-1,1} (\Omega)$.

    Now let us prove (i).
    We point that the energy is controlled in $\mathcal A$ since
    \begin{align*}
        \left|\mathcal{F}_\ee [\rho] \right| & \leq \int_\Omega |U_\ee (\rho)| + \int_\Omega V \rho \leq \max_{\substack{\ee \in [0,1] \\ s \in \left[0,\alpha \right]}} |U_\ee(s)|
        |\Omega | + \alpha \| V \|_{L^1(\Omega)} \eqqcolon \mathsf C .
    \end{align*}
    Thus, we can control the time continuity using the gradient-flow structure
    \[
        \| S^{(\ee)}_{t+h} \rho_0 - S^{(\ee)}_t \rho_0 \|_{W^{-1,1}(\Omega)} \le C(\mathcal F_\ee[S^{(\ee)}_{t} \rho_0] - \mathcal F_\ee [S^{(\ee)}_{t+h} \rho_0] )^{\frac 1 2}h^{\frac 1 2} \le \mathsf C h^{\frac 1 2} .
    \]
    Using the continuous embedding of $L^1(\Omega)$ in $W^{-1,1}(\Omega)$, and the $L^1$-contraction property of the semigroup, we recover that
    \begin{align*}
        \| S^{(\ee)}_{t+h} \rho_0 - S^{(\ee)}_t \eta_0 \|_{W^{-1,1}(\Omega)} & \le \| S^{(\ee)}_{t+h} \rho_0 - S^{(\ee)}_t \rho_0 \|_{W^{-1,1}(\Omega)} + \| S^{(\ee)}_{t} \rho_0 - S_t^{(\ee)} \eta_0 \|_{W^{-1,1}(\Omega)}
        \\
                                                                             & \le \| S^{(\ee)}_{t+h} \rho_0 - S^{(\ee)}_t \rho_0 \|_{W^{-1,1}(\Omega)} + C(\Omega) \| S^{(\ee)}_{t} \rho_0 - S^{(\ee)}_t \eta_0 \|_{L^1(\Omega)} \\
                                                                             & \le \mathsf C h^{\frac 1 2} + C(\Omega) \| \rho_0 -  \eta_0 \|_{L^1(\Omega)} .
    \end{align*}
    This shows that the family $H$ are equicontinuous.

    Therefore, by the Ascoli-Arzelà theorem, there exists $\ee_k \to 0$ such that $S^{(\ee_k)}$ converges to $S$ uniformly over compacts of $X$.
\end{proof}

We also show that the convergence happens in $L^p$ for any $\rho_0$ fixed, although it may not be uniform in $\rho_0$.

\begin{lemma}[$L^p$ pre-compactness]
    \label{lem:Existence weak solution} 
    Given a sequence $\ee_k \to 0$,
    $\rho_0 \in \mathcal A$ and $T > 0$ fixed,
    there exists a sub-sequence and $u \in L^{\infty} ((0,T) \times \Omega)$ such that
    \[
        S^{(\ee_k)} \rho_0 \to u \quad \text{in } L^p ((0,T)\times \Omega) \text{ for all } p \in [1, \infty ).
    \]
\end{lemma}

\begin{proof}%[Proof of \Cref{thm:left side of D}]
    Through the proof we consider the notation $\rho^{(\ee)} = S^{(\ee)} \rho_0$.
    Let us recall that from our assumptions we know that $\Phi$ and $\Phi_\ee$ are non-decreasing and there exists $s_0 \in (0,\alpha)$ such that $\Phi'(s_0) > 0$ and $\Phi_\ee(s_0)=\Phi(s_0)=0$.
    Notice that this implies that $\Phi_\ee, \Phi < 0$ in $(0,s_0)$ and $\Phi_\ee, \Phi > 0$ in $(s_0,\alpha)$.
    We consider
    \begin{equation*}
        \begin{aligned}
            \PhiPrimpp^{(\ee)} (s)
             & =
            \int_{s_0}^{\max\{s,s_0\}} [\Phi_\ee (\alpha) - \Phi_\ee (\sigma)] \Phi_{\ee} (\sigma) d \sigma
            \quad \text{ and } \quad
            \PhiPrimnp^{(\ee)} (s)
            =
            \int_0^{\min\{s,s_0\}} [\Phi_\ee (\sigma) - \Phi_\ee (0)] \Phi_{\ee} (\sigma) d \sigma . 
        \end{aligned}
    \end{equation*}
    Notice that we can write $(\PhiPrimpp^{(\ee)})'(s) = [\Phi_\ee (\alpha) - \Phi_\ee (s)] [\Phi_{\ee} (s)]^+$ and $(\PhiPrimnp^{(\ee)})'(s)=[\Phi_\ee (s) - \Phi_\ee (0)] [\Phi_{\ee} (s)]^-$.
    We will prove convergence of $\PhiPrimpp^{(\ee)} (\rho^{(\ee)})$ and $\PhiPrimnp^{(\ee)} (\rho^{(\ee)})$ and combine them to show the limit of $\rho^{(\ee)}$.
    We divide the proof in several steps. 
    \begin{enumeratesteps}

        \step[Application of Aubin-Lions Lemma]
        \label{step:1 of lem:Existence weak solution}

        We define
        \begin{equation*}
            F_{\ee} = \mobee (\rho^{(\ee)}) \nabla \left( U_{\ee}' (\rho^{(\ee)} )+V \right),
        \end{equation*}
        which is uniformly bounded in $\ee$ in $L^2((0,T) \times \Omega)$ due to \eqref{eq:flow of the free energy} and the uniform bound on $\mobee$. Furthermore, we get
        \begin{equation}
            \label{eq:main term of step 1 of lem:Existence weak solution}
            \frac{\partial}{\partial t} \left( \PhiPrimpp^{(\ee)} (\rho^{(\ee)}) \right) = \diver \left( [\Phi_\ee(\alpha) - \Phi_\ee(\rho^{(\ee)})] [ \Phi_{\ee} (\rho^{(\ee)}) ]^+ F_{\ee} \right) -  \nabla( [\Phi_\ee(\alpha) - \Phi_\ee(\rho^{(\ee)})] [ \Phi_{\ee} (\rho^{(\ee)}) ]^+) \cdot F_{\ee}.
        \end{equation}
        Let us show that this is bounded in $L^1 (0,T; H^{-1} (\Omega))$. 
        We observe that
        \[
            \| [\Phi_\ee(\alpha) - \Phi_\ee(\rho^{(\ee)})] [ \Phi_{\ee} (\rho^{(\ee)}) ]^+ F_{\ee} \|_{L^2 ((0,T) \times \Omega)} \le 2 \|\Phi_\ee\|^2_{L^\infty(0,\alpha)} \|F_\ee \|_{L^2 ((0,T) \times \Omega)},
        \]
        and this is uniformly bounded in $\ee$.
        Therefore, the first term on the right-hand side of \eqref{eq:main term of step 1 of lem:Existence weak solution} is controlled in $L^2(0,T; H^{-1}(\Omega)).$ 
        The second term on the right-hand side of \eqref{eq:main term of step 1 of lem:Existence weak solution} is bounded in $L^1( (0,T) \times \Omega)$
        due to \eqref{eq:nabla phi(rho) a priori} and  \eqref{eq:flow of the free energy}. 
        Hence,  
        $\frac{\partial}{\partial t} \left( \PhiPrimpp^{(\ee)} (\rho^{(\ee)}) \right)$ is uniformly bounded in 
        $L^1 (0,T ; H^{-1} (\Omega ))$. 
        Moreover, from \Cref{lem:hypothesis Phi/Phi' is uniform}, we have that
        \begin{equation*}
            \left| \nabla \PhiPrimpp^{(\ee)} (\rho^{(\ee)}) \right| = \left| \frac{[\Phi_\ee(\alpha) - \Phi_\ee(\rho^{(\ee)})][ \Phi_{\ee} (\rho^{(\ee)}) ]^+}{\Phi_{\ee}' (\rho^{(\ee)})} \nabla \Phi_{\ee} (\rho^{(\ee)} ) \right| \leq C \left| \nabla \Phi_{\ee} (\rho^{(\ee)} ) \right|.
        \end{equation*} 
        The elements $\Phi_\ee(\rho^{(\ee)})$ are integrable and uniformly bounded, for $\ee$ small enough. This, in addition to \eqref{eq:nabla phi(rho) a priori}, implies that $\PhiPrimpp^{(\ee)} (\rho^{(\ee)})$ is uniformly bounded in $L^2(0,T; H^1 (\Omega))$. Let us take the given sequence $\ee_k$. Due to the Aubin-Lions Lemma there exists $\limitpp \in L^2 ((0,T) \times \Omega )$ and a subsequence such that
        \begin{equation}\label{eq:Strong Convergence +}
            \PhiPrimpp^{(\ee_k)} (\rho^{(\ee_k)}) \rightarrow \limitpp \quad \text{strongly in } L^2 ((0,T) \times \Omega ) \text{ and a.e.}
        \end{equation}
        Working analogously, we can prove that there exists $\limitnp \in L^2 ((0,T) \times \Omega )$ such that, up to a further subsequence,
        $
            \PhiPrimnp^{(\ee_{k})} (\rho^{(\ee_{k})}) \rightarrow \limitnp
        $
        strongly in $L^2 ((0,T) \times \Omega )$ and a.e..

        \step[Characterisation of the limit]
        By construction $\Phi_\ee(s)$ is non-decreasing, $\Phi_\ee(s_0) = 0$, and $\Phi_\ee'(s_0) > 0$. Therefore, we have that
        $\Phi_\ee (s) > 0 $ for all $\ee \in [0,1]$ and $s > s_0$.
        Joining this fact with \eqref{eq:Assumption Existence} we recover that  $(\PhiPrimpp^{(\ee)})'(s) > 0$ for all $s \in (s_0, \alpha)$.
        Clearly, $\PhiPrimpp^{(\ee)}(s) = 0$ for $s < s_0$.
        Likewise, $(\PhiPrimnp^{(\ee)})'(s) =\left[ \Phi_\ee(s) - \Phi_\ee (0) \right] \left[ \Phi_{\ee}(s) \right]^- \leq 0$ and is strictly negative in $(0,s_0)$.

        We can invert the functions $s \mapsto \Psi_i^{(\ee)} (s) / \Psi_i^{(\ee)} (\alpha)$.
        For $\varkappa \in (0,1)$ by Bolzano's theorem there exists $s(\varkappa) \in [0,\alpha]$ such that $\Psi_i^{(\ee)}(s(\varkappa)) = \varkappa \Psi_i^{(\ee)}(\alpha)$.
        Because of the construction of $\Phi_i^{(\ee)}$, we know that $s(\varkappa) \in (0,\alpha)$. We define 
        \[
            \InvPhiPrimpp^{(\ee)}  : [0,1] \rightarrow [s_0, \alpha], \qquad \InvPhiPrimpp^{(\ee)} (\varkappa) = \begin{dcases}
                s_0          & \text{if } \varkappa = 0,       \\
                s(\varkappa) & \text{if } \varkappa \in (0,1), \\
                \alpha       & \text{if } \varkappa = 1.
            \end{dcases}
        \]
        This is a non-decreasing function.
        We can make a similar construction for $\Psi_2^{(\ee)}$, which we denote $\InvPhiPrimnp^{(\ee)}$.
        For $\ee \ge 0$, it holds that
        \begin{equation}
            \label{eq:InvPhiPrimpp characterisation}
            \InvPhiPrimpp^{(\ee)} \left(\frac{\PhiPrimpp^{(\ee)} (s)}{\PhiPrimpp^{(\ee)} (\alpha)}\right)
            =
            \begin{dcases}
                s_0 & \text{if } s \in [0, s_0] ,    \\
                s   & \text{if } s \in (s_0,\alpha],
            \end{dcases}
            \quad \text{ and } \quad \InvPhiPrimnp^{(\ee)}  \left(\frac{\PhiPrimnp^{(\ee)} (s)}{\PhiPrimnp^{(\ee)} (\alpha)}\right)
            = \begin{dcases}
                s   & \text{if } s \in [0,s_0],       \\
                s_0 & \text{if } s \in (s_0,\alpha] .
            \end{dcases}
        \end{equation}
        We can therefore write
        \begin{equation}\label{eq:rho^eekl characterisation}
            \begin{aligned}
                \rho^{(\ee)} 
                 &
                = \InvPhiPrimpp^{(\ee)} \left(\frac{\PhiPrimpp^{(\ee)} (\rho^{(\ee)})}{\PhiPrimpp^{(\ee)} (\alpha)}\right)
                + \InvPhiPrimnp^{(\ee)}  \left(\frac{\PhiPrimnp^{(\ee)} (\rho^{(\ee)})}{\PhiPrimnp^{(\ee)} (\alpha)}\right)  - s_0.
            \end{aligned}
        \end{equation}
        To show the convergence along the sequence $\ee_k$ of each term, we write the triangular inequality,
        \begin{align*}
             & \left| \InvPhiPrimpp^{(\ee)} \left(\frac{\PhiPrimpp^{(\ee)} (\rho^{(\ee)})}{\PhiPrimpp^{(\ee)} (\alpha)}\right)
            -  \InvPhiPrimpp^{(0)} \left(\frac{\limitpp  (t,x)}{\PhiPrimpp^{(0)} (\alpha)}\right) \right|                      \\
             & \qquad \leq
            \left|  \InvPhiPrimpp^{(\ee)} \left(\frac{\PhiPrimpp^{(\ee)} (\rho^{(\ee)})}{\PhiPrimpp^{(\ee)} (\alpha)}\right)
            - \InvPhiPrimpp^{(0)} \left(\frac{\PhiPrimpp^{(\ee)} (\rho^{(\ee)})}{\PhiPrimpp^{(\ee)} (\alpha)}\right)
            \right|
            + \left| \InvPhiPrimpp^{(0)} \left(\frac{\PhiPrimpp^{(\ee)} (\rho^{(\ee)})}{\PhiPrimpp^{(\ee)} (\alpha)}\right)
            - \InvPhiPrimpp^{(0)} \left(\frac{\varkappa_1(t,x))}{\PhiPrimpp^{(0)} (\alpha)}\right) \right|.
        \end{align*}
        Due to \eqref{eq:Strong Convergence +}, the uniform convergence of $\Phi_\ee \to \Phi$,
        and the continuity of $\InvPhiPrimpp^{(0)}$
        we have that
        \begin{equation}\label{eq:Existence a.e. First Step}
            \InvPhiPrimpp^{(0)}  \left( \frac{\PhiPrimpp^{(\ee_{k})} (\rho^{(\ee_{k})})}{\PhiPrimpp^{(\ee_{k})} (\alpha)}\right) \rightarrow  \InvPhiPrimpp^{(0)} \left(\frac{\limitpp}{\PhiPrimpp^{(0)} (\alpha)} \right) 
            \quad \text{a.e. } (0,T) \times \Omega
        \end{equation}
        and similarly for $\PhiPrimnp^{(\ee_{k})}$. 

        Hence, it suffices to prove uniform convergence of $\InvPhiPrimpp^{(\ee_{k})}$ to $\InvPhiPrimpp^{(0)}$.
        Let us take $0 < \varkappa_\ast \leq \varkappa^\ast < 1$ arbitrarily. 
        For any $\varkappa \in [\varkappa_\ast, \varkappa^\ast]$
        \begin{equation*}
            0 \le (\InvPhiPrimpp^{(\ee)})'(\varkappa)
            = \frac{\PhiPrimpp^{(\ee)}(\alpha)}{(\PhiPrimpp^{(\ee)})'(s(\varkappa))} \leq \frac{\PhiPrimpp^{(\ee)}(\alpha)}{[ \Phi_{\ee}(\alpha) - \Phi_{\ee} (s(\varkappa^\ast)) ] \Phi_{\ee}(s(\varkappa_\ast)) } \eqqcolon C_\ee(\varkappa_\ast,\varkappa^\ast)
        \end{equation*}
        independent of $\ee$. 
        Due to \eqref{eq:stronger convergence of Phiee} we have that $\Phi_\ee \to \Phi$ and $\PhiPrimpp^{(\ee)} \to \PhiPrimpp$ uniformly in $[0,\alpha]$ as $\ee \to 0$ and hence $C_\ee(\varkappa_\ast,\varkappa^\ast)$ converges as $\ee \to 0$. 
        We conclude that $\InvPhiPrimpp^{(\ee)}$ are uniformly Lipschitz over compact subsets of $(0,1)$.
        We can use the Ascoli-Arzelà to prove that, up to a further subsequence of $\ee_k$, they converge uniformly over compacts of $(0,1)$.
        Furthermore, due to \eqref{eq:InvPhiPrimpp characterisation} we can establish that
        the limit is $\InvPhiPrimpp^{(0)}$.
        Also, we observe that $\InvPhiPrimpp^{(\ee_{k})}(0) = s_0 = \InvPhiPrimpp^{(0)}(0)$, $\InvPhiPrimpp^{(\ee_{k})}(1) = \alpha = \InvPhiPrimpp^{(0)}(1)$.
        Since $\InvPhiPrimpp^{(\ee_{k})}$ are non-decreasing functions and $\InvPhiPrimpp^{({0})}$ is continuous, then the convergence is uniform in $[0,1]$.
        Up to a further subsequence the same reasoning holds for $ \InvPhiPrimnp^{(\ee_{k})}$.

        Taking into account \eqref{eq:rho^eekl characterisation} and \eqref{eq:Existence a.e. First Step}, it follows that
        \begin{equation*}
            \rho^{(\ee_{k})} \rightarrow u \coloneqq \InvPhiPrimpp^{(0)} \left(\frac{\limitpp}{\PhiPrimpp^{(0)}(\alpha)} \right) + \InvPhiPrimnp^{(0)} \left(\frac{\limitnp}{\PhiPrimnp^{(0)}(\alpha)} \right) - s_0 \quad \text{a.e. } (0,T) \times \Omega.
        \end{equation*} 
        By the
        Dominated Convergence Theorem the convergence also holds in $L^p((0,T) \times \Omega)$ for $p \in [1,\infty)$. \qedhere 
    \end{enumeratesteps}
\end{proof}

We are now ready to show the main result of this subsection.

\begin{proof}[Proof of \Cref{thm:left side of D}]
    We divide the proof in several steps.

    \begin{enumeratesteps}
        \step[{For $\rho_0 \in \mathcal A$ the $C([0,T]; W^{-1,1})$ limit also happens in $L^1((0,1)\times\Omega)$}]

        In \Cref{lem:S^ee to S in W-11} we describe the $C([0,T]; W^{-1,1} (\Omega ))$ limit.
        Due to \Cref{lem:Existence weak solution}, for any $\rho_0 \in \mathcal A$ and any $T>0$
        we have that, up to a subsequence $\ee_{k_\ell}$, $S^{(\ee_{k_\ell})} \rho_0 \to u$ strongly in $L^1((0,T)\times \Omega)$.
        Due to the uniqueness of the limit, $u = S \rho_0$.
        By the stability \Cref{lem:L1 stability of weak solutions}, we conclude that $S \rho_0$ is a weak solution of \eqref{eq:the problem Omega}.

        \step[Uniform time continuity for good initial data]
        \label{step:1 of thm:left side of D}
        Let us fix $\ee > 0$. For $\rho_0 \in \mathcal A_+ \cap C^2(\overline \Omega)$ we have the $L^1$ limit
        \[
            \lim_{h \to 0} \frac{S_h^{(\ee)} \rho_0 - \rho_0}{h} = \frac{\partial }{\partial t} \rho (0,x) = \Delta \Phi_\ee (\rho_0) + \diver(\mobee(\rho_0) \nabla V).
        \]
        Notice that if $\Delta \Phi_\ee (\rho_0) + \diver(\mobee(\rho_0) \nabla V) = 0$ then $\rho_0$ is a stationary strong solution, so $S_h^{(\ee)} \rho_0 = \rho_0$. 
        There exists $ h_0(\ee,\delta) > 0$ such that
        \[
            \| {S_h^{(\ee)} \rho_0 - \rho_0} \|_{L^1(\Omega )} \le (1+\delta) h \| \Delta \Phi_\ee (\rho_0) + \diver(\mobee(\rho_0) \nabla V) \|_{L^1(\Omega )}, \qquad \forall h \le h_0.
        \]
        Let us take $t = hk$ for some $h \leq h_0$. Then, using the triangular inequality and the $L^1$-contraction property it follows that
        \begin{align*}
            \| {S_h^{(\ee)} \rho_0 - \rho_0} \|_{L^1(\Omega )} & \leq \sum_{j=1}^k \| {S_{hj}^{(\ee)} \rho_0 - S_{h(j-1)}^{(\ee)}\rho_0} \|_{L^1(\Omega )} \leq k  \| {S_h^{(\ee)} \rho_0 - \rho_0} \|_{L^1(\Omega )} \\
                                                               & \leq (1+\delta) hk \| \Delta \Phi_\ee (\rho_0) + \diver(\mobee(\rho_0) \nabla V) \|_{L^1(\Omega )}                                                                                   \\
                                                               & \leq (1+\delta) t \| \Delta \Phi_\ee (\rho_0) + \diver(\mobee(\rho_0) \nabla V) \|_{L^1(\Omega )}.
        \end{align*}
        We now take $\delta \rightarrow 0$ to recover that
        \begin{equation*}
            \| {S_t^{(\ee)} \rho_0 - \rho_0} \|_{L^1(\Omega )} \leq t \| \Delta \Phi_\ee (\rho_0) + \diver(\mobee(\rho_0) \nabla V) \|_{L^1(\Omega )}.
        \end{equation*}
        Therefore, due to the semigroup property and the $L^1$ contraction, for any $t, s \geq 0$ it follows that
        \begin{equation}\label{eq:uniform time continuity for good data}
            \| {S_t^{(\ee)} \rho_0 - S_s^{(\ee)} \rho_0} \|_{L^1(\Omega )} \leq |t-s| \| \Delta \Phi_\ee (\rho_0) + \diver(\mobee(\rho_0) \nabla V) \|_{L^1(\Omega )},
        \end{equation}
        which is uniformly bounded since $\rho_0 \in \mathcal{A}_+ \cap C^2 (\overline{\Omega})$.

        \step[{Convergence $C([0,T]; L^1 (\Omega))$ for good initial data}]
        \label{step:2 of thm:left side of D}
        Take $\ee_{k_\ell}$ any subsequence of $\ee_k$.
        We recall that $L^1 ((0,T) \times \Omega) = L^1(0,T; L^1(\Omega))$, where the latter $L^1$ space is understood in the Bochner sense.
        As it happens for the usual $L^1$ spaces, we have a.e. convergence in time (see e.g., \cite[Theorem 9.2]{Mikusinski1978BochnerIntegral}), i.e., there exists a further sub-sequence, still denote $\ee_{k_\ell}$, such that
        $S_t^{(\ee_{k_\ell})} \rho_0 \to S_t \rho_0$ in $L^1(\Omega)$ for a.e. $t \in (0,T)$.
        Notice that \eqref{eq:uniform time continuity for good data} gives a uniform-Lipschitz time continuity.
        Let us call
        \[
            L \coloneqq \sup_{\ee \in (0,1]} \| \Delta \Phi_\ee (\rho_0) + \diver(\mobee(\rho_0) \nabla V) \|_{L^1(\Omega )}
        \]
        which is finite since $\rho_0 \in \mathcal A_+ \cap C^2 (\overline \Omega)$, $\Phi_\ee \to \Phi$ in $C^2_{loc} ((0,\alpha))$ due to \eqref{eq:properties of Phi_ee} and $\mobee \to \mob$ in $C([0,\alpha]) \cap C^1_{loc}((0, \alpha))$ due to \eqref{eq:mob uniform}.
        Thus, $S \rho_0 \in C([0,T]; L^1 (\Omega))$.
        Let $\Lambda = \{t \in [0,T] : S^{(\ee_{k_\ell})}_t \rho_0 \to S_t \rho_0 \}$.
        Take $\delta > 0$.
        Since $[0,T] \setminus \Lambda$ has measure $0$, the covering $\{(t - \delta, t+ \delta)\}_{t \in \Lambda}$ admits a finite sub-cover given by $t_1, \ldots, t_N$.
        For $t \in [0,T]$ there exists $t_i$ such that $|t-t_i|<\delta$. We estimate
        \begin{align*}
            \|S^{(\ee_{k_\ell})}_t \rho_0 - S_t \rho_0\|_{L^1(\Omega)} & \le \|S^{(\ee_{k_\ell})}_t \rho_0 - S^{(\ee_{k_\ell})}_{t_i} \rho_0\|_{L^1(\Omega)} + \|S^{(\ee_{k_\ell})}_{t_i} \rho_0 - S_{t_i} \rho_0\|_{L^1(\Omega)} + \|S_{t_i} \rho_0 - S_{t} \rho_0 \|_{L^1(\Omega)} \\
                                                                       & \le 2 L \delta + \|S^{(\ee_{k_\ell})}_{t_i} \rho_0 - S_{t_i} \rho_0\|_{L^1 (\Omega)}.
        \end{align*}
        Since there are finitely many $t_i$ and we have pointwise convergence over all of them, we can write
        \[
            \limsup_{\ell \to \infty} \sup_{t\in[0,T]} \|S^{(\ee_{k_\ell})}_t \rho_0 - S_t \rho_0\|_{L^1 (\Omega)} \le 2L\delta.
        \]
        Since any sub-sequence of $\ee_k$ has a further sub-sequence that converges in $C([0,T]; L^1(\Omega))$ and they all do so to the same limit, the whole sequence converges in $C([0,T]; L^1(\Omega))$.

        \step[$L^1$-contraction for good initial data]
        \label{step:3 of thm:left side of D}
        Take $\rho_0$, $\eta_0 \in \mathcal{A}_+ \cap C^2(\overline \Omega)$. Then, combining
        \Cref{step:2 of thm:left side of D}
        and using \eqref{eq:uniform time continuity for good data}, the $L^1$-contraction property of $S^{(\ee)}$, we have that
        \begin{align*}
            \| S_t \rho_0 - S_t \eta_0 \|_{L^1(\Omega)} & \leq \| S_t \rho_0 - S_t^{(\ee_k)} \rho_0 \|_{L^1(\Omega)} + \| S_t^{(\ee_k)} \eta_0 - S_t \eta_0 \|_{L^1(\Omega)} + \| S_t^{(\ee_k)} \rho_0 - S_t^{(\ee_k)} \eta_0 \|_{L^1(\Omega)} \\
                                                        & \leq  \| S_t \rho_0 - S_t^{(\ee_k)} \rho_0 \|_{L^1(\Omega)} + \| S_t^{(\ee_k)} \eta_0 - S_t \eta_0 \|_{L^1(\Omega)} + \|  \rho_0 -  \eta_0 \|_{L^1(\Omega)} .
        \end{align*}
        If we take the limit $k \rightarrow \infty $ it follows $\| S_t \rho_0 - S_t \eta_0 \|_{L^1(\Omega)} \leq \|  \rho_0 -  \eta_0 \|_{L^1(\Omega)}$. Furthermore, by density, we can extend this result to $\rho_0$, $\eta_0 \in \mathcal{A}$.

        \step[{Convergence in $C([0,T];L^1)$ for $\rho_0 \in \mathcal A$}] 
        Let us take $\eta_0 \in \mathcal{A}_+ \cap C^2(\overline \Omega)$. From the $L^1$-contraction property (\Cref{step:1 of thm:left side of D}) it follows that
        \begin{align*}
            \| S_t^{(\ee_k)} \rho_0 - S_t \rho_0 \|_{L^1 (\Omega)} & \leq \| S_t^{(\ee_k)} \rho_0 - S_t^{(\ee_k)} \eta_0 \|_{L^1 (\Omega )} + \| S_t^{(\ee_k)} \eta_0 - S_t \eta_0 \|_{L^1 (\Omega)} + \| S_t \eta_0 - S_t \rho_0 \|_{L^1 (\Omega)} \\
                                                                   & \leq  2 \| \rho_0 - \eta_0 \|_{L^1 (\Omega)} + \| S_t^{(\ee_k)} \eta_0 - S_t \eta_0 \|_{L^1 (\Omega)}.
        \end{align*}
        Hence, we have that
        \begin{equation*}
            \limsup_{k \rightarrow \infty} \sup_{t \in [0, T]} \| S_t^{(\ee_k)} \rho_0 - S_t \rho_0 \|_{L^1 (\Omega)} \leq 2 \| \rho_0 - \eta_0 \|_{L^1 (\Omega)}.
        \end{equation*}
        Therefore, if we take the infimum over $\eta_0 \in \mathcal{A}_+ \cap C^2(\overline \Omega)$ we conclude the proof.

        \step[$S$ is a \goodSemigroup{}  over $\mathcal{A}$]
        From \Cref{step:3 of thm:left side of D},
        $S$
        is an $L^1$-contraction for $\rho_0$, $\eta_0 \in \mathcal{A}$. From \Cref{lem:nabla phi(rho) a priori,lem:L1 stability of weak solutions}, we know that $S \rho_0$ are weak solutions to \eqref{eq:the problem Omega}.

        Let us now show that $S_t$ is a $C_0$-semigroup.
        The semigroup property is a direct consequence of
        \Cref{step:2 of thm:left side of D}.
        Due to the point-wise convergence $S_0 \rho_0 = \rho_0$, and we know $t \mapsto S_t \rho_0$ is continuous.

        Furthermore, for $\rho_0 \in \mathcal A$, if we use the notation $\rho^{(\ee)}_t = S_t^{(\ee)} \rho_0$ we know
        \begin{equation*}
            \int_{t_1}^{t_2} \int_\Omega \mobee (\rho^{(\ee)}) | \nabla (U'(\rho^{(\ee)}) + V) |^2 = \mathcal{F}_\ee [\rho_{t_1}^{(\ee)}] - \mathcal{F}_\ee [\rho_{t_2}^{(\ee)}].
        \end{equation*}
        We recall that $U_\ee \rightarrow U$ uniformly (\Cref{lem:Uee converges to U}) and $\rho^{(\ee_k)} \rightarrow\rho$ strongly in $C([0,T]; L^1(\Omega))$.
        Hence, it follows that $\mathcal{F}_\ee[\rho^{(\ee)}_t] \rightarrow \mathcal{F} [\rho_t]$.
        We can pass to the limit due to \Cref{lem:free energy dissipation LHS is lsc}.
        We can also pass to the limit in the $C([0,T]; W^{-1,1} (\Omega))$ estimate.

        \step[$S \rho_0 \in \mathcal{A}$]
        The semigroup $S^{(\ee_k)}$ is such that $0 \le S_t^{(\ee_k)} \rho_0 \le \alpha$ for every $k$ and every $t \geq 0$. Therefore, \Cref{step:2 of thm:left side of D} implies that $0 \leq S_t \rho_0 \leq \alpha$ for every $t \geq 0$. \qedhere
    \end{enumeratesteps}
\end{proof}

%% L1-local minimisers
\section{Local minimisers of the free energy. Proof of \texorpdfstring{\Cref{thm:Euler-Lagrange}}{Euler-Lagrange condition}}
\label{sec:Local minimiser}

When $V$ is not radially increasing we could build several constant-in-time weak solutions by ``pasting'' constant-in-time weak solutions of the problem \eqref{eq:the problem Omega} of the form $\inversedU (C - V(x))$ (\ref{item:Semi-group local minimiser}) for different values of the constants $C$, see \Cref{fig:Non-local minimiser}. However, as stated in \Cref{thm:Euler-Lagrange}, there exists a unique $L^1$-local minimiser, which corresponds to the case in which there is only one constant involved.
\begin{figure}[H]
    \centering
    \includegraphics[width=0.8\textwidth]{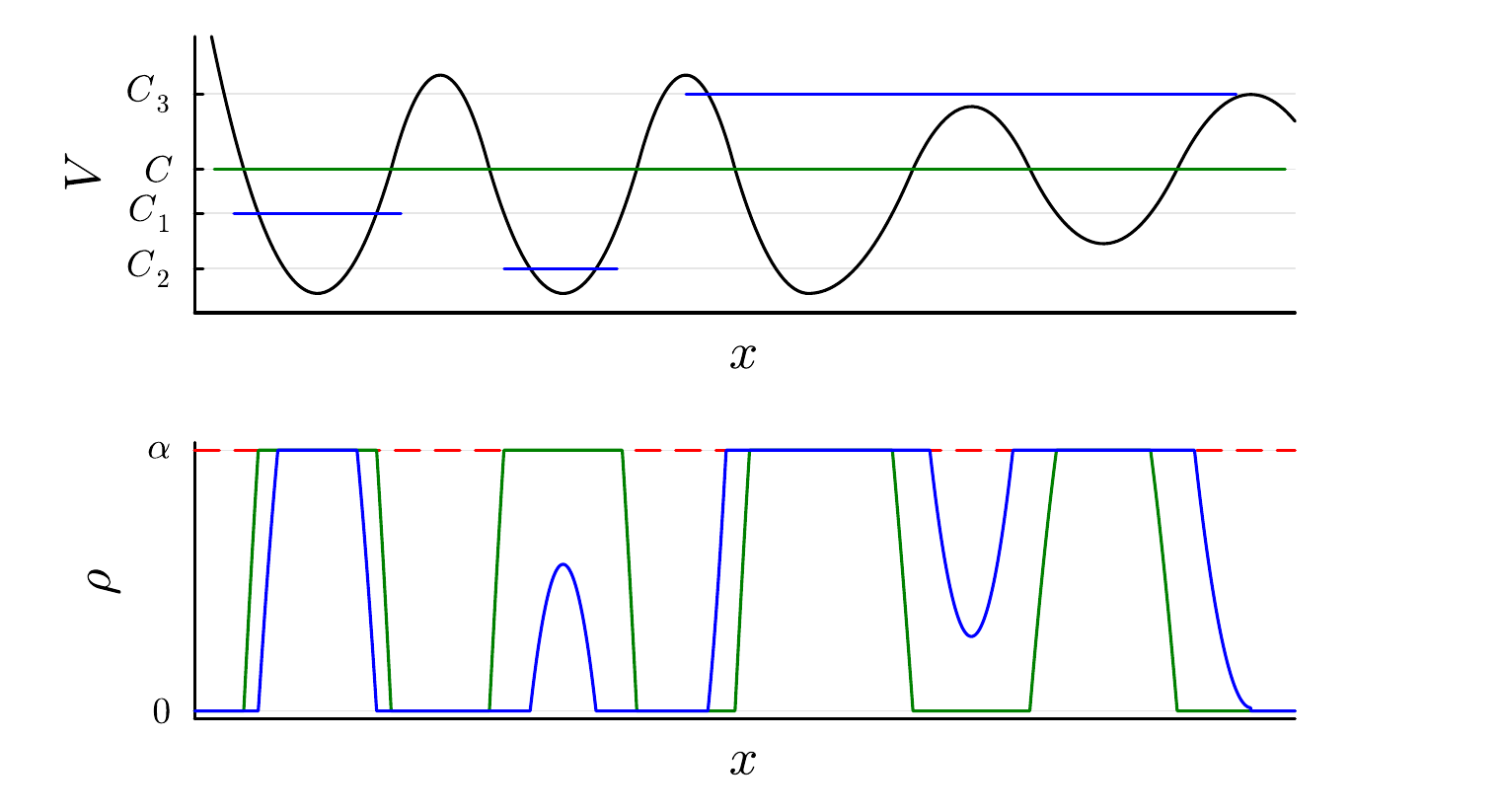}
    \caption{Example of steady states for $U(s)=s^2$ and $V$, the potential above, not radially increasing. The blue steady state is not an $L^1$-local minimiser of the free energy \eqref{eq:Free eenrgy ee=0 Omega}. On the other hand, the green steady state is an $L^1$-local minimiser, since there is only one constant involved, $C$; see \Cref{thm:Euler-Lagrange}.}
    \label{fig:Non-local minimiser}
\end{figure}
The phenomenon exposed in \Cref{fig:Non-local minimiser} is known in the no-saturation case $\mob (\rho) = \rho$. We now proceed to show it in the case with saturation.
Before the proof, we construct some auxiliary scaffolding.

\begin{proof}[Proof of \Cref{thm:Euler-Lagrange}]
    We split the proof in several steps.
    \begin{enumeratesteps}
        \step[Euler-Lagrange condition \eqref{eq:EL for ee = 0}]
        We distinguish two cases
        \begin{enumeratesteps}
            \step[{Euler-Lagrange condition when $0 \le \widehat \rho < \alpha$}]
            We take sets
            $
                A_\oneconstant = \{  x \in \Omega : \widehat\rho (x) < \alpha - \oneconstant \}.
            $
            If $|A_0| = 0$, the statement does not claim anything and the proof is complete.
            For the rest of the proof, we assume that $|A_\oneconstant| > 0$ for all $\oneconstant$ small enough.

            Now, for $\delta > 0$ small, we consider the variation $\rho_{\delta} (x) = \widehat\rho (x) + \delta\varphi (x)$ for $0 \le \psi \in C^\infty (\Omega)$ and
            \begin{equation}
                \label{eq:perturbation}
                \varphi (x) =  \chi_{A_\oneconstant} (x)\left( \psi(x) - \widehat \rho (x) ( \alpha - \widehat \rho (x)) \frac{\int_{A_\oneconstant} \psi(y) \diff y}{\int_{A_\oneconstant} \widehat \rho(y) ( \alpha - \widehat \rho (y)) \diff y}  \right).
            \end{equation}
            By construction $\int_\Omega \varphi = 0$.
            In order to prove $\rho_\delta \ge 0$, we consider
            $$
                0 < \delta < \frac{\int_{A_\oneconstant} \widehat \rho (\alpha - \widehat \rho)}{ \alpha \int_{A_\oneconstant} \psi }.
            $$
            Lastly we check that $\rho_\delta \le \alpha$.
            If $\widehat \rho(x) \ge \alpha - \oneconstant$ then
            $
                \rho_\delta (x) = \widehat \rho (x) \le \alpha.
            $
            If $\widehat \rho(x) < \alpha - \oneconstant$ then it suffices that
            $$
                0 < \delta < \frac{1}{\|\psi\|_{L^\infty}}.
            $$
            We have concluded that $\rho_\delta \in \mathcal A_M$.
            Computing the first variation
            $$
                0 \le \lim_{\delta \to 0^+} \frac{\mathcal F[\rho_\delta] - \mathcal F[\widehat \rho]} \delta = \int_{A_\lambda} \left( U'(\widehat \rho) + V \right) \varphi = \int_{A_\oneconstant} \left( U'(\widehat \rho(x)) + V (x) - C_\oneconstant \right) \psi(x) \diff x,
            $$
            where
            $$
                C_\oneconstant =  \frac{\int_{A_\oneconstant} (U'(\widehat \rho) + V) \widehat \rho ( \alpha - \widehat \rho )  }{\int_{A_\oneconstant} \widehat \rho ( \alpha - \widehat \rho )}.
            $$
            Since this holds for any $\psi \in C^\infty (\Omega)$ non-negative, we conclude
            $$
                U'(\widehat \rho) + V \ge C_\oneconstant , \qquad \text{ a.e.~in } A_\oneconstant.
            $$
            Since these sets $A_\oneconstant$ are monotonically increasing to $A_0$, we can let $\oneconstant \to 0$ to deduce the claim with
            $$
                C =  \frac{\int_{\widehat \rho < \alpha} (U'(\widehat \rho) + V) \widehat \rho ( \alpha - \widehat \rho ) }{\int_{\widehat \rho < \alpha} \widehat \rho ( \alpha - \widehat \rho ) } =
                \frac{\int_{\Omega} (U'(\widehat \rho) + V) \widehat \rho ( \alpha - \widehat \rho ) }{\int_{\Omega} \widehat \rho ( \alpha - \widehat \rho ) }
                .
            $$
            ~

            \step[{Euler-Lagrange condition when $0 < \widehat \rho \le \alpha$}]
            The proof is similar. Let us now take the sets
            $
                A^{\oneconstant} = \left\lbrace x \in \Omega : \widehat\rho (x) > \oneconstant \right\rbrace.
            $
            As we explain before, we assume that $|A^{\oneconstant}|>0$ for all $\oneconstant$ small enough. For $\delta > 0$ small, we choose the variation $\rho^{\delta} (x) = \widehat\rho (x) - \delta\varphi (x)$ for $0 \le \psi \in C^\infty (\Omega)$ and
            $$
                \varphi (x) =  \chi_{A^\oneconstant} (x)\left( \psi(x) - \widehat \rho (x) ( \alpha - \widehat \rho (x)) \frac{\int_{A^\oneconstant} \psi(y) \diff y}{\int_{A^\oneconstant} \widehat \rho(y) ( \alpha - \widehat \rho (y)) \diff y}  \right).
            $$
            With the same assumptions on $\delta$ that we have considered earlier is easy to prove that $\rho^{\delta} \in \mathcal{A}$. Computing the first variation
            $$
                0 \le \lim_{\delta \to 0^+} \frac{\mathcal F[\rho^\delta] - \mathcal F[\widehat \rho]} \delta = - \int_\Omega \left( U'(\widehat \rho) + V \right) \varphi = - \int_{A^\oneconstant} \left( U'(\widehat \rho(x)) + V (x) - C^\oneconstant \right) \psi(x) \diff x,
            $$
            where
            $$
                C^\oneconstant =  \frac{\int_{A^\oneconstant} (U'(\widehat \rho) + V) \widehat \rho ( \alpha - \widehat \rho )  }{\int_{A^\oneconstant} \widehat \rho ( \alpha - \widehat \rho )}.
            $$
            Analogously to what we remarked earlier, it follows that
            $$
                U'(\widehat \rho) + V \le C^\oneconstant , \qquad \text{ a.e.~in } A^\oneconstant,
            $$
            and
            $$
                C^{\oneconstant} \rightarrow C =  \frac{\int_{\widehat \rho >0} (U'(\widehat \rho) + V) \widehat \rho ( \alpha - \widehat \rho ) }{\int_{\widehat \rho >0} \widehat \rho ( \alpha - \widehat \rho ) } =
                \frac{\int_{\Omega} (U'(\widehat \rho) + V) \widehat \rho ( \alpha - \widehat \rho ) }{\int_{\Omega} \widehat \rho ( \alpha - \widehat \rho ) }
                .
            $$
        \end{enumeratesteps}

        \step[Proof of formula \eqref{eq:Euler-Lagrange for P}]
        We re-write the Euler-Lagrange conditions using that $U'$ is non-decreasing as
        \begin{align*}
            \widehat \rho(x)\ge (U')^{-1}( C - V(x) ),   & \qquad \text{ if } 0 \le \widehat\rho(x) <\alpha ,  \\
            \widehat \rho (x) \le (U')^{-1}( C - V(x) ), & \qquad \text{ if } 0 < \widehat\rho(x) \le \alpha .
        \end{align*}
        We distinguish the three possible cases
        \begin{itemize}
            \item Let $x \in \Omega$ be such that $(U')^{-1}( C - V(x) ) \in (0,\alpha)$.
                  If $\widehat \rho(x) = 0$ we get to contradiction with the first condition. If $\widehat \rho(x) = \alpha$ we get into contradiction with the second condition. Thus, we conclude
                  $$
                      \widehat \rho (x) = (U')^{-1}( C - V(x) ).
                  $$

            \item Let $x \in \Omega$ be such that $(U')^{-1}( C - V(x) ) \le 0$. If $\widehat\rho(x) > 0$ we get a contradiction with the second condition. Since we know $\widehat\rho \ge 0$ in $\Omega$, we conclude $\widehat \rho(x) = 0$.

            \item Let $x \in \Omega$ be such that $(U')^{-1}( C - V(x) ) \ge \alpha$. If $\widehat\rho(x) < \alpha$ we get a contradiction with the first condition. Therefore, we conclude that $\widehat \rho(x) = \alpha$.
        \end{itemize}

        \step[Uniqueness of the constant]
        Due to \eqref{eq:U locally strictly convex}, we have that $U'$ is invertible with continuous inverse. Furthermore, the function
        \begin{equation*}
            F(C) \coloneqq \int_\Omega \inversedU (C- V(x))
        \end{equation*}
        is non-decreasing and continuous. Let us recall the definition of $\underline{\zeta}$ and $\overline{\zeta}$ in \eqref{eq:defn xi underline and overline}. Then,
        we distinguish three cases of $C$:
        \begin{enumerate}
            \item $C \ge \max V + \overline{\zeta}$. Then $F(C) = \alpha |\Omega|$.
            \item $C \le \min V + \underline{\zeta}$. Then $F(C) = 0$.
            \item $\min V + \underline{\zeta} < C < \max V + \overline{\zeta}$ in $\Omega$.
                  Then, by the continuity of $V$ there exists $\delta > 0$ and $A \subset \Omega$ of positive measure, such that for all $c\in\R$ satisfying $|c-C| < \delta$ and $x \in A$, then
                          $\underline{\zeta} < c - V (x) < \overline{\zeta}$.
                  Hence, we can write
                  \[
                      F(C) = \int_A (U')^{-1} (C- V(x)) + \int_{\Omega \setminus A} \inversedU (C- V(x)).
                  \]
                  The second term of the sum is still non-decreasing.
                  Due to \eqref{eq:U locally strictly convex}, the first term of the sum is strictly increasing.
        \end{enumerate}
        Thus, if $M \in (0,\alpha|\Omega|)$ there is a unique $C$ such that $F(C) = M$. \qedhere
    \end{enumeratesteps}
\end{proof}

\begin{remark}
    Notice that we are taking perturbations $\widehat \rho + \delta \varphi$ with $\varphi$ given by \eqref{eq:perturbation} (and the corresponding modification at $\alpha$).
    This means that if $\widehat \rho + \delta \varphi$ converges to $\widehat \rho$ in a topology $\mathcal T$, then any local minimiser in the topology $\mathcal T$ satisfies the Euler-Lagrange conditions and it is therefore \eqref{eq:Euler-Lagrange for P}.
    In particular, $\mathcal T$ can be any $L^p$ topology, $C^k$ topology, or even 2-Wasserstein. Notice, however, that $\widehat \rho +  \delta \varphi$ does not converge to $\widehat \rho$ in the $\infty$-Wasserstein topology. In fact, there can be many more $\infty$-Wasserstein local minimisers (see \Cref{rem:W infty minimisers} below).
    For necessary condition for $\infty$-Wasserstein minimisation we point the reader to \cite{Balague_Carrillo_Laurent_Raoul13, Carrillo_Delgadino_Patacchini19}.
\end{remark}

%% Long time behaviour
\section{Existence of long-time behaviour. Proof of \texorpdfstring{\Cref{thm:long-time behaviour}}{long-time behaviour}}
\label{sec:Long time behaviour}

We study the long-time behaviour of the problem using semigroup theory. The proof of this result relies on an Aubin-Lions compactness argument, and the application of $L^1$ contraction techniques, we devote \cref{sec:Semigroup compactness} to this goal. Afterwards, in the next two subsections we discuss the two examples and for each one of them, we construct a \timeLimit{}. First, in \cref{sec:omega-limit} we describe the \timeLimit{} of  \eqref{eq:the problem Omega}, which corresponds to \ref{item:Time limit operator}. Afterwards, in \cref{sec:Convergence stationary state Pee} we do the same for the problem \eqref{eq:the problem regularised}, which corresponds to \ref{item:Time limit operator regularised}.

\subsection{From compactness to convergence}\label{sec:Semigroup compactness}

We present the following auxiliary lemma for the existence of a \timeLimit{}.

\begin{lemma}
    \label{lem:compactness on 0 1 times Omega is enough}
    Let $S$ be a \goodSemigroup{} for \eqref{eq:the problem Omega}. Assume that, for each $\rho_0 \in \mathcal A_+$, there exists $t_n \to \infty$ and $\rho^\infty \in \mathcal A$ such that
    \[
        S_{\bullet + t_n} \rho_0 \to \rho^\infty \text{ in } L^1 ((0,1) \times \Omega).
    \]
    Then, there exists a \timeLimit{}\,for $S_\infty$ and it is an $L^1$-contraction.
\end{lemma}
\begin{proof}
    We divide the proof in several steps.

    \begin{enumeratesteps}
        \step[Analysis for $\rho_0 \in \mathcal A_+$]~

        \begin{enumeratesteps}
            \step[{Convergence in $C([0,1]; L^1(\Omega))$}]
            Consider a sub-sequence in time $t_{n_k}$.
            Let us recall again that $L^1((0,1) \times \Omega) = L^1(0,1; L^1(\Omega))$ where the latter $L^1$ space is in the Bochner sense (see \cite[Theorem 9.2]{Mikusinski1978BochnerIntegral}). Then, up to a further subsequence still denoted $t_{n_k}$, we have that
            \[
                S_{s + t_{n_k}} \rho_0 \to \rho^\infty \qquad \text{in } L^1(\Omega) \text{ for a.e. } s \in [0,1].
            \]
            Let $\Lambda$ be the set of $s \in [0,1]$ where the convergence happens.
            Taking $\delta > 0$, then $\{(s - \delta, s + \delta)\}_{s \in \Lambda}$ is a cover of $[0,1]$, and there exists a finite subcover $\{(s_i - \delta, s_i + \delta)\}_{1 \leq i \leq N}$.
            For $s \in [0, 1]$
            \begin{align*}
                \|S_{s+t_{n_k}} \rho_0 - \rho^\infty \|_{L^1(\Omega)} & \le \|S_{s+t_{n_k}} \rho_0 - S_{s_i+t_{n_k}} \rho_0 \|_{L^1(\Omega)} + \| S_{s_i+t_{n_k}} \rho_0 - \rho^\infty\|_{L^1(\Omega)} \\
                                                                      & \le \|S_{|s-s_i|} \rho_0 - \rho_0 \|_{L^1(\Omega)} + \| S_{s_i+t_{n_k}} \rho_0 - \rho^\infty\|_{L^1(\Omega)} .
            \end{align*}
            Since there is a finite number of $s_i$, we recover
            \[
                \adjustlimits \limsup_{k \to \infty} \sup_{\mathclap{s \in [0,1]}} \|S_{s+t_{n_k}} \rho_0 - \rho^\infty \|_{L^1 (\Omega)} \le \sup_{\tau \in [0,\delta]} \|S_{\tau} \rho_0 - \rho_0 \|_{L^1 (\Omega)}.
            \]
            Letting $\delta \to 0$ we have shown the convergence of this subsequence.
            Since any sub-sequence of 
            $\rho_{t_n}$ 
            has a further subsequence converging in $C([0,1]; L^1 (\Omega))$, and they all do so to $\rho^\infty$, then the whole sequence $S_{\bullet + t_n} \rho_0$ converges in $C([0,1]; L^1(\Omega))$ to $\rho^\infty$.

            \step[$\rho^\infty$ is stationary for $S$, i.e., $S_t \rho^\infty = \rho^\infty$ for all $t\ge 0$]
            Due to the semigroup property for $a \in [0,1]$
            \[
                S_a S_{t_n} \rho = S_{a+t_n} \rho .
            \]
            Since $S_a$ is $1$-Lipschitz in $L^1$ we can pass to the limit on both sides to recover $S_a \rho^\infty = \rho^\infty$. Once more, due to the semigroup property, $S_t \rho^\infty = \rho^\infty$ for all $t > 0$.

            \step[Convergence of the whole sequence]
            Let us define
            \[
                F(t) \coloneqq \|S_{t} \rho_0 - \rho^\infty\|_{L^1(\Omega)} = \|S_{t} \rho_0 - S_{t} \rho^\infty\|_{L^1(\Omega)}.
            \]
            This function is clearly non-negative. Due to the $L^1$ contraction, it is non-increasing.
            We have assumed that $F(t_n) \to 0$. Thus, $F(t) \to 0$ as $t \to \infty$.
        \end{enumeratesteps}

        \step[General $\rho_0 \in \mathcal A$]
        From the previous step, there exists $S_\infty : \mathcal A_+ \to \mathcal A$ that is $1$-Lipschitz. There is a unique continuous extension $S_\infty : \mathcal A \to \mathcal A$, and it is also $1$-Lipschitz.

        For $\rho_0 \in \mathcal A$ we can now simply take $\eta_0 \in \mathcal A_+$ and write
        \begin{align*}
            \| S_t \rho_0 - S_\infty \rho_0 \|_{L^1(\Omega)} & \le \|S_t \rho_0 - S_t \eta_0\|_{L^1(\Omega)} + \|S_t \eta_0 - S_\infty \eta_0 \|_{L^1(\Omega)} + \|S_\infty \eta_0 - S_\infty \rho \|_{L^1(\Omega)} \\
                                                             & \le \|S_t \eta_0 - S_\infty \eta_0 \|_{L^1(\Omega)} + 2 \| \eta_0 - \rho_0 \|_{L^1(\Omega)} .
        \end{align*}
        Letting $t \to \infty$ we arrive at
        \[
            \limsup_{t\to \infty} \| S_t \rho_0 - S_\infty \rho_0 \|_{L^1 (\Omega)} \le  2 \| \eta_0 - \rho_0 \|_{L^1 (\Omega)}.
        \]
        Taking infimum over $\eta_0 \in \mathcal A_+$ we get the result.

        \step[The limit is a constant-in-time weak solution] 
        Since the free energy is bounded from below and the energy $\mathcal F[\rho_t]$ decays with time, it has a limit as $t\to\infty$. Hence, due to the stability of the disperssion term in \Cref{lem:free energy dissipation LHS is lsc}, it follows that $\mob(\rho^\infty)^{\frac 1 2} \nabla (U'(\rho^\infty) + V) = 0$ almost everywhere $\Omega$. 
        Multiplying once again by $\mob(\rho^\infty)^{\frac 1 2}$ we get $\nabla \Phi(\rho^\infty) + \mob(\rho^\infty) \nabla V = 0$. 
        Given a test function $\varphi \in H^1(\Omega)$, we can multiply by $\nabla \varphi$ and integrate in $ \Omega$ to the deduce the weak formulation of the stationary problem. \qedhere
    \end{enumeratesteps}
\end{proof}

%% Existence of long-time behaviour 
\subsection{For problem \texorpdfstring{\eqref{eq:the problem Omega}}{(P)}}\label{sec:omega-limit}

\begin{lemma}\label{lem:omega-limit}
    Let $\rho_0 \in \mathcal A_+$. Then, there exists $\rho^\infty \in \mathcal A$ and $t_n \to \infty$ such that $S_{\bullet + t_n} \rho_0 \to \rho^\infty$ in $L^1((0,1)\times \Omega)$ as $n \rightarrow \infty$.
\end{lemma}

\begin{proof}
    We use the lighter notation $\rho^{[n]}_s = S_{s +t_n} \rho_0$ for $s \in [0,1]$ and $\rho_t = S_{t} \rho_0$ for $t \geq 0$.
    First, we prove convergence by compactness.
    Analogously to the \Cref{step:1 of lem:Existence weak solution} of the proof of \Cref{lem:Existence weak solution}, up to a subsequence, we obtain that for all $p \in [1, \infty )$,
    \begin{equation}\label{eq:Convergence omega-limit}
        \rho^{[n]} \rightarrow \rho^{\infty} \quad \text{strongly in } L^p ([0,1] \times \Omega ).
    \end{equation}
    Using \Cref{lem:negative Sobolev Ascoli-Arzela} we have
    $
        \rho^{[n]} \rightarrow \rho^{\infty}$
    in
    $C([0,1]; W^{-1,1}(\Omega))$.

    Lastly, we show that the limit is stationary. From \Cref{thm:left side of D} it follows that  $S$ is a \goodSemigroup{}. Hence, due to \eqref{eq:W-11 continuity}, 
    $\mathcal F[\rho_t]$ is non-increasing and bounded below, therefore it admits a limit which we denote $\underline{\mathcal F}$, 
    the limit does not depend on time.
\end{proof}

\subsection{For problem \texorpdfstring{\eqref{eq:the problem regularised}}{(Peps)}}\label{sec:Convergence stationary state Pee}

The proof of \Cref{lem:omega-limit} can be applied also to \eqref{eq:the problem regularised}.
However, 
using that approach some of the technical become too complicated. 
We include now an elementary proof that works for $\ee > 0$.

\begin{lemma}
    \label{lem:regularised asymptotics existence subsequence}
    Let $\ee > 0$ and $\rho_0 \in \mathcal A_+$. Then, there exists $\rho^{(\ee),\infty} \in \mathcal{A}$ and $t_n \to \infty$ such that $S_{\bullet + t_n}^{(\ee)} \rho_0 \to \rho^{(\ee),\infty}$ in $L^1 ((0,1) \times \Omega)$ as $n \to \infty$.
\end{lemma}

\begin{proof}
    Once more, we use the lighter notation $\rho^{(\ee), [n]}_s = S_{s +t_n}^{(\ee)} \rho_0$ for $s \in [0,1]$ and $\rho^{(\ee)}_t = S_{t}^{(\ee)} \rho_0$ for $t \geq 0$.
    We prove convergence by compactness arguments.
    Using \Cref{lem:negative Sobolev Ascoli-Arzela} we have that, up to a subsequence,
    $$
        \rho^{(\ee),[n]} \rightarrow \rho^{(\ee), \infty} \quad \text{in } C([0,1];W^{-1,1}(\Omega )).
    $$
    For the point-wise convergence we aim to use the Aubin-Lions Lemma with the spaces $H^1 (\Omega) \subset L^2 (\Omega) \subset W^{-1,1} (\Omega)$.
    From \eqref{eq:nabla rho a priori} we recover
    \begin{equation*}
        \| \nabla \rho^{(\ee), [n]} \|_{L^2((0,1) \times \Omega )}^2 \leq 2 \int_{\Omega} G_{\ee} (\rho_{t_n}^{(\ee)} ) + C \| \nabla V \|_{L^2(\Omega)}^2,
    \end{equation*}
    where $G_{\ee}$ is defined in \eqref{eq:G_ee}. Next, we show that the first term in the RHS is uniformly bounded for each $\ee > 0$ fixed. Taking advantage of the definition of $G_\ee$, \eqref{eq:properties of Phi_ee}, and \eqref{eq:Lp estimate} it follows that
    \begin{equation*}
        \int_{\Omega} G_{\ee} (\rho_{t_n}^{(\ee)} ) = \int_{\Omega}  \int_{\frac{\alpha}{2}}^{{\rho_{t_n}^{(\ee)} (x)}} \int_{\frac{\alpha}{2}}^{\sigma} G_{\ee}''(s) \ds \, d \sigma \dx \leq C(\ee, \alpha) \left( |\Omega| + \| \rho_{t_n}^{(\ee)} \|_{L^1(\Omega)} + \| \rho_{t_n}^{(\ee)} \|_{L^2(\Omega)}^2 \right) \leq C(\ee, \alpha, |\Omega|).
    \end{equation*}
    Hence, for each $\ee > 0$ fixed, the sequence $\rho^{(\ee),[n]}$ is uniformly bounded in $L^2(0,1; H^1(\Omega))$.
    Therefore, we can apply the Aubin-Lions Lemma, up to a subsequence, we have that
    \begin{equation*}
        % \label{eq:strong L2 convergence}
        \rho^{(\ee),[n]} \rightarrow \rho^{(\ee), \infty} \quad \text{in } L^2((0,1) \times \Omega) ,
    \end{equation*}
    Due to \eqref{eq:W-11 continuity}, 
    $\mathcal F_\varepsilon[\rho_t]$ is non-increasing and bounded below, therefore it admits a limit which we denote $\underline{\mathcal F}_\varepsilon$.
    Hence, 
    the limit does not depend on time.
\end{proof}

\Cref{thm:long-time behaviour} follows from the combination of \Cref{lem:compactness on 0 1 times Omega is enough,lem:omega-limit,lem:regularised asymptotics existence subsequence}.

% Characterisation of long-time behaviour
\section{Analysis of the long-time limit}\label{sec:Analysis long time}

In this Section we understand the long-time behaviour. First, in \cref{sec:Asymptotic behaviour ee} we focus on the global attractor of \eqref{eq:the problem regularised} for $\ee >0$. Afterwards, in \cref{sec:Stationary state} we study some properties of the $\omega$-limit of \eqref{eq:the problem Omega}. Finally, in \cref{sec:malicious counterexamples} we construct an example in order to show that uniqueness of constant-in-time solutions of \eqref{eq:the problem Omega} is not necessarily true, and that the extra constant-in-time solutions also attract a large class of initial data.

\subsection{The global attractors for \texorpdfstring{\eqref{eq:the problem regularised}}{(Peps)}. Proof of \texorpdfstring{\Cref{th:ee continuous limit}}{global attractors of (Peps)}}
\label{sec:Asymptotic behaviour ee}

We first show an auxiliary result.

\begin{lemma}[A generalisation of \ref{item:Pee long time 2}]
    \label{lem:Regularity hat rho ee}
    If ${\rho}$ is a constant-in-time weak solution of \eqref{eq:the problem regularised}
    then exactly one of the following holds: $\rho \equiv 0$ in $\Omega$, $\rho \equiv \alpha$ in $\Omega$, or
    $\rho = \rho^{(\ee),\infty}$ given by
    \[
        \rho^{(\ee),\infty} (x) = (U_\ee ')^{-1} (C_\ee - V(x)) ,
    \]
    where $C_\ee$ is uniquely determined by the mass of $\rho$.
\end{lemma}

\begin{proof}
    We divide the proof in several steps.
    \begin{enumeratesteps}
        \step[{$\rho\equiv 0$, $\rho\equiv \alpha$, or $0 < \rho < \alpha$}]
        We write
        \begin{equation*}
            - \Delta \Phi_{\varepsilon} (\rho ) = \diver \left( \mobee (\rho ) \nabla V \right)
        \end{equation*}
        in the weak sense. Let us now pick $w = \Phi_\varepsilon (\rho)$.
        Since $\rho$ is a weak solution of \eqref{eq:the problem regularised} and $\Phi_\ee$ fulfils the assumption \eqref{eq:properties of Phi_ee} we know that $w \in H^1(\Omega) \cap L^\infty(\Omega)$.
        Therefore, $w$ satisfies the equation
        $$
            - \Delta w = \diver \left( f_\ee(w) \nabla V \right) , \qquad f_\ee(w) = \mobee (\Phi_\ee^{-1}(w) )
        $$
        with no-flux boundary condition.
        The right-hand side is in $L^2 (\Omega)$, and so $w \in H^2_{loc}(\Omega)$.
        Notice that $f_\ee(0) = 0$.
        By a bootstrap argument $w \in W_{loc}^{2,\infty} (\Omega)$ (see e.g., \cite{Caffareli_Cabre95}).
        So $\rho \in W_{loc}^{2,\infty} (\Omega)$, and it is a classical solution of the interior equation.

        Assume there is a set of positive measure such that $\rho > 0$.  
        Let $A \subset \Omega$ be the largest set where this is satisfied.
        Since $\rho$ is continuous, then $A$ is an open set.
        Assume, towards contradiction, that $A \subsetneq \Omega$.
        Take a point in $x_0 \in \Omega \cap \partial A$. Let $r$ be small enough that $B(x_0, 4r) \subset \Omega$.
        Notice that $u = \rho$ satisfies the elliptic problem
        \[
            -\diver(a(x) \nabla u + b(x) u )  = 0, \qquad \text{in } \Omega,
        \]
        where $a(x) = \Phi_\ee'(\rho(x)) \ge c(\varepsilon)$ and bounded, and
        \begin{equation*}
            b(x) = \begin{dcases}
                \frac{\mob_\ee(\rho)}{\rho} \nabla V & \text{if } \rho(x) > 0, \\
                \mobee'(0)\nabla V                   & \text{if } \rho(x) = 0
            \end{dcases}
        \end{equation*}
        is continuous.
        Hence, we can use Harnack's inequality (see \cite{GT01}). Then
        \[
            \sup_{ B(x_0, r) } \rho \le C \inf_{ B(x_0, r) } \rho = 0.
        \]
        This contradicts the definition of $x_0$.
        Similarly, taking $u = \alpha - \rho$ we deduce $\rho \equiv \alpha$ in $\Omega$ or $\rho < \alpha$ in $\Omega$.

        \step[{Characterisation of the case $0 < \rho < \alpha$}]
        Let us take $\delta \in (0,1)$.
        Notice that $U_\ee'' = \Phi_\ee' / \mobee$ is singular at $0$ and $\alpha$. For this proof we consider a smoothing of $U_\ee$.
        We define $U_{\ee, \delta}$ such that $U_{\ee, \delta} (\delta) = U_\ee(\delta)$, $U_{\ee,\delta}'(\delta)= U_\ee' (\delta)$ and
        \begin{equation*}
            U_{\ee, \delta}''(s) = \min\{ U_\ee''(s), \delta^{-1} \}.
        \end{equation*}
        Using $\varphi = U'_{\ee , \delta} (\rho) + V$ as a test function, we have that
        \begin{equation}
            \label{eq: regularity hat rho ee free energy diss}
            \begin{aligned}
                0 & = \int_\Omega \mobee(\rho) \nabla \left( U'_\ee (\rho) + V \right) \cdot \nabla \left( U'_{\ee , \delta} (\rho) + V \right) \\
                  & = \int_\Omega \mobee(\rho) U_\ee''(\rho)U_{\ee,\delta}''(\rho) |\nabla \rho|^2
                + \int_\Omega \mobee(\rho)  U_{\ee,\delta}''(\rho)   \nabla \rho \cdot \nabla V
                + \int_\Omega \mobee(\rho)  U_{\ee}''(\rho)   \nabla \rho \cdot \nabla V  + \int_\Omega \mobee(\rho) |\nabla V|^2.
            \end{aligned}
        \end{equation}
        Only the first two terms vary with $\delta$. 
        From this equality we can estimate the first term on the right-hand side. Using that $0 \le U_{\ee,\delta}'' \le U_\ee''$ we have
        \[
            0 \le \mobee(\rho) U_{\ee,\delta}''(\rho) |\nabla \rho| \le \mobee (\rho) U_\ee (\rho) |\nabla \rho| = |\nabla \Phi_\ee(\rho)|.
        \]
        Hence, this quantity is in $L^2 (\Omega)$ due to our notion of weak solution.
        Hence, we can bound
        \begin{align*}
            \int_\Omega \mobee(\rho) U_\ee''(\rho)U_{\ee,\delta}''(\rho) |\nabla \rho|^2
             & \le 2 \int_\Omega |\nabla \Phi_\ee (\rho) | |\nabla V| -  \int_\Omega \mobee(\rho) |\nabla V|^2 
            \le 2 \int_\Omega |\nabla \Phi_\ee (\rho) | |\nabla V|.
        \end{align*}
        We observe that as $\delta \to 0$ we have pointwise that
        \begin{equation*}
            \mobee(\rho) U_\ee'' (\rho) U_{\ee,\delta}'' (\rho) |\nabla \rho|^2 \nearrow \mobee(\rho) U_\ee''(\rho)^2 |\nabla \rho|^2.
        \end{equation*}
        Therefore, by the monotone convergence theorem $\mobee(\rho) |\nabla U_\ee'(\rho)|^2 \in L^1$ as we have convergence as $\delta \to 0$ of the first term in \eqref{eq: regularity hat rho ee free energy diss}.
        Since $\mobee(\rho) U_{\ee, \delta}''(\rho) |\nabla \rho|$ is uniformly bounded in $L^2(\Omega)$, a subsequence admits a weak-$L^2$ limit.
        By point-wise convergence we deduce that, up to a subsequence, $\mob_\ee(\rho) U_{\ee,\delta}''(\rho) \nabla \rho \rightharpoonup \mob_\ee(\rho) U_{\ee}''(\rho) \nabla \rho$ weakly in $L^2(\Omega)$ as $\delta \to 0$.
        We can therefore pass to the limit in \eqref{eq: regularity hat rho ee free energy diss} to deduce
        \[
            \int_\Omega \mobee(\rho) \Big|\nabla (U'_\ee (\rho) + V )\Big|^2 = 0.
        \]
        Since $0 < \rho < \alpha$ a.e.~in $\Omega$, we deduce $U'_\ee (\rho) + V = C_\ee$ and, in particular, $\rho = (U_\ee')^{-1} (C_\ee - V(x))$.
        Using \Cref{thm:Euler-Lagrange} the constant is uniquely defined.
        And the proof is complete. \qedhere
    \end{enumeratesteps}
\end{proof}

Any constant-in-time weak solution of the problem \eqref{eq:the problem regularised} is described in \Cref{lem:Regularity hat rho ee}. The combination of this result with \Cref{thm:long-time behaviour} yields that the \timeLimit{}\, $S_\infty$ for the problem \eqref{eq:the problem regularised} is such that $S_\infty^{(\ee)} \rho_0 = 0$, $S_\infty^{(\ee)} \rho_0= \alpha$ or $S_\infty^{(\ee)} \rho_0 (x) = (U_\ee')^{-1} (C_\ee - V(x))$, where the constant $C_\ee$ depends only on the mass of $\rho_0$, i.e., the only possible constant-in-time weak solutions. Hence, as we point out in the statement,  \ref{item:Pee long time 2} follows as a consequence of \Cref{lem:Regularity hat rho ee}.

Using this result, we are now ready to study further properties of $\widehat{\rho}^{(\ee)}$.

\begin{proof}[Proof of \ref{item:Pee long time 1} and \ref{item:Pee long time 3}]
    The fact that it is the unique $L^1$-local minimiser follows directly from \Cref{thm:Euler-Lagrange}. Combining this with \Cref{prop:Free energy bound from below} we obtain that it is also the unique global minimiser.

    Furthermore, from \Cref{thm:long-time behaviour} and \Cref{lem:Regularity hat rho ee} we have that for any $\rho_0 \in \mathcal{A}$
    \begin{equation*}
        S_t^{(\ee)} \rho_0 \rightarrow S_\infty^{(\ee)} \rho_0 = \rho^{(\ee), \infty} \quad \text{strongly in } L^1(\Omega) \text{ as } t \rightarrow \infty,
    \end{equation*}
    finishing the proof.
\end{proof}

\subsection{The \texorpdfstring{$\omega$}{omega}-limit of
    \texorpdfstring{\eqref{eq:the problem Omega}}{(P)}. Proof of \texorpdfstring{\Cref{thm:Asymptotic (P0)}}{the steady states for (P)}}
\label{sec:Stationary state}

We proceed to study the limit of the constant-in-time weak solutions of \eqref{eq:the problem regularised} as $\ee \rightarrow 0$. In order to do that and study its properties we present some auxiliary results.

\begin{lemma}\label{lem:Limit ee to 0}
    Assume \eqref{eq:U uniformly strictly convex}. Consider the approximation $U_{\ee}$ defined with the condition \eqref{eq:properties of Phi_ee}.
    Then, the functions $(U_{\ee}')^{-1}$ are such that $(U_{\ee}' )^{-1} \rightarrow \inversedU$ pointwise in $\R$ and uniformly over compacts of $\R \backslash \left\lbrace \underline{\zeta} , \overline{\zeta} \right\rbrace$ as $\ee \searrow 0$, where $\underline{\zeta}, \overline{\zeta}$ are defined in \eqref{eq:defn xi underline and overline}.
\end{lemma}

\begin{proof}
    We recall that \Cref{lem:Uee converges to U} ensures
    $U_{\ee}' \rightarrow U'$ pointwise in $(0,\alpha)$.
    We will use the notations \eqref{eq:defn xi underline and overline}.
    In order to show the convergence, we divide the proof into the sets $(\underline{\zeta} , \overline{\zeta})$, $\left\lbrace \underline{\zeta} , \overline{\zeta} \right\rbrace$, and $\R \backslash [\underline{\zeta} , \overline{\zeta}]$.
    \begin{enumeratesteps}

        \step[Convergence in $(\underline{\zeta} , \overline{\zeta})$]
        We first prove that $(U_\ee')^{-1}$ are uniformly Lipschitz over compacts.
        We know that it is non-decreasing. So the derivative is bounded below.
        Take any $\zeta_1 , \zeta_2 \in (\underline{\zeta} , \overline{\zeta} )$ with $\zeta_1 < \zeta_2$. By construction $0 < U'(\zeta_1) \le U'(\zeta_2) < \alpha$. Let us define
        \begin{align*}
            s_1 \coloneqq \frac{(U')^{-1} (\zeta_1 )}{2} \in \left(0, U'(\zeta_1) \right)  \quad \text{and} \quad s_2 \coloneqq \frac{(U')^{-1} (\zeta_2) + \alpha }{2} \in \left(U'(\zeta_2) , \alpha \right).
        \end{align*}
        Due to uniform convergence over compacts (see \Cref{lem:Uee converges to U}), for $\ee$ small enough, $(U_{\ee}')^{-1} ([\zeta_1 , \zeta_2]) \subseteq [s_1, s_2 ]$. We pick $\zeta \in [\zeta_1 , \zeta_2 ]$.
        Recalling \eqref{eq:properties of Phi_ee}, we observe that
        \[
            \left( (U_{\ee}')^{-1} \right)' (\zeta ) = \frac{1}{U_{\ee}'' \left( (U_{\ee}')^{-1} (\zeta ) \right)} = \frac{\mobee\left( (U_{\ee}')^{-1} (\zeta ) \right)}{\Phi_\ee' \left( (U_{\ee}')^{-1} (\zeta ) \right)} \leq \frac{\mobee\left( (U_{\ee}')^{-1} (\zeta ) \right)}{\underline{\Phi_\ee'} \left( (U_{\ee}')^{-1} (\zeta ) \right)}.
        \]
        We distinguish two cases:

        \begin{enumeratesteps}
            \step[$\Phi_\ee'((U_\ee')^{-1}(\zeta)) \ge \kappa(\ee)^{-1}$]
            In this case we simply estimate
            \[
                \left( (U_{\ee}')^{-1} \right)' (\zeta ) \leq \frac{\mobee\left( (U_{\ee}')^{-1} (\zeta ) \right)}{\kappa(\ee)^{-1} + \ee} \le \kappa(\ee) \max_{\ee \in [0,1]} \| \mobee \|_{L^\infty(0,\alpha)}.
            \]
            \step[$\Phi_\ee'((U_\ee')^{-1}(\zeta)) < \kappa(\ee)^{-1}$]
            Then we have that
            \begin{align*}%\label{eq:Inverse Uee' Lipschitz (i)}
                \begin{split}
                    \left( (U_{\ee}')^{-1} \right)' (\zeta )
                     & \le \frac{\mobee\left( (U_{\ee}')^{-1} (\zeta ) \right)}{\mob\left( (U_{\ee}')^{-1} (\zeta ) \right) U'' \left( (U_{\ee}')^{-1} (\zeta ) \right) + \ee}
                    \le  \frac{\max_{\ee \in [0,1]} \| \mobee \|_{L^\infty(0,\alpha)}}{\min_{s\in [s_1,s_2]} \mob\left( s \right)  U'' \left( s \right)}.
                \end{split}
            \end{align*}
            The denominator is positive due to \ref{hyp:mobility} and \eqref{eq:U uniformly strictly convex}. 
        \end{enumeratesteps}
        Thus, we have show that $(U_{\ee}')^{-1}$ are uniformly Lipschitz over compacts of $(\underline{\zeta} , \overline{\zeta} )$.
        Now we prove point-wise convergence.
        Let $\zeta \in (\underline{\zeta}, \overline{\zeta})$ and
        $\delta > 0$ small enough so that $\zeta \in (\underline{\zeta} + \delta, \overline{\zeta} - \delta)$.
        Due to continuity, $s \coloneqq \inversedU(\zeta) \in (0,\alpha)$.
        Let $\zeta_\ee \coloneqq U_\ee'(s)$.
        By \Cref{lem:Uee converges to U}, there exists $\ee_0 > 0$ such that for $\ee < \ee_0$ we have $\zeta_\ee \in (\underline{\zeta} + \delta, \overline{\zeta} - \delta)$.
        In the previous step we have shown that $(U_\ee')^{-1}$ are uniformly Lipschitz in $(\underline{\zeta} + \delta, \overline{\zeta} - \delta)$. Letting $L(\delta)$ be the continuity constant we have
        \begin{equation}\label{eq:Inverse pointwise convergence}
            | (U_{\ee}')^{-1} ({\zeta} ) - \inversedU ({\zeta} ) |  =  | (U_{\ee}')^{-1} ({\zeta} ) - (U'_{\ee})^{-1} (\zeta_{\ee} ) | \le L(\delta) |\zeta - \zeta_\ee| \to 0 \quad \text{as } \ee \to 0.
        \end{equation}
        Since this holds for any $\zeta \in (\underline \zeta, \overline \zeta)$, point-wise convergence holds.
        Joining this fact with the uniformly Lipschitz continuity over compacts of $(\underline{\zeta} , \overline{\zeta} )$, we obtain uniform convergence over compact sets of $(\underline{\zeta} , \overline{\zeta} )$.

        \step[Point-wise convergence at $\left\lbrace \underline{\zeta} , \overline{\zeta} \right\rbrace$]
        Let us now assume $0 < \oneconstant < \overline{\zeta} - \underline{\zeta}$. Then, due to monotonicity it follows that
        \begin{equation*}
            0 \le \limsup_{\ee \to 0^+ }(U_{\ee}')^{-1} (\underline{\zeta} ) \leq \limsup_{\ee \to 0^+ }(U_{\ee}')^{-1} (\underline{\zeta} + \oneconstant) = \inversedU(\underline{\zeta} + \oneconstant) .
        \end{equation*} 
        Letting then $\oneconstant \to 0$ we get $(U_{\ee}')^{-1} (\underline{\zeta} ) \to 0 = \inversedU(\underline{\zeta}) $.
        Analogously,
        $ (U_{\ee}')^{-1} (\overline{\zeta}) \rightarrow \alpha = \inversedU (\overline{\zeta}).
        $

        \step[Uniform convergence on $(-\infty , \underline{\zeta}) $ and $(\overline{\zeta} , \infty )$]
        Finally, for the set $(-\infty , \underline{\zeta})$ we can compute, using that $U_\ee'(\zeta) \in (0,\alpha)$ for all $\zeta \in \mathbb R$ and $\inversedU (\zeta) = 0$ for $\zeta \le \underline \zeta$ we have that
        \begin{align*}
            \sup_{\zeta < \underline{\zeta}} \left| ( U_{\ee}')^{-1} (\zeta ) - \inversedU (\zeta) \right| = \sup_{\zeta < \underline{\zeta}} \, (U_{\ee}')^{-1} (\zeta )  \leq (U_{\ee}')^{-1} (\underline{\zeta} ) \rightarrow 0 \text{ as } \ee \to 0.
        \end{align*}
        Analogously for the set $( \overline{\zeta}, \infty )$.
        \qedhere
    \end{enumeratesteps}
\end{proof}

Now we are ready to prove convergence of $\widehat\rho^{(\ee)}$ as $\ee \rightarrow 0$ to a stationary weak solution of \eqref{eq:the problem Omega}.

\begin{lemma}\label{lem:Convergence Asymptotics of (Pee)}
    Assume $M \in (0,\alpha|\Omega|)$ and \eqref{eq:U uniformly strictly convex}.
    Let $C_\ee$ and $C_0$ be given by \Cref{thm:Euler-Lagrange}.
    Then, as $\ee \to 0$
    $C_\ee \to C_0$ and
    \begin{equation*}
        \widehat\rho^{(\ee)} = \inversedUee( C_\ee - V) \rightarrow \widehat{\rho}^{(0)} \coloneqq \inversedU (C_0 -V) \quad \text{in } L^1(\Omega).
    \end{equation*}
    Furthermore, $\widehat\rho^{(0)}$ is a constant-in-time weak solution to \eqref{eq:the problem Omega}.
    If $S$ is the \goodSemigroup{} coming from \Cref{thm:left side of D}, then $S_t \widehat\rho^{(0)} = \widehat\rho^{(0)}$.
\end{lemma}

\begin{proof}
    We divide the proof into several steps.
    \begin{enumeratesteps}
        \step[{The set $\left\lbrace C_{\ee} \, : \, 0 < \ee < 1 \right\rbrace$ is bounded}] Let us argue by contradiction. Assume there exists a sequence $C_{\ee_k}$ such that $C_{\ee_k} \searrow - \infty$. Then, for every $\zeta \in \R$ there exists a constant $N_0 (\zeta )$ such that
        \begin{equation*}
            C_{\ee_k} - V(x) < \zeta \qquad \text{for all } x \in \overline{\Omega}, \, k \geq N_0 (\zeta ).
        \end{equation*}
        Hence, for $k \geq N_0 (\zeta )$,
        \begin{equation*}
            M = \int_{\Omega} \widehat{\rho}^{(\ee_k)} \leq | \Omega | \left( U_{\ee_k}' \right)^{-1} (\zeta ) .
        \end{equation*}
        Due to \Cref{lem:Limit ee to 0} we pass to the limit in $k \to \infty$ to recover
        \begin{equation*}
            M \le |\Omega| \, T_{[0,\alpha]} \circ (U')^{-1} (\zeta).
        \end{equation*}
        As $\zeta \to -\infty$ we recover $M = 0$, a contradiction.
        Therefore, the set $\left\lbrace C_{\ee} \, : \, 0 < \ee < 1 \right\rbrace$ is bounded from below.

        Similarly, using that $\lim_{\zeta \rightarrow \infty } \left( U_{\ee_k}' \right)^{-1} (\zeta ) = \alpha$, and arguing by contradiction, we are able to prove that $\left\lbrace C_{\ee} \, : \, 0 < \ee < 1 \right\rbrace$ is bounded from above.

        \step[Convergence] Since the set $\left\lbrace C_{\ee} \, : \, 0 < \ee < 1 \right\rbrace \subset \mathbb R$ is bounded there is a sequence $C_{\ee_k}$ and a constant $\widetilde{C}$ such that $C_{\ee_k} \rightarrow \widetilde{C}$. In the following we will prove that $\widehat{\rho}^{(\ee_k)} \rightarrow \inversedU (\widetilde{C} - V )$ in $L^1 (\Omega)$.

        First, we prove point-wise convergence, we separate the domain into two subsets
        \begin{align*}
            A_1 & := \left\lbrace x \, : \, \widetilde{C} - V(x) \neq \underline{\zeta},  \overline{\zeta} \right\rbrace , \qquad
            A_2 := \Omega \setminus A_1.
        \end{align*}
        For $x \in A_1$, \Cref{lem:Limit ee to 0} implies uniform convergence in a neighbourhood of $\widetilde{C} - V(x)$, and in particular it also implies
        \begin{equation*}
            \widehat{\rho}^{(\ee_k)} (x) \rightarrow \inversedU (\widetilde{C} - V(x)).
        \end{equation*}
        For $x \in A_2$ we use a monotonicity argument. Let us take any $\oneconstant > 0$ such that $\widetilde C - V(x) \pm \lambda \ne \underline{\zeta}, \overline{\zeta}$. Then, there exist $\oneconstant_\ee \to \oneconstant$ such that $C_\ee - V(x) \pm \oneconstant_\ee \ne \underline{\zeta}, \overline{\zeta}$. We observe
        \begin{equation*}
            (U_{\ee_k}')^{-1} (C_{\ee_k} - V(x) - \oneconstant_{\ee_k} ) \leq \widehat{\rho}^{(\ee_k)} (x) \leq (U_{\ee_k}')^{-1} (C_{\ee_k} - V(x) + \oneconstant_{\ee_k} ).
        \end{equation*}
        Now, let $k \rightarrow \infty$ We recover that
        \begin{equation*}
            \inversedU ( \widetilde{C} - V(x) - \oneconstant ) \leq \liminf_{k \rightarrow\infty} \widehat{\rho}^{(\ee_k)} (x) \leq \limsup_{k \rightarrow\infty} \widehat{\rho}^{(\ee_k)} (x) \leq \inversedU ( \widetilde{C} - V(x) + \oneconstant ) .
        \end{equation*}
        The function $\inversedU$ is continuous. Therefore, as $\oneconstant \to 0^+$, we have that
        \begin{equation*}
            \lim_{k \rightarrow\infty} \widehat{\rho}^{(\ee_k)} (x) = \inversedU (\widetilde{C} - V(x) ) ,
        \end{equation*}
        from where it follows pointwise convergence in $\Omega$. Furthermore, since $0 \leq \widehat{\rho}^{(\ee_k)} \leq \alpha$, by the Dominated Convergence Theorem it follows that
        \begin{equation*}
            \widehat{\rho}^{(\ee_k)} \rightarrow \inversedU (\widetilde{C} - V ) \quad \text{in } L^1 (\Omega) .
        \end{equation*}
        Due to conservation of mass and \Cref{thm:Euler-Lagrange}, we have that $\widetilde C = C_0(M)$.

        \step[Convergence of the whole sequence]
        Every sequence $C_{\ee_k}$ has a convergent subsequence and the limit is unique, so $C_{\ee} \rightarrow C_0$ and
        \begin{equation*}
            \widehat{\rho}^{(\ee)} \rightarrow \inversedU (C_0 - V) \quad \text{in } L^1(\Omega ).
        \end{equation*}

        \step[$\widehat{\rho}^{(0)}$ is a constant-in-time weak solution]
        Due to \Cref{lem:L1 stability of weak solutions}, $\widehat{\rho}^{(0)}$ is a weak solution of the problem \eqref{eq:the problem Omega} and it does not depend on time.

        \step[$\widehat{\rho}^{(0)}$ is stationary for the semigroup]
        Given that $S^{(\ee_k)}_t \widehat \rho^{(\ee_k)} = \widehat \rho^{(\ee_k)}$, due to the $L^1(\Omega)$ convergence, as  $k \to \infty$ we get $S_t \widehat \rho^{(0)} = \widehat \rho^{(0)}$.
        \qedhere
    \end{enumeratesteps}
\end{proof}

We conclude this section with the proof of the main theorem.

\begin{proof}[Proof of \Cref{thm:Asymptotic (P0)}]
    \ref{item:Semi-group local minimiser} and \ref{item:Limit of steady states} follow as a consequence of \Cref{lem:Convergence Asymptotics of (Pee)}.
    Lastly, to prove
    \ref{item:minimiser free energy},
    from \Cref{thm:Euler-Lagrange} it follows that $\widehat{\rho}^{(0)}$ is the unique $L^1$-local minimiser of the free energy \eqref{eq:Free eenrgy ee=0 Omega}. Moreover, if we combine this result with \Cref{prop:Free energy bound from below} we also have that $\widehat{\rho}^{(0)}$ is the unique global minimiser.
\end{proof}

\subsection{Examples with several steady states with non-trivial basin of attraction}\label{sec:malicious counterexamples}

During this subsection we prove that the diagram \eqref{Diagram} presented in \Cref{sec:Main results} is not commutative. In order to do that, we construct a counterexample. We take $U(s) = \frac{1}{m-1} s^m$ with $m > 1$ and $V$ a double well potential. First, we explain the counterexample for the case with no saturation and linear mobility, and then we generalise it to the problem with saturation.

\subsubsection{Linear mobility}
\label{sec:Malicious example linear}
Let us discuss first the following example.
For $m > 1$ we consider the
famous Barenblatt solution of mass $M$
\[
    \mathfrak B (x, M) = \left( \frac{m-1}{m} \left(  C - \frac{|x|^2}{2}\right) \right)_+^{\frac{1}{m-1}}, \qquad \text{ where $C$ is s.t. } \int_\Rd \mathfrak B(x,M) d x = M.
\]
Since $m > 1$, $\supp \mathfrak B (\cdot, 2) = B_{\widehat R}$ for some $\widehat{R} > 0$. Take $x_1, x_2 \in \mathbb R^d$ with $|x_1 - x_2| > 2\widehat R$ and consider
$$
    \overline \rho (x) = \mathfrak B (x-x_1,2) + \mathfrak B (x - x_2,2). 
$$
We also consider
$$
    V(x) = \begin{cases}
        \frac{|x - x_i|^2}{2} & \text{if } |x - x_i|  \le {\widehat R} \\
        \frac{|x|^2}{2} & \text{if } |x| \gg 1 \\
        \text{smooth} & \text{in the intermediate regions.}
    \end{cases}
$$
For $0 \leq \rho_0 \le \overline \rho$, by the comparison principle the semigroup solution $S_t \rho_0$ is the unique weak solution to
\begin{equation} 
\label{eq:PME FP}
    \frac{\partial \rho}{\partial t} = \Delta \rho^m + \diver( \rho \nabla V )  \qquad \mathrm{in} \, (0, \infty ) \times \mathbb R^d . 
\end{equation}
Due to the comparison principle we get $0 \leq \rho_t \le \overline \rho$. 
The $\rho_t$ is also a solution to $\frac{\partial \rho}{\partial t} = \Delta \rho^m + \diver( \rho (x-x_i) )$ in each ball $B(x_i, \widehat R)$ with Dirichlet and no-flux boundary condition $(\nabla \rho^m + \rho \nabla V) \cdot \nu = 0$ in any set $\Omega$ such that $\overline \Omega \supset B(x_1, \widehat R) \cup B(x_2, \widehat R)$.

We notice that there is no mass exchange between these two balls $B(x_i, \widehat R)$, $i=1,2$. Consider $M_1, M_2 \le 1$ and $0 \leq \rho_0 \le \overline \rho$  such that
\begin{equation}\label{eq:Counterexample initial mass}
    \int_{x_1 + B_{\widehat R}} \rho_0 = M_1 , \qquad \int_{x_2 + B_{\widehat R}} \rho_0 = M_2.
\end{equation}
We show an example of such initial datum in \Cref{fig:Double_Well_V}.
Then, from \cite{Carrillo_Toscani00} it follows that
\begin{equation}\label{eq:Saddle point}
    S_\infty \rho_0 =\mathfrak B(x - x_1, M_1) + \mathfrak B(x - x_2, M_2).
\end{equation}
See \Cref{fig:Double_Well_V}.
Notice that, naturally, $0 \leq S_\infty \rho_0 \le \overline \rho$.
We show a numerical example in \Cref{fig:numerical experiment non-minimising}.

In particular, if $M_1 \ne M_2$ then $S_\infty \rho_0$ satisfies the Euler-Lagrange condition \eqref{eq:EL for ee = 0} with different constants in each component of its support. Hence, $S_\infty \rho_0$ is not an $L^1$-local minimiser of the free energy. Nevertheless, it attracts some initial data.

\begin{figure}[htb!]
    \centering
    \includegraphics[width=0.8\textwidth]{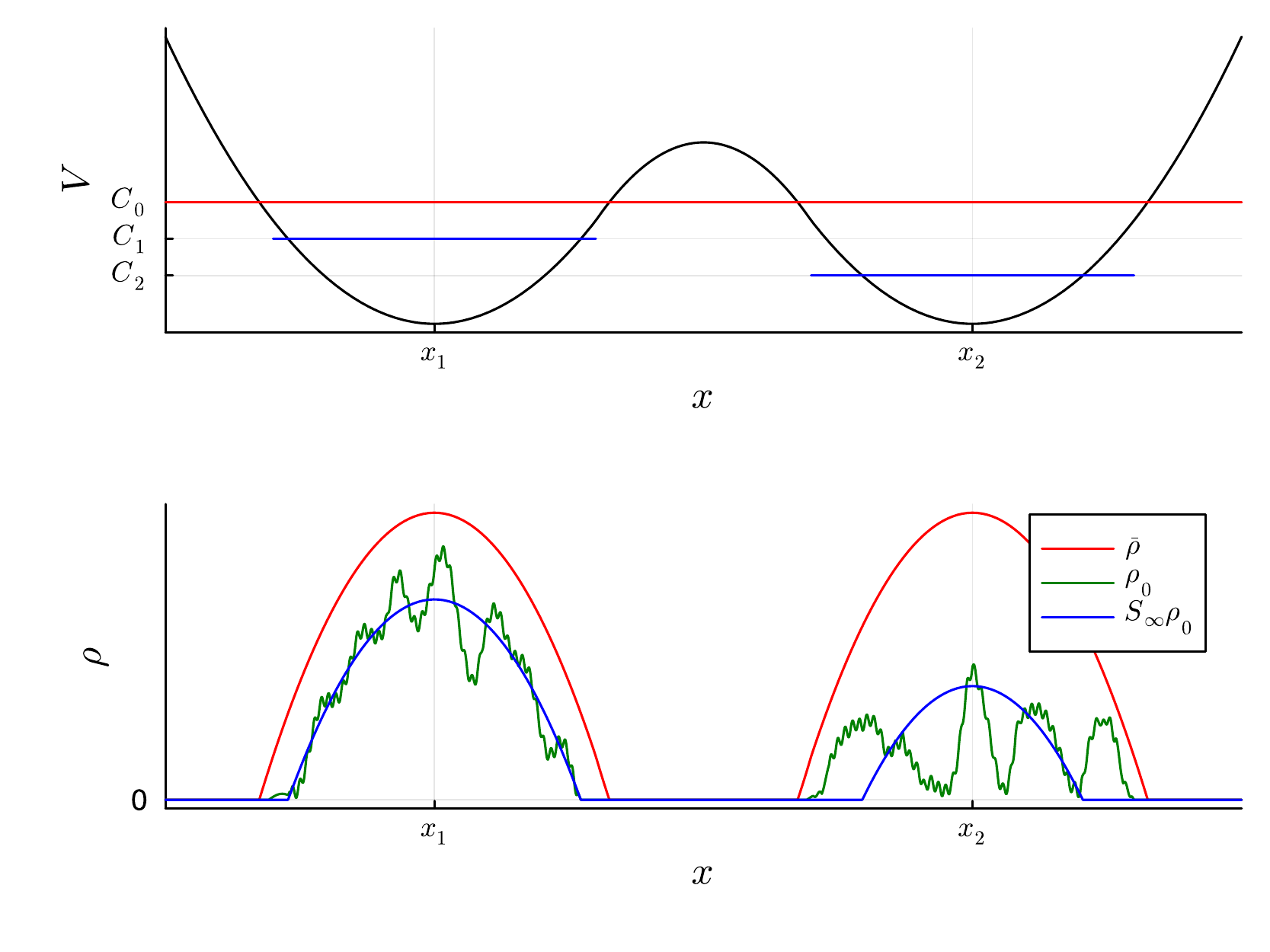}
    \caption{Double-well potential and the levels of energy. Since the initial data $\rho_0$ (green) is such that $0 \leq \rho_0 \leq \overline{\rho}$ (red), it converges to $S_\infty \rho_0$ (blue).}
    \label{fig:Double_Well_V}
\end{figure}

\begin{remark}
    This example can be generalised to $U$ such that $U'$ invertible with $U'(0) = 0$, and potentials $V$ that are uniformly convex, in the sense that 
    $$
        D^2 V \geq \lambda I
    $$
    for some $\lambda > 0$, in a region $B(x_1 , \widehat{R}) \cup B(x_2 , \widehat{R} ) \subset \overline{\Omega}$. Then, from \cite{Carrillo_McCann_Villani03}, we have that
    \[
        S_t \rho_0 \to S_\infty \rho_0 , \qquad \text{ in } \mathcal{W}_2 (\Omega) \text{ as } t \to \infty .
    \]
    Using this convergence result and the $2$-Wasserstein distance, we can reconstruct the previous example for more choices of the potential $V$. 
\end{remark}

\begin{remark}[Local minimiser in $p$-Wasserstein spaces]
    \label{rem:W infty minimisers}
    Let us make $M_1 + M_2 = 1$ and $M_1 \ne M_2$.
    Even though $\widehat \rho$ is not a $L^1$ local minimiser, or even a $2$-Wasserstein local minimiser, it is a local minimiser in the $\infty$-Wasserstein sense, see e.g., \cite{Carrillo_Delgadino_Patacchini19}.
    The $p$-Wasserstein cost of moving a mass $\mathfrak m$ from $B(x_1,\widehat R)$ to $B(x_2,\widehat R)$ can be estimated by
    $$
        \mathfrak m^{\frac 1 p} \min_{\substack{ x \in B(x_1,\widehat R) \\ y \in B(x_2, \widehat R) }} |x-y| 
        \le cost_p 
        \le \mathfrak m^{\frac 1 p} \max_{\substack{ x \in B(x_1,\widehat R) \\ y \in B(x_2, \widehat R) }} |x-y|.
    $$
    Hence, for $M$ fixed the curve of steady states
    \begin{equation}
    \label{eq:curve of non-minimising steady states}
        M_1 \in (0, M) \mapsto \mathfrak B (x-x_1,M_1) + \mathfrak B (x - x_2,M - M_1)
    \end{equation}
    is continuous in $p$-Wasserstein distance for $p \in [1,\infty)$.
    As $p \to \infty$ we get an infinite cost to pass any mass, so \eqref{eq:curve of non-minimising steady states} is not continuous in $\infty$-Wasserstein distance. We show this curve and the effect on the free energy in \Cref{fig:Mass_saddle_point}.
    The minimum is achieved for the case $M_1 = \frac{M}{2}$, i.e., $C_1 = C_2$.  
    \Cref{fig:Mass_saddle_point} also shows that (for these examples) the cases $M_1 \neq M_2$ (i.e., $C_1 \neq C_2$) are saddle points of the free energy \eqref{eq:Free eenrgy ee=0 Omega} with respect to the $L^1$ topology.
    
    Furthermore, even though $\widehat \rho$ is not a critical point of the $2$-Wasserstein gradient flow it is indeed stationary, and in fact a saddle point. We show that it has a large basin of attraction.

    \begin{figure}[htb!]
    \centering
    \includegraphics[width=1.0\textwidth]{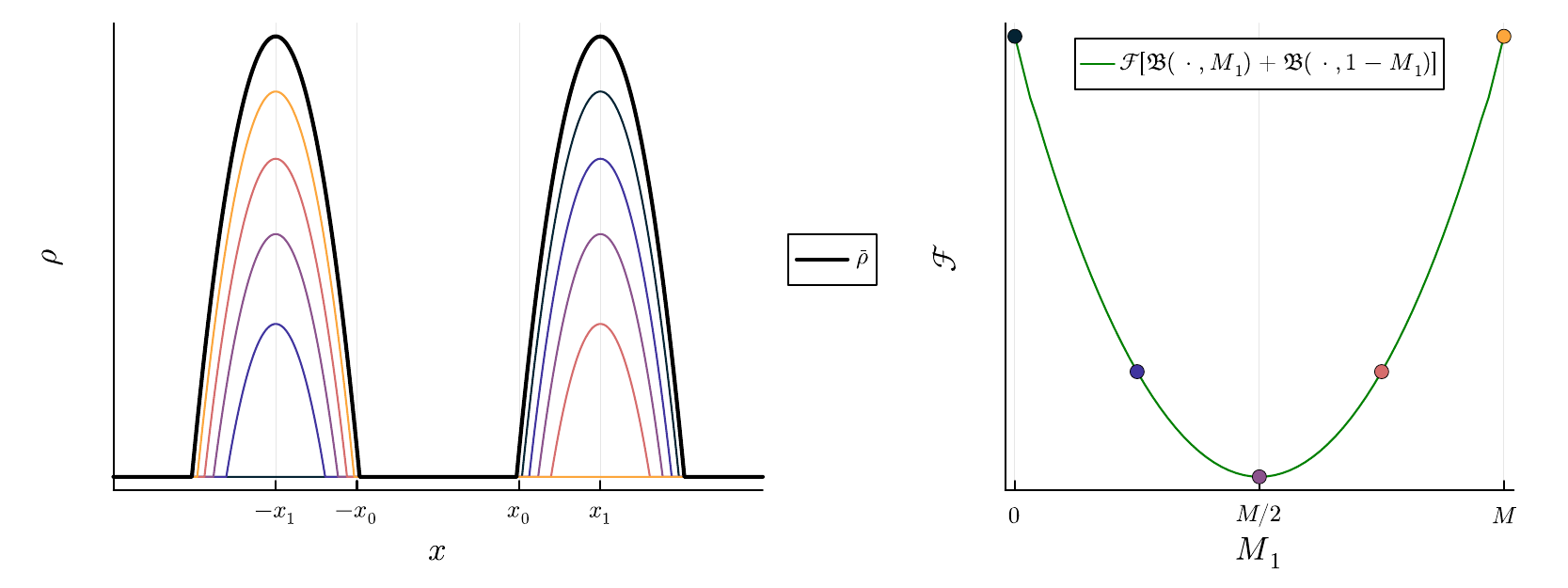}
    \caption{Free energy of \eqref{eq:Saddle point} for different values of $M_1$.}
    \label{fig:Mass_saddle_point}
\end{figure}
\end{remark}

\subsubsection{Extension to the saturation case}\label{sec:Malicious example saturation}

Using the linear mobility case, we extend the counterexample, and we show that \eqref{Diagram} is not a commutative diagram. We keep the notation from the previous subsection.

For the construction of this counter-example let us set $\alpha = 4 \sup_{x \in \overline{\Omega}} \overline{\rho}$, and the mobility $\mob (s) = \mob^{(1)}(s)\mob^{(2)}(s)$ with,
\begin{equation}\label{eq:double well mobility}
	\mob^{(1)} (s) 
	=
	\begin{dcases}
		s & \text{if } \quad 0 \leq s \leq \frac{\alpha}{4}, \\
		\alpha & \text{if } \quad \frac{3\alpha }{4}  \leq s \leq \alpha,
	\end{dcases}
        \qquad \mob^{(2)} (s) 
	=
	\begin{dcases}
		1 & \text{if } \quad 0 \leq s \leq \frac{\alpha}{4}, \\
		1 - \frac s \alpha & \text{if } \quad \frac{3\alpha}{4}   \leq s \leq \alpha,
	\end{dcases}
\end{equation}
and regular for $s \in \left(\frac{\alpha}{4}, \frac{3 \alpha}{4} \right)$. 
Notice that with the construction, if $\rho \le \overline \rho$ then $\mob(\rho) = \rho$ and thus, solving \eqref{eq:PME FP} is equivalent to solving
\begin{equation} 
\label{eq:Saturation PME FP}
    \partial_t \rho = \diver \left( \mob(\rho) \nabla \left( \frac{m}{m-1} \rho^{m-1} + V \right)\right).
\end{equation}
Taken $V$ as in the previous section,
$\overline \rho$ and $S_\infty \rho_0$ are also steady solutions of
\eqref{eq:Saturation PME FP}.
In fact, $S_\infty \rho_0$ is still an attractor for all initial data $0 \leq \rho_0 \le \overline \rho$ satisfying the mass condition \eqref{eq:Counterexample initial mass}.

Finally, let us remark that we can also study this effect at $\rho=\alpha$. 
Let us consider the same construction of the non-linear mobility $\mob(s)$. If $\rho$ is a solution of \eqref{eq:Saturation PME FP} then $u = \alpha - \rho$ solves
\begin{equation}
\label{eq:Saturation PME FP on alpha}
    \partial_t u = \diver \left( \mob(u) \nabla ( U'(u) - V )\right),
\end{equation}
where $U'(s) = \tfrac{-m}{m-1} (\alpha - s)^{m-1}$.
The comparison principle for this problem is inherited through the change of variables.
We can use the sub-solution
$\underline{u} = \alpha - \overline{\rho}$ to be a sub-solution of the problem \eqref{eq:Saturation PME FP on alpha}.
If $\alpha \geq u_0 \ge \underline u$ with correct mass $M_1$ in $B(x_1, \widehat R)$ and mass $M_2$ in $B(x_2, \widehat R)$ then its limit is $u_\infty = \alpha - S_\infty \rho_0$. Once more, there are different constants in each part of the support, so it is not an $L^1$-local minimiser.

Moreover, these examples can be generalised to larger families of potentials $U$ and $V$.

\section{An upwind numerical scheme}\label{sec:Numerical Analysis}

In this Section we study the numerical method  \eqref{def:Numerical method}. For simplicity, we restrict ourselves to the case of dimension $1$. However, these results can be extended to $\Omega \subset \Rd$ with $d > 1$, see e.g., \cite{Cances_Venel23}.

We consider a variation of the numerical method from \cite{BCH23} that can be also adapted to the regularised problem, i.e., $\ee > 0$. In \cref{sec:Numerical method} we recall some of the basic properties of the method from \cite{BCH23}. Here we also study well-posedness, convergence of the regularised method \eqref{eq:problem regularised discrete} to \eqref{def:Numerical method} as $\ee \rightarrow 0$, and convergence from the discrete to continuous solutions as $\Delta \rightarrow 0$.

Afterwards, in \cref{sec:Discrete regalurised stationary state} we study the long-time behaviour. We find the global attractors of \eqref{eq:problem regularised discrete}, and we study their limit as $\ee \rightarrow 0$. Nevertheless, analogous to the phenomenon observed in the continuous case, this last limit might not be the unique constant-in-time solution \eqref{def:Numerical method}. In  \cref{sec:Num counterexample}, applying the same strategy from \cref{sec:malicious counterexamples}, we construct examples of a different constant-in-time solution of \eqref{def:Numerical method} with a large basin of attraction. Therefore, the back face of the diagram \eqref{Numeric Diagram} is not commutative for the numerical method \eqref{def:Numerical method} and its regularisation \eqref{eq:problem regularised discrete} either.

Finally, in \cref{sec:Numerical experiments} we perform some numerical experiments.

\subsection{The numerical method. Presentation and analysis}\label{sec:Numerical method}

In \cite{BCH23} the authors propose the scheme \eqref{def:Numerical method} with $\mob^{(1)} (s) = s$, and $U$ locally bounded. In this subsection we present and recall some of the properties of the scheme. For a variation of the method that suits our case better, we obtain existence, uniqueness, and a comparison principle. Later, we also show convergence of solutions of \eqref{eq:problem regularised discrete} to a solution of \eqref{def:Numerical method} as $\ee \rightarrow 0$. This is all included in \Cref{thm:scheme well-posedness}.
Furthermore, under high regularity of the continuous solutions, we show that the discrete solution of \eqref{eq:problem regularised discrete} and \eqref{def:Numerical method} converges to the one of the continuous case \eqref{eq:the problem regularised} and \eqref{eq:the problem Omega} respectively as $\Delta\rightarrow 0$, \Cref{thm:Disc to continuous}.
Therefore, this section contains the analysis of the left face of the diagram \eqref{Numeric Diagram}, i.e., the results concerning the numerical scheme at finite time. 
Let us first prove the lemma on the decomposition of $\mob$
\begin{proof}[Proof of \Cref{lem:existence of decomposition}]
    Our problem is equivalent to decomposing $\log \mob$ as a the sum of a non-decreasing and non-increasing function. Since $\frac{\diff}{\diff s} \log \mob (s) = \frac{\mob'(s)}{\mob(s)}$ we can simply expand the Fundamental Theorem of Calculus
    \begin{equation*}
        \log \mob (s)
        =
        \log \mob(\tfrac{\alpha}{2}) +
        \int_{\frac \alpha 2}^s \frac{(\mob'(\sigma))_+}{\mob(\sigma)} \diff \sigma +
        \int_{\frac \alpha 2}^s \frac{(\mob'(\sigma))_-}{\mob(\sigma)} \diff \sigma
    \end{equation*}
    Therefore, we can write
    \begin{equation*}
        \mob^{(1)}(s) = \mob(\tfrac{\alpha}{2}) \exp\left(\int_{\frac \alpha 2}^s \frac{(\mob'(\sigma))_+}{\mob(\sigma)} \diff \sigma  \right)  ,
        \qquad
        \text{and}
        \qquad
        \mob^{(2)}(s) = \exp\left( \int_{\frac \alpha 2}^s \frac{(\mob'(\sigma))_-}{\mob(\sigma)} \diff \sigma \right).
    \end{equation*}
    Under the additional assumption, let the neighbourhoods be $(0,\delta)$ and $(\alpha - \delta, \alpha)$.
    Due to the assumption
    $\frac{\diff m^{(1)}}{\diff s} = 0$ in $(\alpha - \delta, \alpha)$ and $\frac{\diff m^{(2)}}{\diff s} = 0$ is constant in $(0,\delta)$.
    We can write
    \begin{equation*}
        \frac{\diff m^{(1)}}{\diff s} (s) = \mob^{(1)}(s) \frac{(\mob'(s))_+}{\mob(s)} = \frac{(\mob'(s))_+}{\mob^{(2)}(s)}
    \end{equation*}
    In $(0,\delta)$ we have $\frac{\diff \mob^{(1)}}{\diff s} (s) = \mob'(s) / \mob^{(2)}(0^+)$, so it is bounded. Similarly for $\mob^{(2)}$.
\end{proof}

In \cite[Proposition 2.2]{BCH23}, the authors prove that the scheme preserves boundedness and non-negativity, 
i.e., $0 \le \rho^n_i \le \alpha$. 
We show this also holds for our scheme.

We write the system \eqref{def:Numerical method} as $H(\rho^{n+1}) = \rho^{n}$ where
\begin{equation}\label{def:Numerical method step}
    H_i(\rho) = \rho_i + \Delta t \frac{F_{i+\frac{1} 2 }(\rho) - F_{i-\frac{1} 2 }(\rho)}{\Delta x}
\end{equation}
We show first a continuous dependence result and comparison principle.
\begin{lemma}[Continuous dependence]
    \label{lem:scheme comparison}
    Let $\underline\rho, \overline \rho \in \mathcal A_{\Delta}$,
    and that for some
    $H : Q \to \mathbb R^{|I|}$
    where
    \[
        Q = \Big \{ \rho \in \mathbb R^{|I|} : \min\{\underline\rho_i ,\overline \rho_i \} \le \rho_i \le \max\{ \underline\rho_i , \overline \rho_i \} \Big \}
    \]
    and $H$ is such that it satisfies:
    \begin{enumeratetheorem}
        \item
        We have mass conservation in the sense that
        $
            \sum_{i} H_i(\rho)= \sum_i \rho_i
        $
        for all
        $\rho \in Q.$

        \item
        \label{item:H monotonicity}
        In the weak sense we have the monotonicity condition:
        $
            \frac{\partial H_i}{\partial s_j} (\rho) \le 0 $  for all $j \ne i$ and $\rho \in Q.$
    \end{enumeratetheorem}
    Assume that $\underline\rho, \overline \rho$ are almost solutions in the sense that
    \[
        |H_i( \underline\rho) - \underline f_i| \le \underline g_i, \qquad |H_i(\overline \rho) - \overline f_i| \le \overline g_i .
    \]
    Then, we have that
    \begin{equation}
        \label{eq:scheme L1 continuous dependence}
        \sum_i |\underline\rho_i - \overline \rho_i| \le \sum_i |\underline f_i - \overline f_i| + 2 \sum_i \max\{ \underline g_i, \overline g_i \},
    \end{equation}
    and
    \[
        \sum_i (\underline\rho_i - \overline \rho_i)^+ \le \sum_i (\underline f_i - \overline f_i)^+ + 2 \sum_i \max\{\underline g_i, \overline g_i \}.
    \]
\end{lemma}
\begin{proof}
    Due to the conservation of mass we have
    $
        \sum_i \frac{\partial H_i}{\partial s_j} = 1,
    $
    so $\frac{\partial H_j}{\partial s_j} \ge 0$.
    Abusing slightly the notation, we write $H_i (\rho) = H_i (\rho_i, (\rho_j)_{j\ne i})$. With the monotonicity indicated
    \[
        H_i \Big(\underline\rho_i, (\min\{ \underline\rho_j , \overline \rho_j \} )_{j\ne i}\Big) \ge H_i (\underline\rho) \ge  \underline f_i - \underline g_i, \qquad  H_i \Big(\overline \rho_i, (\min\{ \underline \rho_j , \overline \rho_j \} )_{j\ne i}\Big) \ge H_i (\overline \rho) \ge \overline  f_i - \overline g_i.
    \]
    We have that
    \[
        H_i \Big((\min\{ \underline\rho_j , \overline \rho_j \} )_{j \in I}\Big) =
        \begin{dcases}
            H_i \Big(\underline\rho_i, (\min\{ \underline\rho_j , \overline \rho_j \} )_{j\ne i}\Big), & \text{if } \underline\rho_i \le \overline \rho_i , \\
            H_i \Big(\overline \rho_i, (\min\{\underline \rho_j , \overline \rho_j \} )_{j\ne i}\Big), & \text{if } \overline\rho_i \le \underline\rho_i .
        \end{dcases}
    \]
    In either case $H_i (\min\{\underline\rho,\overline \rho\}) \ge \min\{\underline f_i, \overline  f_i\} - \max \{\underline g_i, \overline g_i\} $.
    Similarly, we get $H_i (\max\{\underline \rho,\overline \rho\}) \le \max\{ \underline f_i,  \overline  f_i \} + \max \{\underline g_i, \overline g_i\}$.
    Then, we write
    \begin{align*}
        \sum_{i \in I} \left( \max\{ \underline\rho_i, \overline \rho_i \} - \min\{\underline \rho_i, \overline \rho_i \} \right)
         & = \sum_{i \in I} \Big( H_i (\max\{\underline \rho,\overline \rho\}) - H_i (\min\{\underline \rho,\overline \rho\})  \Big)                      \\
         & \le \sum_{i \in I} (\max\{\underline f_i, \overline f_i\} - \min\{\underline f_i, \overline f_i\} + 2 \max\{\underline g_i, \overline g_i\} ).
    \end{align*}
    Taking into account that $|a-b| = \max\{a,b\} - \min \{a,b\}$ we recover
    \[
        \sum_{i \in I} |\underline \rho_i - \overline \rho_i|
        \le \sum_{i \in I} |\underline f_i - \overline f_i| + 2\sum_{i \in I} \max\{\underline g_i, \overline g_i\} .
    \]
    Similarly to before, we also have due to mass conservation that
    \[
        \sum_{i \in I} (\underline \rho_i - \overline \rho_i) = \sum_{i \in I} (H_i (\underline \rho) - H_i(\overline \rho_i)) \le
        \sum_{i \in I} (\underline f_i - \overline f_i) + \sum_{i \in I} (\underline g_i + \overline g_i) .
    \]
    Using that $a^+ = \frac{|a| + a}{2}$ we deduce that
    \[
        \sum_{i \in I} (\underline \rho_i - \overline \rho_i)^+
        \le \sum_{i \in I} (\underline f_i - \overline{f}_i)^+ + 2 \sum_{i\in I} \max\{\underline g_i, \overline g_i\} . \qedhere
    \]
\end{proof}

Thanks to this result, we recover the uniqueness and the comparison principle.

\begin{lemma}[Comparison principle]
    \label{lem:discrete monotonicity}
    In the hypotheses of \Cref{lem:scheme comparison}
    if $H_i(\underline \rho) \le H_i(\overline \rho)$ for all $i$, then $\underline \rho_i \le \overline \rho_i$ for all $i$.
\end{lemma}
\begin{proof}
    Define $\underline f_i \coloneqq H_i(\underline \rho)$ and $\overline f_i \coloneqq H_i(\overline \rho)$. By assumption $\underline f_i \le \overline f_i$. Therefore, we deduce that $(\underline f_i-\overline f_i)^+ = 0$. Hence, we recover
    \[
        \sum_i (\underline \rho_i - \overline \rho_i)^+ \le 0,
    \]
    and we deduce $\overline \rho_i \ge \underline \rho_i$.
\end{proof}

In particular, the up-winding scheme that we have chosen is such that $H$ satisfies mass conservation, conservation of non-negativity, and it is monotone.

\begin{lemma}
    \label{lem:scheme is monotone}
    If we take $F$ from \eqref{def:Numerical method}, then the corresponding $H$ defined in
    \begin{equation}
        H_i(\lambda, \rho) = \rho_i + \lambda \Delta t  \frac{F_{i+\frac{1} 2 }(\rho) - F_{i-\frac{1} 2 }(\rho)}{\Delta x},
    \end{equation}
    satisfies the hypothesis of \Cref{lem:scheme comparison} for $\lambda \ge 0$.
\end{lemma}

This lemma applies also for \eqref{eq:problem regularised discrete}.

\begin{proof}
    The mass conservation follows almost immediately from the telescopic sum and the no-flux condition
    \begin{align*}
        \sum_{i \in I} H_i (\rho) = \sum_{i \in I} \rho_i + \lambda \frac{\Delta t}{\Delta x} \sum_{i \in I} \left( F_{i + \frac{1}{2}} (\rho) - F_{i - \frac{1}{2}} (\rho) \right) = \sum_{i \in I} \rho_i + \lambda \frac{\Delta t}{\Delta x} \Big(F_{N+\frac{1}{2}} (\rho) - F_{\frac{1}{2}} (\rho) \Big) = \sum_{i \in I} \rho_i.
    \end{align*}
    Let us now prove the monotonicity of $H$, by computing the derivative $\frac{\partial H_{i}}{\partial \rho_j}$.
    We proceed by taking derivatives in each term of the composition. First, we compute
    \begin{align*}
        \allowdisplaybreaks
        \frac{\partial \xi_i}{\partial \rho_j} & =
        \begin{dcases}
            U''(\rho_i) & \text{if } j = i, \\
            0           & \text{otherwise}
        \end{dcases}
        \qquad
        \frac{\partial v_{i+\frac 1 2}}{\partial \rho_j} =
        \begin{dcases}
            -\frac{ U''(\rho_{i+1})}{\Delta x} & \text{if } j = i+1, \\
            \frac{ U''(\rho_{i})}{\Delta x}    & \text{if } j = i,   \\
            0                                  & \text{otherwise}.
        \end{dcases}
    \end{align*}
    The function $F_{i+\frac 12}$ is Lipschitz (but not smoother due to the presence of the positive and negative part), but we can suitably differentiate it to obtain (using the Kronecker delta notation)
    \begin{align*}
        \frac{\partial F_{i+\frac 1 2}}{\partial \rho_j} & =
        (\mob^{(1)})' (\rho_i) \delta_{ij} \mob^{(2)}(\rho_{i+1}) ( v_{i+\frac 1 2} (\rho) )^+
        \\
                                                         & \quad + \mob^{(1)} (\rho_i)  (\mob^{(2)})'(\rho_{i+1}) \delta_{i+1,j} (v_{i+\frac 1 2} (\rho))^+
        \\
                                                         & \quad + \mob^{(1)} (\rho_i)  \mob^{(2)}(\rho_{i+1})  \sign^+(v_{i+\frac 1 2} (\rho)) \frac{\partial v_{i+\frac 1 2}}{\partial \rho_j}    \\
                                                         & \quad
        +(\mob^{(1)})' (\rho_{i+1}) \delta_{i+1,j} \mob^{(2)}(\rho_{i}) ( v_{i+\frac 1 2} (\rho) )_-
        \\
                                                         & \quad + \mob^{(1)} (\rho_{i+1})  (\mob^{(2)})'(\rho_{i}) \delta_{ij} (v_{i+\frac 1 2} (\rho))_-
        \\
                                                         & \quad + \mob^{(1)} (\rho_{i+1})  \mob^{(2)}(\rho_{i})  \sign_-(v_{i+\frac 1 2} (\rho)) \frac{\partial v_{i+\frac 1 2}}{\partial \rho_j}.
    \end{align*}
    Therefore, using $(\mob^{(1)})' \ge 0$ and $(\mob^{(2)})' \le 0$  we have that
    \[
        \frac{\partial F_{i+\frac 1 2}}{\partial \rho_i} =
        \begin{dcases}
            (\mob^{(1)})' (\rho_i) \mob^{(2)} (\rho_{i+1}) v_{i+\frac{1}{2}} (\rho) + \mob^{(1)} (\rho_i) \mob^{(2)} (\rho_{i+1}) \frac{\partial v_{i+\frac 1 2}}{\partial \rho_i} \ge 0     & \text{if } v_{i+\frac 1 2} \ge 0 , \\
            \mob^{(1)} (\rho_{i+1})  (\mob^{(2)})'(\rho_{i}) v_{i+\frac 1 2} (\rho) + \mob^{(1)} (\rho_{i+1})  \mob^{(2)}(\rho_{i})  \frac{\partial v_{i+\frac 1 2}}{\partial \rho_i} \geq 0 & \text{if } v_{i+\frac 1 2} < 0 .
        \end{dcases}
    \]
    Similarly, $\frac{\partial F_{i+\frac 1 2}}{\partial \rho_{i+1}} \le 0$. It is trivial to see that $\frac{\partial F_{i+\frac 1 2}}{\partial \rho_j}=0$ if $j \ne i, i+1$.
    Lastly, we can compute
    \[
        \frac{\partial H_i}{\partial \rho_j} =
        \begin{dcases}
            \lambda \frac{\Delta t}{\Delta x}\frac{\partial F_{i+\frac 1 2}}{\partial \rho_{i+1}} \le 0  & \text{if } j = i+1,           \\
            -\lambda \frac{\Delta t}{\Delta x}\frac{\partial F_{i-\frac 1 2}}{\partial \rho_{i-1}} \le 0 & \text{if } j = i-1,           \\
            0                                                                                            & \text{if } j \ne i-1, i, i+1.
        \end{dcases}
    \]
    This concludes the proof.
\end{proof}

\begin{lemma}[Constant sub and super solution]
    \label{lem:discrete positive sub solution}
    Assume $\ee \in [0,1]$.
    Let $F^\ee$ be given as in \eqref{eq:problem regularised discrete}. We define
    \begin{equation}
        \label{eq:H with lambda}
        H^\ee_i(\lambda, \rho) = \rho_i + \lambda \Delta t  \frac{F^\ee_{i+\frac{1} 2 }(\rho) - F^\ee_{i-\frac{1} 2 }(\rho)}{\Delta x},
    \end{equation}
    and we take $c^n \in (0,\alpha)$.
    Then, there exist $\underline c^{n+1}, \overline c^{n+1} \in (0,\alpha) $ such that for $\underline\rho = (\underline c^{n+1}, \cdots, \underline c^{n+1})$ and $\overline \rho = (\overline c^{n+1}, \cdots, \overline c^{n+1}) \in \mathbb{R}^{|I|}$ satisfy
    \[
        H_i^\ee(\lambda , \underline\rho) \le c^n\le H_i^\ee(\lambda , \overline \rho)
    \]
    for any $i \in I$, $\lambda \in [0,1]$.
\end{lemma}
\begin{proof}
    Consider $\ee \in (0,1]$.
    Notice that if $\rho_i = \underline c^{n+1}$ for all $i \in I$ then $v_{i + \frac 1 2}^\ee = - \frac{V(x_{i+1}) - V(x_i)}{\Delta x}$.
    Therefore, must only satisfy
    \begin{align*}
         & \underline c^{n+1} + \lambda (\Delta t) \mobee(\underline c^{n+1}) \frac{v_{i+\frac 1 2}^\ee - v_{i-\frac 1 2}^\ee}{\Delta x} \le c^n ,
        \quad  \text{ for } 2 \leq i \leq N-1 ,                                                                                                    \\
         & \underline c^{n+1} + \lambda (\Delta t) \mobee(\underline c^{n+1}) \frac{ v_{1+\frac 1 2}^\ee}{\Delta x} \le c^n , \qquad
        \underline c^{n+1} -
        \lambda (\Delta t) \mobee(\underline c^{n+1}) \frac{v_{N - \frac 1 2}^\ee}{\Delta x} \le c^n.
    \end{align*}
    Since we are in the discrete setting, we can simply set the equation
    \[
        \underline c^{n+1} + 2 (\Delta t) \frac{\|\nabla V\|_{L^\infty}}{\Delta x} \sup_{\ee \in (0,1]}\mobee(\underline c^{n+1})  \le c^n.
    \]
    We are trying to solve a problem of the form $F(\underline c^{n+1}) \le c^n$.
    Since the family $(\ee, s) \mapsto \mobee (s)$ is $C([0,1] \times [0,\alpha])$, $F$ is a continuous function.
    We have $F(0) = 0$ and $F(s) \ge 0$ for $s \in [0,\alpha]$.
    By Bolzano's Theorem, there exists $\underline c^{n+1} > 0$ such that $F(\underline{c}^{n+1}) \le c^n$. Thus, the result also holds for $\ee = 0$.

    For the super-solution we apply the same argument to
    \[
        \alpha - \overline c^{n+1} + \Delta t \, 2 \frac{\|\nabla V\|_{L^\infty}}{\Delta x} \sup_{\ee \in (0,1]}\mobee( \overline c^{n+1}) \le \alpha - c^n. \qedhere
    \]
\end{proof}

\begin{remark}
    Notice that this sub and super solutions do not pass to the limit when $\Delta \to 0$. The key problem is the behaviour at the endpoints.
    If we assume that $V \in  C^3(\overline \Omega)$ and $\nabla V \cdot {n} = 0$ then one can write a better equation for the sub and super solutions, depending on $\Delta V$ and $D^3 \rho$, that will, in fact, converge to the constant sub and super solutions
    \[
        \frac{d\underline c}{dt} = - \mob(\underline c) \|\Delta V \|_{L^\infty}, \qquad  \frac{d\overline c}{dt} = \mob(\overline c) \|\Delta V \|_{L^\infty}.
    \]
    Notice that if $\mob$ is not Lipschitz at $0$ or $\alpha$, there can be finite-time extinction to $0$ or $\alpha$.
\end{remark}

We can also show that the problem has energy dissipation, a simplified version of \cite[Theorem 2.4]{BCH23} for the case $W = 0$ which we include for the convenience of the reader.

\begin{lemma}[Energy dissipation]\label{lem:disc energy dissipation}
    Let $\ee \geq 0$ and $\lambda \ge 0$. Assume $\rho \in \R^{|I|}$ with $\Ediscee[\rho] < \infty$ satisfies
    $
        H^\ee(\lambda, \rho) = \rho^{\ee, n}$,
    then $\Ediscee[\rho] \leq \Ediscee [\rho^{\ee, n}]$.
\end{lemma}

\begin{proof}
    Since $U_\ee$ is convex
    \begin{align*}
        \Ediscee[\rho] - \Ediscee[\rho^{\ee, n}]
         & = \Delta x\sum_i (U_\varepsilon(\rho_i) - U_\varepsilon(\rho_i^{\ee , n})) + \sum_i V(x_i) (\rho_i - \rho_i^{\ee , n})
        \\
         & \le \Delta x\sum_i (U'_\ee(\rho_i) + V(x_i)) \frac{\rho_i - \rho^{\ee , n}_i}{\Delta t}
        = - \lambda \Delta x \sum_i \xi_i^\ee (\rho)\frac{F_{i+\frac 1 2}^\ee (\rho) - F_{i-\frac 1 2}^\ee (\rho)}{\Delta x}      \\
         & = \lambda \Delta x \sum_i \frac{\xi_{i+1}^\ee (\rho) - \xi_i^\ee (\rho)}{\Delta x} F_{i+\frac 1 2}^\ee (\rho)          \\
         & 
        = -\lambda \Delta x \sum_i \left(
        \mobee^{(1)} (\rho_i) \mobee^{(2)} (\rho_{i+1} )  (( v_{i+ \frac{1}{2}} (\rho))^+)^2 +  \mobee^{(1)} (\rho_{i+1}) \mobee^{(2)} (\rho_{i} )  (( v_{i+ \frac{1}{2}} (\rho))^-)^2 \right)
        \le 0
        . \qedhere
    \end{align*}
\end{proof}

In the next Lemma  we discuss the existence of a solution of the numerical scheme.
We recall the notion of topological degree in $\mathbb R^{|I|}$, see e.g.,  \cite{deimling1985NonlinearFunctionalAnalysis}.
The topological degree is a function $\degree :X \to \mathbb Z$ where
\[
    X = \{ (f, D, y) : D \subset \mathbb R^d \text{ is open and bounded, } f: \overline D \to \mathbb R^d \text{ continuous } , y \in \mathbb R^d \setminus f(\partial D) \}.
\]
We use three key properties. The first one is that $\degree(\mathrm{id},D, y) = 1$.
The second one is that if $h: [0,1] \times \overline D \to \mathbb R^d$ is continuous and $y \notin h(\lambda, \partial D)$ then
\[
    \degree( h(\lambda, \cdot), D, y ) \text{ is constant in } \lambda.
\]
Lastly, we use the fact that
\begin{theorem}[{\cite[Theorem 3.1]{deimling1985NonlinearFunctionalAnalysis}}]
    If $\degree(f, D, y) \ne 0$, then there exists $x \in D$  such that $f(x) = y$.
\end{theorem}
Using this tool we prove the following
\begin{lemma}[Existence for \eqref{def:Numerical method}]\label{lem:disc Existence}
    Assume that $U \in C^1([0,\alpha])$,
    and that $\rho^n \in \mathcal A_{\Delta }$.
    Then, there exists $\rho \in \mathcal A_{\Delta }$ that solves \eqref{def:Numerical method}.
\end{lemma}

\begin{proof}
    First, we consider the extension
    \[
        \widetilde {U'} (\rho) = \begin{dcases}
            U'(0)      & \text{if } s < 0            \\
            U'(s)      & \text{if } s \in [0,\alpha] \\
            U'(\alpha) & \text{if } s > \alpha
        \end{dcases}
        \qquad
        \widetilde {\mob}^{(i)} (\rho) = \begin{dcases}
            0          & \text{if } s < 0            \\
            \mob^{(i)} & \text{if } s \in [0,\alpha] \\
            0          & \text{if } s > \alpha
        \end{dcases}
    \]
    We can extend $\widetilde F_{i+\frac 1 2}$ with these definitions.
    We define
    \[
        \widetilde H_i(\lambda, \rho) = \rho_i + \lambda \Delta t \frac{\widetilde F_{i+\frac{1} 2 }(\rho) - \widetilde F_{i-\frac{1} 2 }(\rho)}{\Delta x}  \qquad \text{ for } \lambda \in [0,1] \text{ and } \rho \in \mathbb R^{|I|}.
    \]
    Now we need to pick $D$ such that $\widetilde H(\lambda, \partial D) \not \ni y$ for any $\lambda \in [0,1]$. We look at the one-parameter family of open sets
    \[
        D_R =
        \left\{ \rho \in \mathbb R^{|I|} : \sum_{i\in I}|\rho_i| < R \right\}.
    \]
    Let $\rho^{n} = \widetilde H(\lambda, \rho)$ with $\rho \in \partial D_R$.
    Because of how we have constructed the extension, we can apply \Cref{lem:scheme comparison} to $\widetilde H(\lambda, \cdot)$.
    Since $ \underline \rho= (0,\cdots, 0)$ and $ \overline \rho = (\alpha, \cdots, \alpha)$ satisfy $\rho = \widetilde H(\lambda, \rho)$ and $0 \le \rho_i^n \le \alpha$
    we conclude that $0 \le \rho_i \le \alpha$ for all $i \in I$.
    We observe
    \[
        \sum_i |\rho_i| = \sum_i \rho_i = \sum_i \widetilde H_i(\lambda, \rho) = \sum_i \rho^n_i.
    \]
    Hence, for $R > \sum_{i \in I} \rho_i^n$ it is clear that $\rho_i^n \notin \widetilde H(\lambda, \partial D_R)$ for any $\lambda > 0$.
    Therefore, we can state that
    $$
        \degree ( \widetilde H(\lambda, \cdot) , D_R, \rho^n ) = \degree ( \widetilde H(0,\cdot), D_R, \rho^n ) = \degree(\mathrm{id}, D_R, \rho^n ) =  1 \ne 0.
    $$
    Hence, there exists at least one solution $\rho \in \mathbb R^{|I|}$ of the extended problem.
    By the comparison principle $\rho \in \mathcal A_{\Delta }$.
    Due to the construction of the extension, $\rho$ solves \eqref{def:Numerical method}.
\end{proof}

\begin{lemma}[Existence under general assumptions]
    \label{lem:disc regularised Existence}
    Let $\rho^{n} \in \mathcal A_{\Delta,+}$.
    Then, there exists $\rho \in \mathcal A_{\Delta,+}$ that solves \eqref{def:Numerical method}.
\end{lemma}

Notice that this lemma applies also for \eqref{eq:problem regularised discrete}.

\begin{proof}
    We make some adaptations on the proof of \Cref{lem:disc Existence} to be able to work in the more general case $U \in C^1((0,\alpha))$.
    Consider $\delta > 0$ such that $\delta \le \rho_i^{n} \le \alpha - \delta$.
    We use \Cref{lem:discrete positive sub solution} to show the existence of $\delta_1, \delta_2 > 0$ such that
    \begin{equation}
        H_i(\lambda, (\delta_1, \cdots, \delta_1) ) \le \rho_i^n \le H_i(\lambda, (\alpha - \delta_2, \cdots, \alpha - \delta_2)) \text{ for all } \lambda \in [0,1].
    \end{equation}
    We can now apply the same reasoning as in \Cref{lem:disc Existence} using the open set
    \[
        D = \{ \rho \in \mathbb R^{|I|}: \rho_i \in (\tfrac {\delta_1} 2, \alpha - \tfrac {\delta_2} 2)   \}.
    \]
    If $\rho^{n} = H(\lambda, \rho)$ because of the sub and super solution and \Cref{lem:discrete monotonicity} we know $\rho \notin \partial D$.
    We recover a solution $\rho \in D$.
\end{proof}

We now use the combination of all these Lemmas to prove the main result.

\begin{proof}[Proof of \Cref{thm:scheme well-posedness}]
    If $\rho^0 \in \mathcal A_{\Delta , +}$ existence for \eqref{def:Numerical method} and \eqref{eq:problem regularised discrete} follows as in \Cref{lem:disc regularised Existence} and uniqueness by \eqref{lem:discrete monotonicity}.
    We due to \Cref{lem:discrete positive sub solution} we can also build uniform sub and super solutions  that they are uniform in $\ee \in [0,1]$.
    Due to compactness, any $\ee_k \to 0$ has a subsequence so that $J_{\ee_k}^\Delta \rho^0 \to u$.
    We now we use the uniform convergence over compacts set of $U_\ee' \to U'$ (\Cref{lem:Uee converges to U}) to show it is the unique solution to \eqref{def:Numerical method}, so $u = \numericStep \rho^0 $.
    Since every sequence has a convergent sub-sequence converging, and they all share the same limit, then the whole sequence $J_{\ee}^\Delta \rho^0$ converges to $\numericStep \rho^0$ as $\ee \to 0$.
    If $\rho_0 \in \mathcal A_{\Delta }$ and $U'\in C^1([0,1])$ existence follows by \Cref{lem:disc Existence} and uniqueness by \Cref{lem:discrete monotonicity}.
    In both cases we can use \Cref{lem:scheme comparison} to prove $L^1_{\Delta }$ contraction and \Cref{lem:disc energy dissipation} to show that there is free-energy dissipation.
\end{proof}

\begin{remark}\label{rem:Bailo's paper Remark}
    In \cite{BCH23} the authors have a more general setting where $\xi_i = U'(\rho_i) + V(x_i) + \sum_j W(x_i - x_j) \rho_j$.
    For bounded domains, however, it is more natural to write
    $\xi_i = U'(\rho_i) + V(x_i) + \sum_j K(x_i, x_j) \rho_j$ where $K(x_i, x_j) = K(x_j,x_i)$.
    The proof of existence is still valid.
    Our proof of uniqueness uses strongly the monotonicity of the problem (i.e., the existence of a comparison principle), which does not hold for general $W$.
    Hence, when $K \ne 0$ but it is smooth, existence and uniqueness
    can be obtained using \eqref{eq:scheme L1 continuous dependence} to prove continuous dependence with respect to $\nabla V$, and arguing by fixed-point for $\Delta t$ small.
\end{remark}

We now discuss convergence of discrete solutions of the scheme \eqref{def:Numerical method} to solutions of the continuous problem \eqref{eq:the problem Omega} (including both, $\ee > 0$ and $\ee=0$) under high regularity of the solution. In particular, this completes  the analysis of the left face of the diagram \eqref{Numeric Diagram}.

\begin{proof}[Proof of \Cref{thm:Disc to continuous}]
    We will simply check that $u_i^{ n} = \rho(t_n, x_i)$ has the correct consistency rate.
    To avoid confusion, we denote the velocity and the flux of the exact solutions by
    \[
        \mathfrak v(t,x) = - \nabla ( U(\rho) + V ), \qquad \mathfrak F = \mobnd(\rho) \mobni(\rho)  \mathfrak v.
    \]
    It is easy to see that
    \begin{equation}\label{eq:velocity discrete to continuous}
        | v_{i + \frac 1 2}^{n} - \mathfrak v(t_n, x_{i+\frac 1 2}) | \le C( [U'']_{C^\gamma} , \| \nabla \rho \|_{L^\infty} , [\nabla V]_{C^{\gamma}}) (\Delta x)^\gamma.
    \end{equation}
    Similarly, we can show that
    \begin{equation}\label{eq:acceleration discrete to continuous}
        \left| \frac{v_{i+\frac 1 2}^{n} - v_{i-\frac 1 2}^{n}}{\Delta x} - \diver \mathfrak v(t_n, x_i) \right| \le C( [U''']_{C^\gamma} , [\nabla \rho ]_{C^{1,\gamma}} , [\nabla V]_{C^{\gamma}}) (\Delta x)^\gamma .
    \end{equation}
    We write the decomposition
    \[
        \diver \mathfrak F = (\mob^{(1)})' (\rho) \mob^{(2)} (\rho) \nabla \rho \cdot \mathfrak v  + \mob^{(1)} (\rho) (\mob^{(2)})' (\rho) \nabla \rho \cdot \mathfrak v + \mob^{(1)} (\rho) \mob^{(2)} (\rho) \diver \mathfrak v.
    \]
    We separate four cases.
    First, let us consider $v_{i+\frac 1 2}^{ n}  , v_{i-\frac 1 2}^{n}  \ge 0$. Then, we have
    \begin{align*}
        \frac{F_{i+\frac 1 2}^{n} - F_{i-\frac 1 2}^{ n}}{\Delta x}
         & = \frac{1}{\Delta x} \left(  \mob^{(1)}(u^{n}_i) \mob^{(2)}(u^{n}_{i+1}) v_{i+\frac 1 2}^{ n} - \mob^{(1)}(u^{n}_{i-1}) \mob^{(2)}(u^{n}_{i}) v_{i-\frac 1 2}^{ n}) \right) \\
         & = \frac{ \mob^{(1)}(u^{n}_i) - \mob^{(1)}(u^{ n}_{i-1}) }{\Delta x} \mob^{(2)}(u^{n}_{i+1}) v_{i+\frac 1 2}^{n}
        + \mob^{(1)}(u^{n}_{i-1}) \frac{\mob^{(2)}(u^{n}_{i+1}) - \mob^{(2)}(u^{n}_{i}) }{\Delta x } v_{i+\frac 1 2}^{\ee ,n}
        \\
         & \quad + \mob^{(1)}(u^{n}_{i-1}) \mob^{(2)}(u^{n}_{i}) \frac{v_{i+\frac 1 2}^{n} - v_{i-\frac 1 2}^{n}}{\Delta x}.
    \end{align*}
    Therefore, using the regularity and the previous estimates \eqref{eq:velocity discrete to continuous} and \eqref{eq:acceleration discrete to continuous} we obtain that
    \begin{equation}
        \label{eq:consistency error}
        \left| \frac{F_{i+\frac 1 2}^{ n} - F_{i-\frac 1 2}^{ n}}{\Delta x} - \diver \mathfrak F (t_n, x_i) \right| \le C (\Delta x)^\gamma.
    \end{equation}
    Similarly, if $v_{i+\frac 1 2}^{ n} , v_{i-\frac 1 2}^{n} \le 0$.
    The third case is $v_{i + \frac 1 2}^{n} \ge 0 \ge v_{i - \frac 1 2}^{ n}$. In this setting,
    \begin{align*}
        \frac{F_{i+\frac 1 2}^{n} - F_{i-\frac 1 2}^{n}}{\Delta x} & = \frac{1}{\Delta x} \left(  \mob^{(1)}(u^{ n}_i) \mob^{(2)}(u^{n}_{i+1}) v_{i+\frac 1 2}^{ n} - \mob^{(1)}(u^{ n}_{i}) \mob^{(2)}(u^{ n}_{i-1}) v_{i-\frac 1 2}^{n} \right)                                                        \\
                                                                   & = \mob^{(1)}(u^{ n}_{i}) \frac{\mob^{(2)}(u^{ n}_{i+1}) - \mob^{(2)}(u^{ n}_{i-1}) }{\Delta x } v_{i+\frac 1 2}^{ n} + \mob^{(1)}(u^{n}_{i}) \mob^{(2)}(u^{ n}_{i-1}) \frac{v_{i+\frac 1 2}^{ n} - v_{i-\frac 1 2}^{ n}}{\Delta x}.
    \end{align*}
    Due to the change of sign and the similarity to $\mathfrak v(t_n, x_{i+\frac 1 2})$ we deduce
    $
        |\mathfrak v (t_n, x_{i+\frac 1 2})| , |v_{i+\frac 1 2}^{ n}| \le C (\Delta x)^\gamma.
    $
    Using this estimate,
    \[
        \left| \mob^{(1)}(u^{n}_{i}) \frac{\mob^{(2)}(u^{ n}_{i+1}) - \mob^{(2)}(u^{ n}_{i-1}) }{\Delta x } v_{i+\frac 1 2}^{ n}  \right| \le \mob^{(1)}(\alpha) \|(\mob^{(2)})'\|_{L^\infty} \|\nabla \rho\|_{L^\infty} C (\Delta x)^\gamma .
    \]
    The same happens in the corresponding terms of $\diver \mathfrak F$. Hence, we have \eqref{eq:consistency error}.
    The final case is $v_{i + \frac 1 2}^{ n} \le 0 \le v_{i - \frac 1 2}^{ n}$, which is analogous to the third.
    Thus, we have that
    \begin{align*}
        \left|\frac{u_i^{ n+1} - u_i^{n}}{\Delta t}
        + \frac{F_{i+\frac 1 2}^{ n} - F_{i-\frac 1 2}^{ n} }{\Delta x} \right|
         & \le \left|\frac{u_i^{ n+1} - u_i^{ n}}{\Delta t}
        - \frac{\partial \rho}{\partial t} \right|
        + \left|\frac{\partial \rho}{\partial t}
        - \diver \mathfrak F \right| + \left| \frac{F_{i+\frac 1 2}^{ n} - F_{i-\frac 1 2}^{ n}}{\Delta x} - \diver \mathfrak F^{(\ee)} (t_n, x_i) \right|
        \\
         &
        \le C( (\Delta t)^\beta + (\Delta x)^\gamma ).
    \end{align*}
    We deduce the estimate in the statement from the stability of the numerical scheme.
\end{proof}

\subsection{Long-time behaviour for the numerical method}
\label{sec:Discrete regalurised stationary state}

In this subsection we study the long-time behaviour of the numerical method. First we focus on the regularised problem \eqref{eq:problem regularised discrete} for $\ee > 0$. 

The  long-time behaviour analysis we study here is based on the gradient flow structure of the problem. We explain this in more detail in the following remark.

\begin{remark}
    Notice that we can write
    \begin{align*}
        F_{i+\frac 1 2}^{\ee , n+1}   &=\Theta_{i + \frac 1 2}^{\ee , n+1} v_{i+\frac 1 2}^{\ee , n+1} , \\
        \Theta_{i + \frac 1 2}^{\ee , n+1} &=  
        \mobee^{(1)}(\rho_i^{\ee , n+1})  \mobee^{(2)} (\rho_{i+1}^{\ee, n+1}) \sign^+(v_{i+\frac 1 2}^{\ee , n+1}) -   \mobee^{(1)} (\rho_{i+1}^{\ee , n+1} )  \mobee^{(2)} (\rho_{i}^{\ee , n+1}) \sign^- (v_{i+\frac 1 2}^{\ee , n+1}).
    \end{align*}
    From the result on energy dissipation obtained at \Cref{lem:disc energy dissipation} and the notation that we use for the problem \eqref{eq:problem regularised discrete}, it follows that 
    \begin{equation}\label{eq:discrete entropy dissipation}
        0 \leq \Delta t \Delta x \sum_{m=n}^{n+k-1} \sum_i \Theta_{i + \frac{1}{2}}^{\ee, m+1} \left| v_{i + \frac{1}{2}}^{\ee, m+1} \right|^2 \leq \Ediscee [ \rho^{\ee, n}] - \Ediscee [\rho^{\ee, n+k}].
    \end{equation}
\end{remark}

\begin{proof}[Proof of \ref{it:discrete approx unique steady state}] 
    If $H^\ee(\rho) = \rho$, the free energy dissipation states that $F^\ee_{i+\frac 1 2} (\rho) = 0$. Since $0 < \rho_i <\alpha$ for all $i\in I$, we get $v^\ee_{i+\frac 1 2} (\rho) = 0$ for all $i \in I$, and thus $\xi_i (\rho)$ is constant.  
    This concludes the proof.
\end{proof}

\begin{lemma}\label{lem:W-11 norm discrete}
    For $\rho \in \R^{|I|}$, the discrete version of the $W^{-1,1}$ norm, i.e.,
    \begin{equation}\label{def:W-11 norm discrete}
        \| \rho  \|_{ W^{-1,1} _{\Delta} (0,1) } = \inf \left\{ \Delta x \sum_{i=0}^{N} |F_{i+\frac 1 2}|  : \text{for each } i \in I \text{ we have that } \rho_i = \frac{F_{i+\frac 1 2} - F_{i-\frac 1 2}   }{\Delta x} \right\},
    \end{equation}
    defines a norm.
\end{lemma}

\begin{proof}
    The homogeneity with constants and the triangle inequality are obvious. Lastly, if $\| u  \|_{W^{-1,1} _{\Delta} } = 0$ then $F = 0$ so $u=0$. Therefore, $\| \cdot \|_{W^{-1,1} _{\Delta}}$ is a norm.
\end{proof}

All norms in $\R^{|I|}$ are equivalent, which is a great advantage of the discrete setting. Notice that, in the definition \eqref{def:W-11 norm discrete}, we have not specified that $F_{\frac 1 2}$ or $F_{N +\frac 1 2}$ vanish, but this is allowed. \Cref{lem:W-11 norm discrete} is a key step in the proof of the next result.

\begin{proof}[Proof of \ref{it:discrete approx global attractor}] 
\let\thelimit\relax 
\newcommand{\thelimit}{u}
We divide the proof in several steps.
    \begin{enumeratesteps}
    \step[The constant is determined uniquely by $\sum_i \rho^{\ee,\infty}$]
    Due to the no-flux boundary condition it follows immediately that the mass is preserved, i.e.,
    \begin{equation*}
        \Delta x \sum_{i=1}^{|I|} \rho_i^{\ee, n+1} = \Delta x \sum_{i=1}^{|I|} \rho_i^{\ee, n} = \int_0^1 \rho_0 \quad \text{for every } n.
    \end{equation*}
    Due to the convergence in $\mathbb R^{|I|}$, $\rho^{\ee,\infty}$ also satisfies the mass condition.
    Define the function
    \begin{equation*}
        P(C) = \Delta x \sum_{i=1}^{|I|} (U_\ee')^{-1} \left( C - V(x_i) \right) .
    \end{equation*}
    It is easy to see that it is  the strictly monotone, continuous, with $P(- \infty ) = - \infty$, and $P(\infty ) = \infty$. Therefore, the mass condition
    \begin{equation*}
        \Delta x \sum_{i=1}^{|I|} (U_\ee')^{-1} \left( C_\ee^{\Delta } - V(x_i) \right) = \int_{0}^1 \rho_0 \dx
    \end{equation*}
    determines $C_\ee^{\Delta }$ uniquely.

    \step[Convergence in time]
    Due to the no-flux assumption of our problem 
    \begin{equation*}
        \sum_i | \rho^{\ee , n}_i | = \sum_i | \rho^{0}_i | \quad \text{for all } n.
    \end{equation*}
    Therefore, up to a subsequence $n_j$, there exists $u \in \R^{|I|}$ such that 
    $
        \rho^{\ee, n_j} \rightarrow \thelimit \quad \text{in } \R^{|I|}.
    $
    From the construction of the norm $W^{-1,1} _{\Delta}$, the sequence $\rho^{\ee , n}$ satisfies,
    \begin{align*}
        \left\| \frac{\rho^{\ee , n+1} - \rho^{\ee ,n}}{\Delta t} \right\|_{W^{-1,1}_{\Delta} (0,1) }^2 & \leq \left( \Delta x \sum_{i\in I} |F_{i+\frac{1}{2}}^{\ee, n+1}| \right)^2 \leq (\Delta x) |I|  (\Delta x) \sum_{i\in I} |F_{i+\frac{1}{2}}^{\ee, n+1}|^2 \\
        & =  \Delta x \sum_{i\in I} (\Theta_{i+\frac{1}{2}}^{\ee, n+1})^2 |v_{i+\frac{1}{2}}^{\ee, n+1}|^2.
    \end{align*}
    Hence, using \eqref{eq:discrete entropy dissipation} and the fact that
    $
        | \Theta_{i + \frac 1 2}^{\ee , n} | \leq \mobeend(\alpha)   \mobeeni(0)
    $,
    we recover
    \begin{equation*}
        \| \rho^{\ee , n+1} - \rho^{\ee ,n} \|_{W^{-1,1}_\Delta(0,L)} \leq  C (\Delta t )^{\frac{1}{2}} \left( \Delta t \Delta x \sum_{i\in I} \Theta_{i+\frac{1}{2}}^{\ee, n+1} |v_{i+\frac{1}{2}}^{\ee, n+1}|^2 \right)^{\frac{1}{2}} \leq C (\Delta t )^{\frac{1}{2}} (\Ediscee[\rho^{\ee, n}] - \Ediscee[\rho^{\ee , n+1}])^{\frac 1 2}.
    \end{equation*} 
    We recall that $\Ediscee[\rho^{\ee, n}]$ is a non-increasing sequence due to \Cref{lem:disc energy dissipation}. Furthermore, it is also bounded from below by $\min_{s \in [0, \alpha]} U_\ee (s) > - \infty$. Hence, there exists $\Edisceeinfty  \in \R$ such that as $n \rightarrow \infty$
    \begin{equation*}
        \Ediscee[\rho^{\ee , n}] \searrow \Edisceeinfty. 
    \end{equation*}
    Hence, it follows that
    \begin{equation*}
        \lim_{k \rightarrow \infty} \| \rho^{\ee , n_k} - \rho^{\ee, n_k +1} \|_{W^{-1,1} _{\Delta} (0,L)} \le  C (\Delta t)^{\frac{1}{2}} \lim_{k \rightarrow \infty} (\Ediscee[\rho^{\ee, n_k}] - \Ediscee[\rho^{\ee, n_k+1}])^{\frac 1 2} = 0.
    \end{equation*}
    Therefore, we obtain that
   $\rho^{\ee, n_k +1} \to u$ in $\R^{|I|}$.

    \step[Stationary sub and super solution]
    Since $\rho^0 \in \mathcal A_{\Delta , +}$ we can pick $C_1, C_2 \in \mathbb R$ such that
    \[ 
        (U_\ee')^{-1} (C_1 - V(x_i)) \le \rho^{0}_i \le (U_\ee')^{-1} (C_2 - V(x_i)), \qquad \forall i \in I.
    \]
    These upper and lower bounds are stationary solutions, and we have a comparison principle, they preserved for all time.
    Thus, we have that $\rho_i^{\ee,n} \in [\delta , \alpha - \delta]$ for some $\delta > 0$ and all $n, i$.

    \step[Solution of the scheme \eqref{eq:problem regularised discrete}]
    Due the convergence and the continuity of the mobility $\mobee^{(j)}$ (with $j=1, 2$) and $U_\ee'$ in $[\delta, \alpha - \delta]$, it follows that
    \begin{equation*}
        \mobee^{(j)} (\rho^{\ee, n_k+1}) \rightarrow \mobee^{(j)} (\thelimit) , \quad
        U_\ee'(\rho^{\ee, n_k+1}) \rightarrow U_\ee' (\thelimit) \quad \text{in } \R^{|I|}.
    \end{equation*}
    Since $H^\ee(\rho^{\ee, n_k+1}) = \rho^{\ee , n_k}$, we have that $\thelimit$ satisfies $H^\ee(u) = u$. 
    Since $0 < u_i < \alpha$ for all $i\in I$ we conclude that for all $i\in I$ we have $u_i = (U_\ee')^{-1} (C - V(x_i))$ for some $C$.
    By conservation of mass we precisely characterise $u = \rho^{\ee, \infty}$.

    \step[The whole sequence converges]
    Since every sequence has a convergent subsequence, and they all share the same limit.
    \qedhere
    \end{enumeratesteps}
\end{proof}

We connect the discrete and continuous stationary states, and we provide a rate of convergence on $\Delta x$. We start with the case $\ee > 0$ fixed.

\begin{proof}[Proof of \ref{thm:Asymptotic convergence discrete to continuous ee}]
    Let us remark that,
    \begin{equation}\label{eq:Discrete and continuous mass}
        \Delta x \sum_i  \left( U_{\ee}' \right)^{-1} \left( C_{\ee}^{\Delta } - V(x_i) \right) = M = \int_0^1 \left( U_{\ee}' \right)^{-1} \left(C_\ee - V(x) \right) \dx.
    \end{equation}
    Since we know that $0 < \widehat\rho^{(\ee)} < \alpha$, in particular 
    $$
        \left| \left(\left( U_{\ee}' \right)^{-1} \right)' \left( C_{\ee} - V(x_i) \right) \right| \leq C \quad \forall \, i\in I.
    $$
    Since $\widehat \rho^{(\ee)} \in \mathcal A_+$, we can perform a first order Taylor expansion to obtain that for $h \in [0, \Delta x )$,
    \begin{equation*}
        \widehat\rho^{(\ee )} (x_i + h) = \left( U_{\ee}' \right)^{-1} \left( C_{\ee} - V(x_i) \right) + O(h).
    \end{equation*}
    Then, in particular,
    \begin{equation}\label{eq:Taylor expansion for the mass}
        \int_0^1 \left( U_{\ee}' \right)^{-1} \left(C_\ee - V(x) \right) \dx = \Delta x \sum_i \left( \left( U_{\ee}' \right)^{-1} \left( C_{\ee} - V(x_i) \right) + O(\Delta x ) \right).
    \end{equation}
    Combining \eqref{eq:Discrete and continuous mass} and \eqref{eq:Taylor expansion for the mass} it follows that,
    \begin{align*}
        O ( \Delta x ) & = \Delta x \sum_i \left( \left( U_{\ee}' \right)^{-1} \left( C_{\ee}^{\Delta } - V(x_i) \right) - \left( U_{\ee}' \right)^{-1} \left( C_{\ee} - V(x_i) \right) \right) \\
        & = \Delta x \sum_i \left( \left( U_{\ee}' \right)^{-1} \right)' (\zeta_i) (C_\ee^{\Delta } - C_{\ee}) ,
    \end{align*}
    with $\zeta_i \in ( C_{\ee}^{\Delta } - V(x_i) , C_{\ee} - V(x_i))$. It satisfies, 
    \begin{equation*}
        0 < \underline C \leq \left( \left( U_{\ee}' \right)^{-1} \right)' (\zeta_i) \leq \overline C < \infty, \quad \forall \, i \in I.
    \end{equation*}
    From here, it follows,
    \begin{equation*}
        \left| C_\ee^{\Delta } - C_\ee \right| = \left| \frac{O( \Delta x )}{\Delta x \sum_i \left( \left( U_{\ee}' \right)^{-1} \right)' (\zeta_i)} \right| \leq \left| \frac{O( \Delta x )}{\Delta x \sum_i \underline C} \right| = O( \Delta x ) . \qedhere
    \end{equation*}
\end{proof}

Once we have understood the long-time behaviour of the regularised numerical method \eqref{eq:problem regularised discrete}, we now proceed to study \eqref{def:Numerical method} and the properties of its long-time behaviour.

\begin{proof}[Proof of \Cref{thm:Asymptotics discrete P0}]
 
    \ref{it:discrete P0 exists asymptotic limit} follows analogously to \ref{it:discrete approx global attractor}. 
    Since $\numericStep$ is an $L^1_{\Delta }$ contraction, so is $(J^{\Delta})^n$, and hence $J^{\Delta,\infty}$.
    \ref{it:discrete P0 asymptotic limit is L1 contraction} follows as $n\to \infty$.

    \textit{\ref{it:discrete P0 discrete global minimiser}. Constant-in-time solution to \eqref{def:Numerical method}.} 
    The first part of the proof is analogous to the one in the continuous case done for \Cref{thm:Asymptotic (P0)}. In order to finish the proof, we show that $\rho^{0, \infty}_i = \inversedU (C_0^{\Delta } - V(x_{i}))$ is a stationary state at every point $x_i$ for $i \in I$. The argument that we consider is the following: $\rho^{0, \infty}$ is a stationary state if and only if $F_{i\pm\frac 1 2} = 0$ for all $i \in I$.

    Let us assume $\rho^{0, \infty}_i = 0$ and $0 < \rho^{0, \infty}_{i+1} < \alpha$. 
    This means that $C_0^{\Delta } - V(x_i) \le \underline{\zeta} = U'(0^+) < C_0^{\Delta } - V(x_{i+1}) < \overline{\zeta} = U'(\alpha^-)$.
    Hence, we can compute that
    \begin{align*}
        v_{i+\frac{1}{2}} (\rho^{0, \infty}) & = - \frac{U' (\rho^{0, \infty}_{i+1}) +V(x_{i+1}) - ( U' (\rho^{0, \infty}_{i}) +V(x_{i}))  }{ \Delta x} = \frac{ U' (0) +V(x_{i}) - C_0^{\Delta } }{\Delta x} \geq 0,
    \end{align*}
    where, in the last inequality, we are using that $U'$ is increasing. 
    Since $\mobeend (0) = 0$, we get that
    \begin{equation*}
        F_{i+ \frac{1}{2}}( \rho^{0, \infty} ) = \mobeend (\rho^{0,\infty}_i) \mobeeni (\rho^{0,\infty}_{i+1}) (v_{i+ \frac{1}{2}} (\rho^{0, \infty}))^+ =  0 ,
    \end{equation*}
    and $\rho^{0, \infty}$ is a stationary state at the point $x_i$. Arguing analogously we can prove the same result for all the different combinations  
    of $\rho_i^{0,\infty}$ and $\rho_{i+1}^{0,\infty}$ taking values $0$, $\alpha$, or in $(0,\alpha)$.

    \textit{\ref{it:discrete asymptotic state limit as ee to 0}. Convergence of numerical steady states as $\ee \to 0$.}
    We argue analogously to \Cref{lem:Convergence Asymptotics of (Pee)} to show that there exists $\ee_k$ such that $C_{\ee_k}^\Delta \rightarrow \widetilde{C}$ as $k \rightarrow \infty$ and
    \begin{equation}
    \label{eq:discrete asymptotic steady-state convergence point-wise}
        (U_{\ee_k}')^{-1} (C_{\ee_k}^{\Delta } - V(x) ) \rightarrow \inversedU (\widetilde{C} -V(x)) \quad \text{pointwise in } [0,\alpha].
    \end{equation}
    Due to the mass condition $\widetilde C = C_0^{\Delta }$.
    As usual, we realise that every sequence $\ee_k$ has a subsequence where \eqref{eq:discrete asymptotic steady-state convergence point-wise} holds, and the limit is shared amongst convergent sequences. This proves the convergence as $\ee \rightarrow 0$.

    \textit{\ref{cor:Convergence disc to cont stat state ee=0}. Convergence of the numerical solution steady state as $\Delta \to 0$.}
    Convergence without rates follows as in the previous step. 
    Let us now prove rates of convergence.
    Since $(U')^{-1}$ is Hölder, then $\inversedU$ is also Hölder. 
    Let us also assume we are in the case $0 < M < \alpha$. This second step of the proof works analogously to the proof of \ref{thm:Asymptotic convergence discrete to continuous ee}. We just need to adapt the Taylor expansion to its Hölder version. If we do that we arrive at 
    \begin{align}\label{eq:Disc to continuous ee=0 order gamma}
    \begin{split}
        O((\Delta x)^\gamma) & = \int_0^1 ( \inversedU (C_0^{\Delta } - V(x)) - \inversedU(C_0 - V(x)) ) \\\underline{\zeta}
    & = (C_0^{\Delta } - C_0) \int_0^1 (\inversedU)'(\widehat \zeta (x)) dx ,
    \end{split}
    \end{align}
    where $\widehat \zeta (x)$ lies between $C_0^{\Delta } - V(x) , C_0 - V(x)$. Since $0 < M < \alpha$, there exists $z \in (0,1)$ such that $U'(0^+) \eqqcolon \underline{\zeta} < C_0 - V(z) < \overline{\zeta} \coloneqq U'(\alpha^-)$. Let us define
    $$
        \ell := \min \left\lbrace (C_0 - V(z))- \underline{\zeta} , \overline{\zeta} - (C_0 -V(z)) \right\rbrace.
    $$
    From the previous step, we know that there exists $\delta > 0$ small enough such that for every $\Delta x < \delta$ we have that 
    $|C_0^{\Delta } - C_0 | \leq \frac{\ell}{4}$. From the continuity of $V$, there exists an interval $K \subseteq (0,1)$ small enough such that for every $x \in K$ we have that $|V(z)-V(x)| \leq \frac{\ell}{4}$. Therefore, it follows that
    \begin{equation*}
        |\widehat \zeta (x)-\underline{\zeta} |, | \overline{\zeta} - \widehat \zeta (x) | \geq \frac{\ell}{2} > 0 \quad \text{for all } x \in K.
    \end{equation*}
    In particular, $(U')^{-1} \left( \left\lbrace \widehat \zeta (x) : x \in K \right\rbrace\right) \subseteq [s_1,s_2] $, a compact subset of $(0, \alpha)$. Then, we obtain that 
    \begin{align*}
        \int_0^1 (\inversedU)'(\widehat \zeta (x)) dx & \geq \int_K ((U')^{-1})' (\widehat \zeta (x)) dx = \int_K \frac{1}{U''((U')^{-1} (\widehat \zeta (x)))} dx \\
        & \geq \frac{|K|}{\max_{s \in [s_1,s_2]} U''(s)} > 0.
    \end{align*}
    Thus, combining this bound with \eqref{eq:Disc to continuous ee=0 order gamma} we have that $
        |C_0^{\Delta } - C_0 | = O((\Delta x)^\gamma).$ \qedhere
\end{proof}

\subsection{Time asymptotics for \texorpdfstring{$\ee=0$}{eps 0}. The minimiser might not be an attractor}\label{sec:Num counterexample}

We can reproduce the example at \cref{sec:malicious counterexamples} to prove that the back face of the diagram \eqref{Numeric Diagram} is not commutative for the numerical scheme \eqref{eq:problem regularised discrete} either. Analogously to the continuous case we construct a counterexample using $U(s) = \frac{1}{m-1}s^m$ and $V$ a double-well potential.

For $m> 1$ we consider the discrete Barenblatt
\begin{equation*}
    \mathcal{B}_i = \left( \frac{m-1}{m} \left( C_0^\Delta - \frac{|x_i|^2}{2} \right) \right)^{\frac{1}{m-1}}_+. 
\end{equation*} 
We select $C_0^\Delta$ such that $\Delta x \sum_{i=1}^M \mathcal{B}_i > 1$. Working analogously to the proof of \Cref{thm:Asymptotics discrete P0} we can show that $\mathcal{B}_i$ is a discrete stationary state. Let us remind the reader that the method \eqref{eq:problem regularised discrete} applies only to dimension $1$. Since $m>1$, $\supp \mathcal B \subset \{ j : | x_j | < \widehat R \} $. Take $L, K \in I$ with $|x_L-x_K| > 2 \widehat{R}$, and consider
$
    \overline{\rho}_i = \mathcal{B}_{i-L} + \mathcal{B}_{i-K} .
$
We choose again the potential, 
$$
    V(x) = \begin{cases}
        \frac{|x - x_L|^2}{2} & \text{if } |x - x_L|  \le {\widehat R} \\
        \frac{|x - x_K|^2}{2} & \text{if } |x - x_K|  \le {\widehat R} \\
        \frac{|x|^2}{2} & \text{if } |x| \gg 1 \\
        \text{smooth} & \text{in the intermediate regions}
    \end{cases}
$$
For each $0 \leq \rho_i^0 \leq \overline{\rho}_i$ we solve \eqref{eq:problem regularised discrete} with linear mobility $\mob(s) = s$. We are solving \eqref{eq:problem regularised discrete} in each interval $(x_j - \widehat{R}, x_j + \widehat{R})$, $j = L,K$. Due to the Comparison Principle from \Cref{thm:scheme well-posedness}, we get $0 \leq \rho_i^n \leq \overline{\rho}_i$. Consider
\begin{equation*}
    \widehat{\rho}_i = \left( \frac{m-1}{m} \left( C_L^\Delta - \frac{|x_i-x_L|^2}{2} \right) \right)^{\frac{1}{m-1}}_+ + \left( \frac{m-1}{m} \left( C_K^\Delta - \frac{|x_i-x_K|^2}{2} \right) \right)^{\frac{1}{m-1}}_+,
\end{equation*}
and select $C_l^\Delta, C_k^\Delta \leq C_0^\Delta$ such that
\begin{equation*}
    \Delta x \sum_{|x_i-x_K| \leq \widehat{R} } \widehat{\rho}_i = \Delta x \sum_{|x_i-x_K| \leq \widehat{R} } \rho_i^0 , 
    \qquad 
    \Delta x \sum_{|x_i-x_L| \leq \widehat{R} } \widehat{\rho}_i = \Delta x \sum_{|x_i-x_L| \leq \widehat{R} } \rho_i^0 .
\end{equation*}
In particular, $\widehat{\rho}_i$ is such that, $0 \leq \widehat{\rho}_i \leq \overline{\rho}_i$. 

Due to compactness, the scheme $\rho_i^n$ is such that, up to a subsequence,
$
    \rho_i^n \rightarrow \xi_i \quad \text{in } \R^{|I|}.
$
From the Comparison Principle, the only candidate for $\xi_i$ is $\widehat{\rho}_i$ and therefore, up to a subsequence,
$
    \rho_i^n \rightarrow \widehat{\rho}_i \quad \text{in } \R^{|I|}.
$
When 
$$
    \Delta x \sum_{|x_i-x_L| \leq \widehat{R} } \rho_i^0 \neq \Delta x \sum_{|x_i-x_K| \leq \widehat{R} } \rho_i^0  ,
$$
then $C_L^\Delta \neq C_K^\Delta$.

\begin{remark}
    Following \cref{sec:Malicious example saturation} we can extend this case to further examples with saturation.
\end{remark}

\subsection{Numerical experiments}\label{sec:Numerical experiments}

We implement the scheme \eqref{def:Numerical method} using the \texttt{julia} language \cite{Bezanson_Edelman_Karpinski_Shah17}. The fixed point problem for the implicit time-stepping is solved using a Newton method through the \texttt{NLSolve.jl} package (see \cite{NLSolver}) using automatic differentiation via the \texttt{ForwardDiff.jl} package.

\paragraph{Convex potential}

We exemplify the behaviour for the nonlinear diffusion $\rho^2$ with a convex potential and free boundaries at levels $\rho = 0,\alpha$ in \Cref{fig:numerical experiment convex potential}. 

In the left plot we show the profiles at different times illustrating the formation of the free boundary and the kinks in the upper constraint. In the right plot we display the speed of convergence to the steady state. Numerical experiments indicate that it is exponential as suggested by the linear character of the log plot of the error towards the exact steady solution. For $t$ large the error becomes so small that its $\log$ is computationally $-\inf$.
\begin{figure}[ht!]
    \centering 
\includegraphics[width=.9\textwidth]{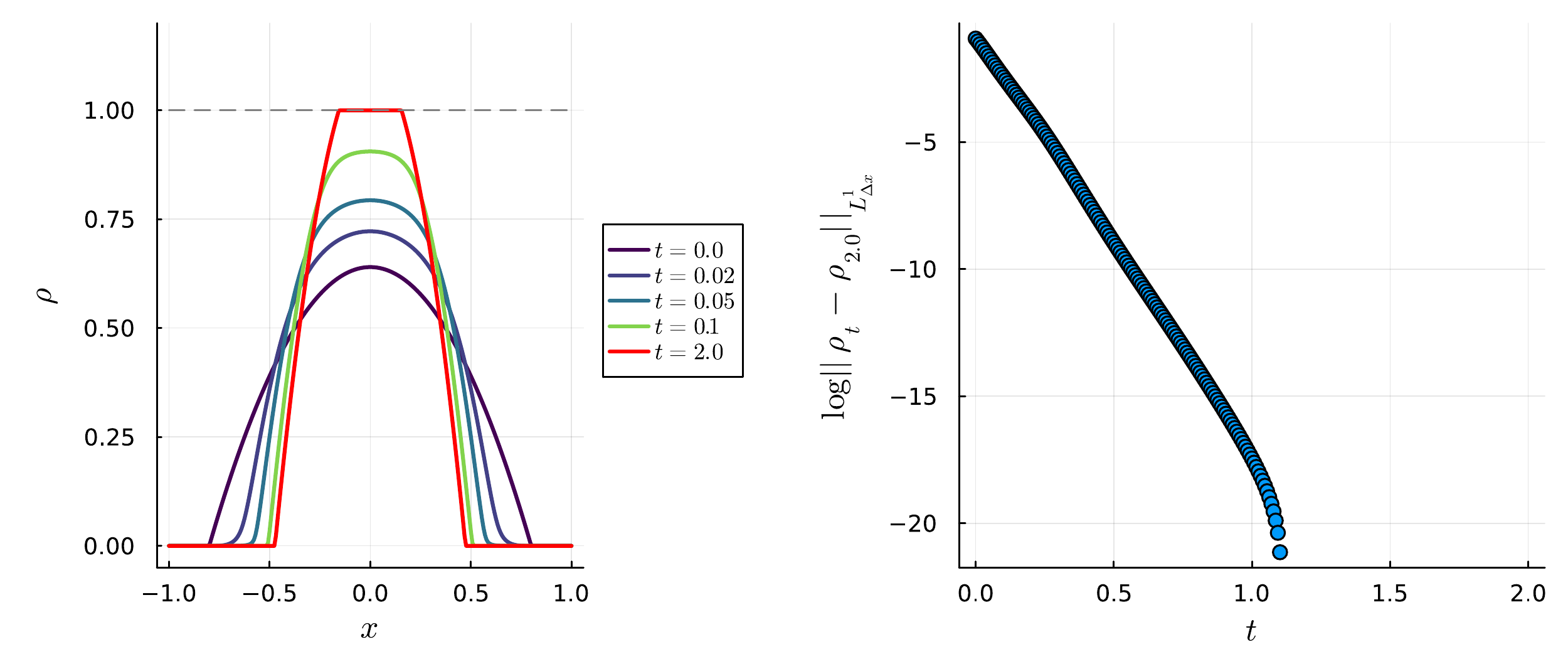}
    \caption{
    $\mob(\rho) = \rho(1-\rho)$, $U(\rho) = \rho^2$ and $V(x) = 10x^2$.
    $\Delta t = \Delta x = 2^{-7}$.
    Left: profiles at different times.
    Right: distance from $\rho_t$ to $\rho_2$.
    }
        \label{fig:numerical experiment convex potential}
\end{figure}

\paragraph{Double well potential. Formation of Barenblatt from above}

We exemplify the behaviour for a double well potential where a Barenblatt profile appears from level $\rho = \alpha$ in \Cref{fig:numerical experiment Barenblatt from above}. This shows the formation of a gap in the upper free boundary of the problem.

\begin{figure}[ht!]
    \centering 
    \includegraphics[width=.58\textwidth]{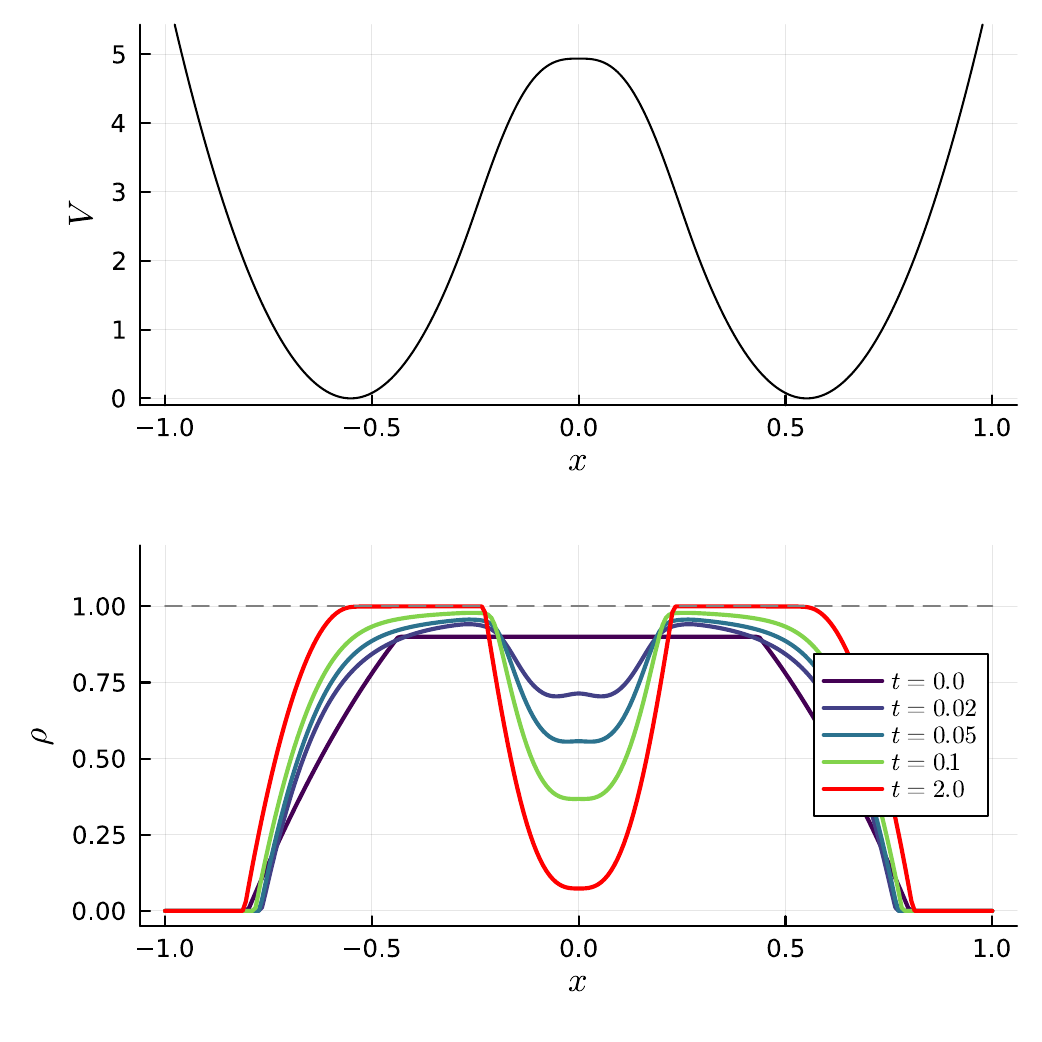}
    \caption{$\mob(\rho) = \rho(1-\rho)$, $U(\rho) = \rho^2$. $\Delta t = \Delta x = 2^{-7}$}
    \label{fig:numerical experiment Barenblatt from above}
\end{figure}

\paragraph{Non-minimising steady states}

Finally, we show in \Cref{fig:numerical experiment non-minimising} a numerical experiment in which the asymptotic state is a non-minimising linear combination of two Barenblatt profiles with disjoint supports. This illustrates our theoretical findings in Section \ref{sec:Local minimiser} and \Cref{thm:Euler-Lagrange}, observing that the global minimizer of the free energy does not always attract all initial data and that these non-minimising steady states have a large basin of attraction.
\begin{figure}[ht!]
    \centering 
    \includegraphics[width=.58\textwidth]{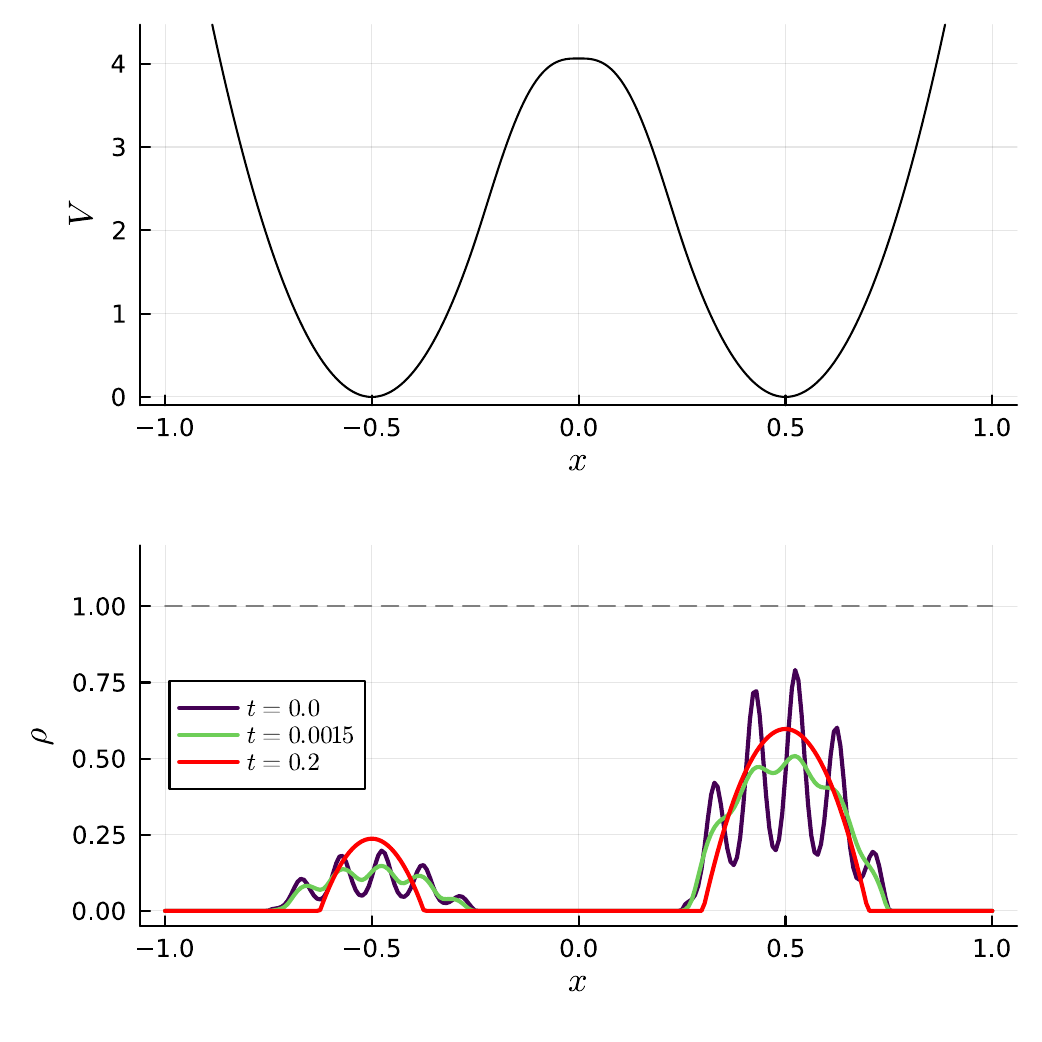}
    \caption{$\mob(\rho) = \rho(1-\rho)$, $U(\rho) = \rho^2$. $\Delta t = 2^{-12}, \Delta x = 2^{-7}$.}
    \label{fig:numerical experiment non-minimising}
\end{figure}

\subsection*{Acknowledgements}
JAC and AFJ were supported by the Advanced Grant Nonlocal-CPD (Nonlocal PDEs for Complex Particle Dynamics: Phase Transitions, Patterns and Synchronization) of the European Research Council Executive Agency (ERC) under the European Union’s Horizon 2020 research and innovation programme (grant agreement No. 883363).
JAC was also partially supported by the EPSRC grant number  EP/V051121/1.
DGC was supported by RYC2022-037317-I
and partially supported by PID2023-151120NA-I00 from the Spanish Government MCIN/AEI/10.13039/501100011033/
FEDER, UE.
The authors are thankful to Yao Yao (National University Singapore) for useful suggestions on an early stage of the manuscript. 
We would like to thank the reviewers for the careful reading and useful suggestions.

\printbibliography

\end{document}